\documentclass[10pt,twoside,a4paper,english,leqno]{article}

\title{Rectification of a deep water model for surface gravity waves}

\author{Vincent Duch\^ene\thanks{Univ Rennes, CNRS, IRMAR - UMR 6625
F-35000 Rennes, France. E-mail: \url{vincent.duchene@univ-rennes1.fr}}
\and
Benjamin Melinand\thanks{ CEREMADE, CNRS, Universit\'e Paris-Dauphine, Universit\'e PSL, 75016 Paris, France; Email \url{melinand@ceremade.dauphine.fr}}
}
\date{\today}

\usepackage{babel}         
\usepackage[utf8]{inputenc}  
\usepackage[T1]{fontenc}    
\usepackage[ec]{aeguill}    
\usepackage{amsmath, amsthm, amsfonts,amssymb, mathtools, bm, stmaryrd}
\usepackage{float}     

\usepackage[pdftex]{graphicx}     
\usepackage{fancyhdr}     
\usepackage{url}

\usepackage{hyperref}     
\usepackage{verbatim}      
\usepackage{cleveref}

\usepackage{xcolor}

\usepackage{subcaption}
\usepackage[justification=centering]{caption}

\numberwithin{equation}{section}

\makeatletter
\let\Title\@title
\let\Author\@author
\makeatother
 
\fancypagestyle{mystyle}{%
	\fancyfoot{}
	\fancyhead[CO]{\Title} 
	\fancyhead[CE]{\scriptsize V.~Duchêne \& B. Melinand} 
	\fancyhead[RE]{\scriptsize\today}  
	\fancyhead[LO]{}  
	\fancyhead[LE,RO]{\scriptsize\thepage}
	}
\fancypagestyle{plain}{\fancyhead{} } 

\newcommand{\RR}{\mathbb{R}}

\newcommand{\ZZ}{\mathbb{Z}}
\newcommand{\NN}{\mathbb{N}}

\newcommand{\TT}{\mathbb{T}}

\newcommand{\Hdot}{\mathring{H}}
\newcommand{\dotH}{\mathring{H}}

\newcommand{\psia}{\psi_{(\alpha)}}
\newcommand{\psib}{\psi_{(\beta)}}
\newcommand{\zetaa}{\zeta_{(\alpha)}}

\newcommand{\fPi}{{\sf \Pi}}

\DeclareMathOperator{\dd}{d\!}

\DeclareMathOperator*\esssup{ess\,sup}

\renewcommand{\i}{{\rm i}}
\DeclareMathOperator{\Id}{Id}

\renewcommand{\r}{m}
\renewcommand{\a}{\ell}

\newcommand{\br}{{\bm{r}}}

\newcommand{\bx}{{\bm{x}}}
\newcommand{\bxi}{{\bm{\xi}}}
\newcommand{\bta}{{\bm{\eta}}}
\newcommand{\bk}{{\bm{k}}}
\newcommand{\bz}{{\bm{0}}}
\newcommand{\eqdef}{\stackrel{\rm def}{=}}

\newcommand{\cC}{\mathcal{C}}
\newcommand{\cE}{\mathcal{E}}
\newcommand{\cH}{\mathcal{H}}

\newcommand{\cO}{\mathcal{O}}
\newcommand{\cX}{\mathcal{X}}

\newcommand{\N}[2]{\norm{\langle\cdot\rangle^{#1} #2}_{L^\infty}} 
\newcommand{\dotN}[2]{\norm{\lvert\cdot\rvert^{#1} #2}_{L^\infty}}
\newcommand{\No}[1]{\norm{#1}_{L^\infty}} 
\newcommand{\Ni}[1]{\max\big(\big\{\norm{#1}_{L^\infty},\norm{\langle\cdot\rangle\nabla #1}_{L^\infty}\big\}\big)} 
\newcommand{\dotNi}[1]{\max\big(\big\{\norm{#1}_{L^\infty},\norm{\lvert\cdot\rvert\nabla #1}_{L^\infty}\big\}\big)} 

\renewcommand{\L}{{\sf L}}
\newcommand{\J}{{\sf J}}
\newcommand{\T}{{\bf{\sf T}}}

\newcommand{\mfP}{(G^\mu_0)^{1/2}} 
\newcommand{\mfa}{\mathfrak{a}^\mu}

\newcommand{\ie}{{\em i.e.}~}
\newcommand{\eg}{{\em e.g.}~}
\newcommand{\id}[1]{\left\vert_{_{#1}}\right.}

\DeclarePairedDelimiter\norm{\big\lvert}{\big\rvert}
\DeclarePairedDelimiter\Norm{\big\lVert}{\big\rVert}

\newtheorem{Theorem}{Theorem}[section]
\newtheorem{Definition}[Theorem]{Definition}
\newtheorem{Proposition}[Theorem]{Proposition}
\newtheorem{Corollary}[Theorem]{Corollary}
\newtheorem{Lemma}[Theorem]{Lemma}
\newtheorem{Remark}[Theorem]{Remark}

\Crefname{Theorem}{Theorem}{Theorems}
\Crefname{Proposition}{Proposition}{Propositions}
\Crefname{Lemma}{Lemma}{Lemmas}

\begin{document}

\maketitle
%\tableofcontents

\begin{abstract}
In this work we discuss an approximate model for the propagation of deep irrotational water waves, specifically the model obtained by keeping only quadratic nonlinearities in the water waves system under the Zakharov/Craig--Sulem formulation. We argue that the initial-value problem associated with this system is most likely ill-posed in finite regularity spaces, and that it explains the observation of spurious amplification of high-wavenumber modes in numerical simulations that were reported in the literature. 
This hypothesis has already been proposed by Ambrose, Bona, and Nicholls~\cite{bona_ill} but we identify a different instability mechanism. On the basis of this analysis, 
we show that the system can be ``rectified''. Indeed, by introducing appropriate regularizing operators, we can restore the well-posedness without sacrificing other desirable features such as a canonical Hamiltonian structure, cubic accuracy as an asymptotic model, and efficient numerical integration. 
This provides a first rigorous justification for the common practice of applying filters in high-order spectral methods for the numerical approximation of surface gravity waves. While our study is restricted to a quadratic model, we believe it can be generalized to any order and paves the way towards the rigorous justification of a
robust and efficient strategy to approximate water waves with arbitrary accuracy. 
Our study is supported by detailed and reproducible numerical simulations.
\end{abstract}

\section{Introduction}

\subsection{Motivation }

It is well known \cite{Zakharov,Craig_Sulem_1} that the propagation of surface gravity waves ---that is the motion of inviscid, incompressible, homogeneous and potential flows with a free surface under the influence of gravity--- can be described through two scalar evolution equations for unknowns describing the surface deformation and the trace of the velocity potential at the surface; see~\eqref{WW} in Appendix~A. 
These equations involve the so-called Dirichlet-to-Neumann operator, which consists of solving the Laplace problem for the velocity potential in the fluid domain with Dirichlet condition at the surface and Neumann condition at the bottom, and returning the (rescaled) normal component of the velocity at the surface. The Dirichlet-to-Neumann operator turns out to be shape-analytic (see ref.~\cite{Lannes_ww} and references therein), so that it enjoys a converging expansion with respect to sufficiently small and regular surface deformation variables. Since each term of the expansion can be computed efficiently by Fourier pseudo-spectral methods, this provides an attractive strategy for the numerical integration of the equations, which roughly speaking consists in replacing the Dirichlet-to-Neumann operator by the expansion truncated at a sufficiently large order to obtain the desired accuracy, and then using Fourier spectral methods for the discretization in space and a high-order time integrator (typically fourth-order Runge-Kutta). This strategy was proposed independently by several authors in~\cite{DommermuthYue87,WestBruecknerJandaEtAl87,Craig_Sulem_2} and successfully applied in various situations (see \eg \cite{LeTouze03,Schaeffer08,WilkeningVasan15,Nicholls16} for a detailed account and comparisons).

However, it was noted for example in \cite{DommermuthYue87,numeric_Guyenne_moving_bott} that, in situations where a large number of Fourier modes are numerically computed, the highest wavenumber modes suffer from a rapid amplification
of computational errors. These high-frequency instabilities are typically presented as numerical artifacts and are often dealt with by applying low-pass filters, or wavenumber-dependent damping terms such as parabolic regularization. In~\cite{bona_ill}, Ambrose, Bona, and Nicholls suggest that these instabilities are not spurious results of the numerical discretization but rather take roots in instabilities at the level of the continuous evolution equations. They support their hypothesis through tailored numerical simulations and the analysis of toy models associated with the quadratic-order and cubic-order Dirichlet-to-Neumann expansion.

In this work, we support the general statement of Ambrose, Bona, and Nicholls while diagnosing a different instability mechanism. The difference between the two proposed instability mechanisms is easily seen when comparing the toy model introduced in~\cite[(2.3)]{bona_ill} and our toy model~\eqref{RWW2-toy}, introduced and studied in \Cref{S.toy}. They however share common aspects: in particular the proposed instability mechanism is nonlinear and triggered by low-high frequency interactions, in the sense that it is the presence of a substantial low-frequency component of the flow which causes the rapid amplification of the high-frequency component. The main difference between the two proposed instability mechanisms is that ours emanates from cubic contributions, while the one of in~\cite[(2.3)]{bona_ill} is of quadratic nature. 
 {\em The cubic nature of our diagnosed instability mechanism concerning a system of equations which contains only quadratic nonlinearities is one of our key observations and is essential for the ``rectification'' procedure that we expound below. 
 By the terminology ``rectification'' we mean that it is possible to consider a slightly modified system that does not suffer from unwanted instabilities while maintaining the accuracy of the original system as a model for water waves. }

 \subsection{The equations and their rectification}
 
 Let us now introduce the systems of equations we consider in this work. Starting from the Zakharov/Craig--Sulem formulation of the water waves system~\eqref{WW} and discarding all cubic and higher order contributions stemming from the Dirichlet-to-Neumann expansion, we obtain, after suitable rescaling and borrowing notations from \cite[\S8]{Lannes_ww}, the following system
 \begin{equation}\label{WW2-intro}\tag{WW2}\left\{\begin{array}{l}
 		\partial_t\zeta- G^\mu_0\psi+ \epsilon G^\mu_0 (\zeta G^\mu_0 \psi)+\epsilon\nabla\cdot( \zeta\nabla\psi)=0,\\[1ex]
 		\partial_t\psi+\zeta+\frac\epsilon{2 }  \left( |\nabla\psi|^2- (G^\mu_0 \psi )^2 \right)=0.
 	\end{array}\right.\end{equation} 
Here, $\zeta$ (resp. $\psi$) represents the surface deformation (resp. the trace of the velocity potential at the surface) at time $t\in\RR$ and horizontal location $\bx\in\RR^d$ (with $d\in\{1,2\}$),  and $G^\mu_0=|D|\tanh(\sqrt\mu|D|)$ is the Dirichlet-to-Neumann operator linearized about the rest state, \ie the Fourier multiplier defined by (see forthcoming \Cref{S.main} for the definition of functional spaces)
\[ \forall f\in \Hdot^{1}(\RR^d) , \qquad \widehat{G^\mu_0f}(\bxi) = |\bxi|\tanh(\sqrt\mu|\bxi|) \widehat f(\bxi) \in L^2(\RR^d).\]
The dimensionless parameter $\mu$, defined as the square of the ratio of the depth of the layer at rest to a characteristic horizontal length, is the shallowness parameter. It can take arbitrarily large values greater than one in this work. The dimensionless parameter $\epsilon$, defined as the ratio of the amplitude of the wave to the characteristic horizontal length, represents the steepness of the waves. It is typically small. Indeed, the quadratic system~\eqref{WW2-intro} can be justified as an approximation to the fully nonlinear water waves system~\eqref{WW} with precision $\cO(\epsilon^2)$; see details in \Cref{S.WW}.

One asset of the quadratic system~\eqref{WW2-intro} with respect to other models for the propagation of deep water waves 
(see \eg \cite{Stokes47,BenneyLuke64,matsuno_flat,Choi_2D,smith,AkersMilewski08,bonneton_lannes,Alvarez_Lannes,saut_xu,shkoller_etal2019}) 
is that it retains the canonical Hamiltonian structure of the fully nonlinear equations. Indeed~\eqref{WW2-intro} reads
 \begin{equation}\label{eq.Hamilton}
\left\{\begin{array}{l}
\partial_t\zeta-\delta_\psi \cH^\mu=0,\\[1ex]
\partial_t\psi+\delta_\zeta \cH^\mu=0,
\end{array}\right.
\end{equation}
with the functional
\begin{equation}\label{eq.Hamiltonian}
\cH^\mu(\zeta,\psi)\eqdef\frac12\int_{\RR^d} \zeta^2+ \psi G^\mu_0 \psi + \epsilon \zeta \left( |\nabla\psi|^2-( G^\mu_0  \psi)^2\right)\dd \bx.
\end{equation}

As already mentioned, we strongly believe that~\eqref{WW2-intro} suffers from strong high-frequency instabilities which in particular prevent the well-posedness of the initial-value problem in spaces of finite regularity. We propose to consider the following modified system
 \begin{equation}\label{RWW2-intro}\tag{RWW2}\left\{\begin{array}{l}
\partial_t\zeta- G^\mu_0 \psi+ \epsilon G^\mu_0 ((\J^\delta \zeta) G^\mu_0 \psi)+\epsilon\nabla\cdot((\J^\delta \zeta)\nabla\psi)=0,\\[1ex]
\partial_t\psi+\zeta+\frac\epsilon{2 } \J^\delta \left( |\nabla\psi|^2- (G^\mu_0 \psi )^2 \right)=0,
\end{array}\right.\end{equation}
where we have introduced the Fourier multiplier $\J^\delta=J(\delta D)$ which can be freely chosen in a space of ``admissible rectifiers'' and where $\delta>0$ is a scaling parameter measuring the ``strength'' of the rectifier and can be adjusted according to needs (see below for a detailed discussion).
\begin{Definition}[ Admissible rectifiers ]\label{D.J-intro}
	Let $\J^\delta=J(\delta D)$ with $J\in L^{\infty}(\RR^d)$,  real-valued and even. 
	\begin{itemize}
		\item We say that $\J^\delta$ is {\em regularizing} of order $\r\leq0$ if $\langle\cdot\rangle^{-\r} J \in L^\infty(\RR^d)$, with $\langle\cdot\rangle\eqdef(1+|\cdot|^2)^{1/2}$.
		\item We say that $\J^\delta$ is  {\em regular} if $\J^\delta$ is regularizing of order $-1$ and, additionally,  $\langle\cdot\rangle \nabla J \in L^\infty(\RR^d)$.
		\item We say that $\J^\delta$ is {\em near-identity} of order $\a\geq0$ if $|\cdot|^{-\a}(1-J) \in L^\infty(\RR^d)$.
	\end{itemize}
	{\em Admissible rectifiers} $\J^\delta$ are regular and near-identity of order $\a>0$.
\end{Definition}
We do not see any physical motivation that would dictate a specific choice for the symbol profile, $J$. In fact, the strength $\delta$ is the main object of interest. Examples of admissible rectifiers, which we use in our numerical simulations, are
\begin{equation}\label{Rectifier_example}
\J^\delta = \min(\{1,|\delta D|^{\r}\}),\qquad  \r\leq -1.
\end{equation}
They are regularizing of order $\r$ and near-identity of order $\a$ for any $\a>0$.

Notice that introducing Fourier multipliers $\J^\delta$ is essentially costless from the point of view of the numerical integration through Fourier pseudo-spectral methods. 
Moreover, since admissible rectifiers $\J^\delta$ are symmetric for the $L^2(\RR^d)$ inner product,~\eqref{RWW2-intro} preserves the canonical Hamiltonian structure~\eqref{eq.Hamilton} with the modified functional
\begin{equation}\label{eq.Hamiltonian-mod}
	\cH^\mu_{\J^\delta}(\zeta,\psi)\eqdef\frac12\int_{\RR^d} \zeta^2+ \psi G^\mu_0 \psi + \epsilon(\J^\delta \zeta)\left( |\nabla\psi|^2-( G^\mu_0 \psi)^2\right)\dd \bx.
\end{equation}
In particular by Noether's theorem or direct inspection, conserved quantities (invariants) of~\eqref{RWW2-intro} include the Hamiltonian (representing the total energy) $\cH^\mu_{\J^\delta}(\zeta,\psi)$, the excess of mass, $\int_{\RR^d}\zeta\dd\bx$, and the horizontal impulse, $\int_{\RR^d}\zeta\nabla\psi\dd\bx$.

\subsection{Description of the results and recommendations}

Let us now describe and quantify gains and losses associated with the introduction of rectifiers. 

First of all, let us acknowledge that
it is obvious that embedding regularizing operators in system~\eqref{WW2-intro} allows to provide the local-in-time well-posedness of the initial-value problem in suitable finite regularity spaces, by means of the Picard–Lindelöf (or Cauchy–Lipschitz) theorem in Banach spaces. In fact it is fairly easy to check ---as we do in this work--- that~\eqref{RWW2-intro} is indeed of semilinear nature. However, the time of existence of solutions is expected to vanish as $\delta\searrow 0$, since then $\J^\delta$ approaches the identity. Following the approach described above, we obtain a lower bound on the existence time scaling as $\delta/\epsilon$.  Our first main result provides a {\em large time} existence, \ie an existence time scaling as $1/\epsilon$ under the assumption that $\delta \gtrsim \epsilon$. In order to reach this timescale, we use delicate energy estimates, relying in particular on the equivalent of Alinhac's good unknowns which are crucial in the analysis of the water waves systems; see~\cite[\S4]{Lannes_ww}. This allows to pinpoint contributions which crucially require regularization in order to avoid high-frequency instabilities. The fact that these contributions scale as cubic terms is, as already mentioned, one of the key observations of our analysis. Indeed, after applying rectifiers, these terms will lead to a restriction on the time of existence scaling as $\delta/\epsilon^2$, thus motivating the aforementioned restriction $\delta \gtrsim \epsilon$. Incidentally, our analysis also uncovers a stability criterion for~\eqref{RWW2-intro} of Rayleigh-Taylor type, which is automatically satisfied provided that $\epsilon$ and $\delta^{-1}\epsilon^2$ are sufficiently small. 

The second effect of introducing rectifiers is that it may deteriorates the accuracy of the produced solutions with respect to solutions of the water waves system. In fact, we prove that introducing rectifiers near-identity of order $\a$ will induce a loss of precision of order $\cO(\epsilon \delta^\a)$, associated with a loss of $\a+p$ derivatives (with $p$ some constant) between the control of the data and the control of the error. Recalling that the quadratic model~\eqref{RWW2-intro} already generates errors of size $\cO(\epsilon^2)$, one immediately realizes that the errors induced by rectifiers can be made asymptotically negligible by imposing an upper bound on $\delta$ depending on $\epsilon$.

As clarified in the above discussion, we observe that taking advantage of the regularizing effect of rectifiers in one hand, and accommodating their cost in terms of precision on the other hand, drive competing demands on the scaling parameter $\delta$. Yet the aftermath is that {\em it is possible to choose $\J^\delta$ and especially $\delta$ as a function of $\epsilon$ so that considering~\eqref{RWW2-intro} rather than~\eqref{WW2-intro} does not deteriorate the precision of the system as an asymptotic model for the water waves system as $\epsilon\searrow 0$, while allowing to control solutions on a relevant time interval and eventually fully justify~\eqref{RWW2-intro}.}

Specifically, we recommend the use of $\J^\delta$ defined by~\eqref{Rectifier_example} with $\r=-1$, 
and the scalings expressed in our theoretical results argue for a choice of $\delta$ which is proportional to $\epsilon$.
Yet in practice one can simply set $\delta$ by trial and error, choosing $\delta$ sufficiently large so that the numerically computed high-wavenumber modes do not suffer from spurious amplification, yet sufficiently small so that the outcome of the numerical simulation does not depend on $\delta$ up to the desired accuracy.

\subsection{Discussion and prospects}\label{S.discussion}

\paragraph{The instability mechanism} 

Our diagnosis of the instability mechanism and the motivation for introducing rectifiers in~\eqref{RWW2-intro} stems from the forthcoming \Cref{P.quasi}, where we extract a ``quasilinear structure'' of the system, of the form
\[\left\{\begin{array}{l}
	\partial_t \dot\zeta-G^\mu_0 \dot\psi  + \epsilon \nabla\cdot(\nabla\psi(\J^\delta \dot\zeta)) =\text{lower order terms} ,\\[1ex]
	\partial_{t} \dot\psi + \mfa_{\J^\delta}[\epsilon  \zeta, \epsilon\psi] \dot\zeta + \epsilon\J^\delta \big( \nabla \psi \cdot \nabla \dot\psi \big) = \text{lower order terms},
\end{array}\right.
\]
where the ``Rayleigh--Taylor'' (see forthcoming \Cref{R.RT}) operator is defined as
\[
\mfa_{\J^\delta}[\epsilon \zeta, \epsilon\psi] f \eqdef f - \epsilon (G^\mu_0 \J^\delta \zeta) \J^\delta  f - \epsilon^2\J^\delta \Big((G^\mu_0 \psi) |D| \big\{ (G^\mu_0 \psi) \J^\delta f \big\}\Big).
\]
In the absence of rectifiers (that is setting $\J^\delta=\Id$), this quasilinear structure is of elliptic type (with the same nature as Cauchy--Riemann equations) because $(\mfa_{\Id}[\epsilon\zeta,\epsilon\nabla\psi] f,f)_{L^2}$ can take arbitrarily large negative values for some smooth $f$ with $\Norm{f}_{L^2}=1$, as soon as $\epsilon G^\mu_0\psi\neq 0$. As such, the instability mechanism we exhibit is strikingly similar to the well-studied Kelvin--Helmholtz instabilities (see \eg~\cite[\S4]{Lannes13} and references therein). Notice however that in our case the instability mechanism does not relate to any physically relevant phenomenon, and that it is fully nonlinear in the sense that it cannot be revealed through linearization about equilibria. Our strategy for taming these instabilities also differs from the works we are aware of on Kelvin--Helmholtz instabilities, and in particular~\cite{Lannes13}: while we could (artificially) add surface tension contributions, it appears less invasive and closer to standard numerical approaches to incorporate rectifiers $\J^\delta$ as we do. In both approaches, the goal is to turn the nature of the quasilinear system from elliptic to hyperbolic (at least for sufficiently regular and small data) by ensuring the Rayleigh--Taylor condition:
\[
\exists\mfa_\star>0, \ \forall f\in L^2(\RR^d) ,\quad \left( f , \mfa_{\J^\delta}[\epsilon  \zeta,\epsilon \psi]  f\right)_{L^2} \geq \mfa_\star \norm{ f }_{L^2}^{2}.
\]
As mentioned previously, for sufficiently regular and bounded functions $\zeta$ and $\psi$, the above holds provided that $\J^\delta$ is regularizing of order $\r=-1/2$ and that $\epsilon$ and $\delta^{-1}\epsilon^2$ are sufficiently small.

\paragraph{A toy model}
We investigate our proposed instability mechanism in more details through a toy model in \Cref{S.toy}. There we prove local well-posedness results when rectifiers are sufficiently regularizing, as well as a strong ill-posedness result when rectifiers are not sufficiently regularizing. The proof of ill-posedness exhibits three successive stages in the instability mechanism: 
\begin{enumerate}
	\item first the amplification of the high-frequency component is triggered by the presence of a substantial low-frequency component;
	\item then the high-frequency component reaches a threshold so that it is able to fuel its own amplification;
	\item finally nonlinear effects come into play and achieve the finite-time blowup by means of a Riccati inequality.
\end{enumerate} 
Numerical experiments confirm that, despite its simplicity, our toy model is able to describe at least the first stage of the instability mechanism for~\eqref{WW2-intro} not only qualitatively, but also correctly predicts the different scales involved and in particular blowup times. While we do not present all these numerical experiments  in this manuscript for the sake of brevity, a detailed report is available in the preliminary version \cite{DucheneMelinand}. A discussion comparing our results on the toy model and corresponding results for~\eqref{WW2-intro} is presented in \Cref{R.toy-for-the-win}.

\paragraph{Other models with quadratic precision}

A key ingredient in our analysis is the observation that the instability mechanism stems from cubic terms 
while~\eqref{WW2-intro} is a quadratic model. This is consistent with the fact that the fully nonlinear system~\eqref{WW} is well-posed, from which one can expect that any system based on truncated expansions should be in some sense ``stable up to higher order terms'', and thus likely to be amenable to our rectification procedure. However it is natural to ask whether it is possible to exhibit models ---either using different sets of unknowns or adding additional terms--- for which the initial-value problem is well-posed, without relying on artificial regularizing operators. This was in fact the main motive that triggered this work. As discussed in the introduction, several quadratic models for deep water waves have been proposed in the literature, but only the one studied in~\cite{saut_xu} has been shown to be well-posed in finite regularity spaces. Note, however, that the latter model uses a change of variables for the unknown describing the surface deformation, which can be considered undesirable. In Appendix C of the preliminary version of this manuscript \cite{DucheneMelinand} (again withdrawn for the sake of brevity), we propose another model with a relatable set of unknowns which has the same precision as~\eqref{WW2-intro}, and which we argue is well-posed in finite regularity spaces. However, we acknowledge that this system involves fully nonlinear operators, and is therefore less suitable than~\eqref{RWW2-intro} for numerical simulations. Moreover, it appears that the algebra used to produce this model cannot be extended to cubic or higher order models, while we do hope that the ``rectification'' method introduced in this work can be extended to a general framework; see below.

\paragraph{The infinite-depth situation }

We restrict our study to the deep-water framework, in the sense that our theoretical results are shown to hold uniformly with respect to $\mu\in[1,+\infty)$, {\em finite}. A difficulty arises in the infinite-depth case (\ie setting $\mu=\infty$, $G^\infty_0 =|D|$) due to the fact that the operator $\mfP:\Hdot^{s}(\RR^d)\to H^{s-1/2}(\RR^d)$ is not bounded uniformly with respect to $\mu\geq 1$ (see \Cref{L.control_P}), and hence Beppo Levi spaces are no longer suitable as functional spaces for the unknown describing the trace of the velocity potential, $\psi$. A first remark is that our results continue to hold and extend straightforwardly to the infinite-depth case if we restrict the framework to $\psi\in H^{s}(\RR^d)$ instead of $\psi\in \dotH^{s}(\RR^d)$ since in that case $\mfP:H^{s}(\RR^d)\to H^{s-1/2}(\RR^d)$ is bounded uniformly with respect to $\mu\geq 1$. This was the choice made in~\cite{Lannes_ww} (see Remark 2.50 therein), but it can be considered as too restrictive as it imposes some decay at infinity through the assumption $\psi\in L^2(\RR^d)$. A bigger space would consist in using instead $\psi \in \breve{H}^s(\RR^d)=\dot{H}^{1/2}(\RR^d)\cap \Hdot^s(\RR^d)$ (with $s\geq 1/2$) which is the natural energy space defined by the Hamiltonian \eqref{eq.Hamiltonian} when $\mu=\infty$. Here, the homogeneous Sobolev space $\dot{H}^{1/2}(\RR^d)$ is defined as $\dot{H}^{1/2}(\RR) = \{ f \in BMO(\RR) \text{ , } |D|^{\frac12} f \in L^{2}(\RR) \}$ when $d=1$ and by $\dot{H}^{1/2}(\RR^2) = \{ f \in L^{4}(\RR^2) \text{ , } |D|^{\frac12} f \in L^{2}(\RR^2) \}$ when $d=2$. Therein, we define $|D|^{\frac12} f$ for any tempered distribution $f$ though the Littlewood--Paley decomposition, namely as $\sum_{j \in \mathbb{Z}} |D|^{\frac12} (\phi(2^{-j-1} |D|) - \phi(2^{-j} |D|)) f$ where $\phi$ is a given real-valued smooth even function supported in the ball $[-2,2]$ and that is equal to $1$ on $[-1,1]$, realizing a dyadic partition of unity. Note that  $\dot{H}^{1/2}(\RR)/\RR$ and $\dot{H}^{1/2}(\RR^2)$ are complete and can be identified to the completion of Schwartz class functions for the norm $\| |D|^{\frac12} \cdot \|_{L^{2}}$; see~\cite{MonguzziPelosoSalvatori20,BGV21}. 

\paragraph{The shallow-water situation } In the opposite direction, let us discuss the shallow-water situation, namely $\mu\in(0,1]$.  Firstly, a rescaling must be performed so that~\eqref{WW2-intro} remains non-trivial as $\mu\searrow 0$ (see~\cite[Appendix~A]{Alvarez_Lannes} for the water waves system). This amounts in replacing~\eqref{WW2-intro}
with
\begin{equation}\label{WW2-SW}\tag{WW2'}\left\{\begin{array}{l}
		\partial_t\zeta- \frac1{\sqrt\mu}G^\mu_0\psi+ \varepsilon G^\mu_0 (\zeta G^\mu_0 \psi)+\varepsilon\nabla\cdot( \zeta\nabla\psi)=0,\\[1ex]
		\partial_t\psi+\zeta+\frac\varepsilon{2 }  \left( |\nabla\psi|^2- (G^\mu_0 \psi )^2 \right)=0,
	\end{array}\right.
\end{equation}
where the dimensionless parameter $\varepsilon\eqdef\epsilon/\sqrt\mu$ represents the ratio of the amplitude of the wave to the depth of the layer at rest.  
One of the key ingredients in our analysis is the forthcoming \Cref{L.commutator-tanh}, exhibiting ``commutator'' estimates on the operator $\psi\mapsto G^\mu_0 (\zeta G^\mu_0 \psi)+\nabla\cdot( \zeta\nabla\psi)$. These estimates do not hold uniformly with respect to $\mu\in(0,1]$, since we have (for smooth data) $ G^\mu_0 (\zeta G^\mu_0 \psi)=\cO(\mu)$ as $\mu\searrow 0$.
However, one should notice that formally setting $\mu=0$ in~\eqref{WW2-SW} yields 
\[\left\{\begin{array}{l}
	\partial_t\zeta+ \nabla\cdot((1+\varepsilon \zeta)\nabla\psi)=0,\\[1ex]
	\partial_t\psi+\zeta+\frac\varepsilon{2 }  \left( |\nabla\psi|^2\right)=0.
\end{array}\right.
\]
Taking the gradient of the second equation, we obtain as expected the shallow-water system, which is a well-known example of symmetrizable hyperbolic systems. As such, its initial-value problem is well-posed for $(\zeta,\nabla\psi)\in H^s(\RR^d)^{1+d}$ for any $s>1+d/2$ under the non-cavitation condition $\inf_{\RR^d}(1+\varepsilon\zeta)>0$; see \eg~\cite{Lannes_ww}. Hence we believe that our results can be adapted to the shallow-water situation, although such a study should take into account the aforementioned commutator estimate in combination with the symmetric structure appearing in the shallow-water equation.

\paragraph{Higher order spectral method} The climax of our analysis, displayed in forthcoming \Cref{T.Convergence}, is that the ``rectified'' model,~\eqref{RWW2-intro}, is able to approximate solutions to the water waves system,~\eqref{WW}, with the desired accuracy for well-chosen choices of rectifiers, and suitable values for the strength $\delta$. Introducing rectifiers is essentially costless from the point of view of numerical integration when pseudo-spectral schemes are employed. Our rectification strategy is in some sense related to standard numerical strategies which consist in applying appropriate low-pass filters. The difference (in addition to provide a rigorous justification) is that we are able to point out  precisely where regularization should be introduced, and to assess (i) its order as a regularizing operator, and (ii) its strength in order to provide sufficient regularization while at the same time not deteriorate the accuracy of the computed solution. 

We believe that our strategy can be adapted to the whole hierarchy of systems with arbitrary order put forward by Craig and Sulem in~\cite{Craig_Sulem_1}. Specifically, if we denote $G_N^\mu$ the truncated expansion at order $N$ of the Dirichlet-to-Neuman operator, 
\[  G[\epsilon\zeta]\psi= G_N^\mu[\epsilon \zeta] \psi +\cO(\epsilon^{N+1}),\]
and consider the system obtained by Hamilton's equations
 \begin{equation}\label{eq.Hamilton-N}
	\left\{\begin{array}{l}
		\partial_t\zeta-\delta_\psi \cH^\mu_{N,\J^\delta}=0,\\[1ex]
		\partial_t\psi+\delta_\zeta \cH^\mu_{N,\J^\delta}=0,
	\end{array}\right.
\end{equation}
associated with the functional
\begin{equation}\label{eq.Hamiltonian-N}
	\cH^\mu_{N,\J^\delta}(\zeta,\psi)\eqdef\frac12\int_{\RR^d} \zeta^2+ \psi G_N^\mu[\epsilon \J^\delta \zeta] \psi \dd \bx,
\end{equation}
then we conjecture that for any fixed $N\in\NN$, the initial-value problem for system~\eqref{eq.Hamilton-N}--\eqref{eq.Hamiltonian-N} is well-posed on a time interval of size $1/\epsilon$ and produces approximate solutions with precision $\cO(\epsilon^{N+1})$, as soon as rectifiers $\J^\delta$ are admissible and $\delta\approx \epsilon^N$.

 Moreover, we also conjecture the stronger result that fixing $\epsilon$ sufficiently small and (say) $\delta=\epsilon^N$, then denoting $(\zeta_N,\psi_N)$ the solutions to~\eqref{eq.Hamilton-N}--\eqref{eq.Hamiltonian-N} emerging from sufficiently regular initial data and $(\zeta,\psi)$ the corresponding solution the water waves system~\eqref{WW}, we have $(\zeta_N,\psi_N)\to (\zeta,\psi)$ as $N\to\infty$  on a time interval determined by $(\zeta,\psi)$ and hence independent of $N$, with a convergence rate of order $\cO((C\epsilon)^{N+1})$ with $C\epsilon<1$.

Such a result would rigorously justify a robust and efficient strategy to approximate solutions to the water waves system with arbitrary accuracy. 
We leave this topic for a future study.

\subsection{Outline}

Let us now describe the remaining content of this paper.

In \Cref{S.numerics} we report on an in-depth numerical study in view of validating the proposed instability mechanism as well as our rectifying strategy.

In \Cref{S.main} we collect notations used in this work and state our main analytical results, that is 
\begin{enumerate}
	\item \Cref{T.WP-delta} on the large-time well-posedness of the initial-value problem for~\eqref{RWW2-intro};
	\item \Cref{T.Consistency} on the consistency of~\eqref{RWW2-intro} with respect to the water-waves system~\eqref{WW};
	\item \Cref{T.Convergence} measuring the difference between solutions to~\eqref{WW} and solutions to~\eqref{RWW2-intro} emerging from the same initial data.
\end{enumerate} 

In \Cref{S.technical} we introduce some important technical tools: product and commutator estimates and key results on the  operator $G^\mu_0$ and its infinite-depth counterpart $G^\infty_0=|D|$.

In \Cref{S.consistency} we prove \Cref{T.Consistency,T.Convergence} that provide the full justification of~\eqref{RWW2-intro} as a deep-water model for the water wave system~\eqref{WW}.

In \Cref{S.WP} we prove \Cref{T.WP-delta} as the consequence of two results: 
\begin{itemize}
	\item an unconditional ``short-time'' well-posedness result, \Cref{P.WP-short-time}, proved in \Cref{S.WP-short-time};
	\item a conditional ``large-time'' well-posedness result, \Cref{P.WP-large-time}, which is proved in \Cref{S.WP-large-time}.
\end{itemize}
The latter may be considered as the main and most technical result of this work. Finally, we prove a global-in-time well-posedness result for sufficiently small data, \Cref{P.WP-global-in-time}, in \Cref{S.WP-global}.

We recall some information on the water waves system and the Dirichlet-to-Neumann expansion in \Cref{S.WW}.

A toy model for the proposed instability mechanism is introduced and studied in details in \Cref{S.toy}.

\section{Numerical study}\label{S.numerics}

In this section we perform numerical experiments in views of illustrating our results and evaluating their sharpness. In order to do that, we shall numerically investigate different parameters at stake in our rectification procedure. First we examine how, in the absence of any regularization, the numerical discretization of system~\eqref{WW2-intro} generates dramatic high-wavenumber amplification. Then we validate that theses instabilities are absent when using the proposed rectified system~\eqref{RWW2-intro}, provided that the rectifier, $\J^\delta$, is regularizing of sufficiently negative order and its strength, $\delta$, is sufficiently large. In both aspects, the observed threshold on the order of the rectifier and its strength is shown to fully comply with our diagnosis of the instability mechanism. Finally, we examine the loss of accuracy caused by introducing the rectifiers. Once again we find that the numerical results reproduce the different features of our consistency and convergence estimates. Finally we confirm that, at least for the initial data and values of dimensionless parameters we considered, it is possible to choose appropriately rectifiers and their strength so as to suppress undesirable high-frequency instabilities while being harmless from the point view of the accuracy of the model.

\medskip

Let us first describe the numerical method we employ in order to produce approximate solutions to systems~\eqref{WW2-intro} and~\eqref{RWW2-intro} (restricted here to horizontal dimension $d=1$). The figures are obtained using a code written in the Julia programming language~\cite{Julia}, and specifically the package written by the first author and P. Navaro \cite{DucheneNavaro22}. They are fully reproducible using the scripts available at \texttt{WaterWaves1D.jl/examples/StudyRectifiedWW2.jl}.

\paragraph{Numerical scheme}
We use Fourier (pseudo-)spectral methods for the spatial discretization. We shortly describe the principles thereafter, and let the reader refer to \eg~\cite{Trefethen00} for more details. A periodic\footnote{ \label{f1}In the foregoing numerical experiments we use rapidly decaying initial data and $L$ large, so that the periodic problem, $x\in L\TT$, provides a close approximation of the real-line problem, $x\in\RR$, at least for sufficiently small times.}  function with period $P=2L$ is approximated as the superposition of an even number, $N$, of monochromatic waves: 
\[ f(x) \ \approx \ \sum_{k=-N/2}^{N/2-1} \widehat f_k e^{\i\frac{\pi}{L}k x}\]
where we refer to $(\widehat f_k)_{k=-N/2,\dots,N/2-1}$ as the discrete Fourier coefficients. 
In practice, we can efficiently compute the discrete Fourier coefficients from the values of the function at regularly spaced collocation points, $(f(x_n))_{n=0,\dots,N-1}$ where $x_n=-L+2\frac{n}{N} L$ through the Fast Fourier Transform (FFT), and conversely recover values at collocation points from discrete Fourier coefficients through the inverse Fast Fourier Transform (IFFT). This allows to perform efficiently (and without introducing approximations except for rounding errors) the action of Fourier multiplication operators through pointwise multiplication on discrete Fourier coefficients, and the action of multiplication in the physical space at collocation point. However the latter operation cannot be performed exactly due to aliasing effects, stemming from the fact that $x\mapsto 1$ and $x\mapsto e^{\i\frac{\pi}{L}Nx}$ are indistinguishable in our collocation grid. Aliasing effects are known to possibly generate spurious numerical instabilities. In this case one customarily compute a Galerkin approximation (since the error of the approximation is orthogonal all considered monochromatic waves) of products by setting first a sufficient number (since our problem includes only quadratic nonlinearities, we use Orszag's 3/2 rule~\cite{Orszag72}) of discrete Fourier coefficients to zero. Hence in practice we solve, unless otherwise stated,
\begin{equation}\label{RWW2-dealias}\left\{\begin{array}{l}
		\partial_t\zeta- G^\mu_0\psi+ \epsilon \fPi G^\mu_0 ((\J^\delta \zeta) G^\mu_0 \psi)+ \epsilon \fPi\nabla\cdot((\J^\delta \zeta)\nabla\psi)=0,\\[1ex]
		\partial_t\psi+\zeta+\frac\epsilon{2 } \fPi \J^\delta \left( |\nabla\psi|^2- (G^\mu_0 \psi )^2 \right)=0,
	\end{array}\right.\end{equation}
where $\fPi$ is the ``dealiasing'' projection operator onto $(2L)$-periodic functions with all but the first $\lfloor 2N/3\rfloor$ discrete Fourier coefficients equal to zero, initial data $\zeta_0=\zeta(t=0,\cdot)$, $\psi_0=\psi(t=0,\cdot)$ such that $\fPi\zeta_0=\zeta_0$ and $\fPi\psi_0=\psi_0$  and where either $\J^\delta=\Id$ or
	\begin{equation}\label{num-rectifier}
		\J^\delta= J(\delta D) \quad \text{ with } J(\bk)=\min(\{1,|\bk|^{\r}\}),
	\end{equation}
for some $\r\leq 0$. 
As for the time integration, we use the standard explicit fourth order Runge-Kutta method. Because our problem is stiff (since it involves unbounded operators before discretization), a stability condition on the time step must be secured to avoid spurious numerical instabilities. By Dahlquist's stability theory (see~\cite[Chapter 10]{Trefethen00}), and given the nature and order of the operators at stake, we expect that it is sufficient to enforce $\triangle\! t\leq C {\triangle\! x}^{1/2} $ with $\triangle\! t$ the time step, $\triangle\! x=L/N$, and $C$ a sufficiently small constant. In practice we validate that the time-discretization induces no spurious result by checking that the observations are unchanged when varying the time step. Unless otherwise noted, we use $\triangle\! t=0.001$ in reported numerical experiments. 
Let us report that in the experiments corresponding to \Cref{F.Exp7d01}, the total energy (that is $\cH^\mu$) is relatively preserved up to $1.56\, 10^{-6}$ if $\triangle\! t=0.1$, $1.93\, 10^{-11}$ if $\triangle\! t =0.01$, and $9.13\, 10^{-16}$  if $\triangle\! t =0.001$; consistently with the expected accumulated error of order $\cO({\triangle\! t}^4)$.
This gives us a good confidence in the validity of the results of the following numerical experiments.

\paragraph{Instabilities without rectification}
First we exhibit the high-wavenumber amplification developed by numerical solutions when the equations are not regularized, that is considering~\eqref{RWW2-dealias} with $\J^\delta=\Id$, either with or without the dealiasing operator, $\fPi$. The upshot of these numerical experiments is that instabilities arise as soon as a sufficient number of Fourier modes are included, with or without dealiasing. 
We use for initial data 
\begin{equation}\label{num-init}
	\zeta_0(x)=\zeta(t=0,x)=\exp(-|x|^2)  \quad \text{ and } \quad \psi_0'(x)=\psi'(t=0,x)=0\ \quad (x\in(-L,L)) ,
\end{equation}
and set $\mu=1$, $\epsilon=0.1$, and $L=20$.

\begin{figure}[!htb]
	\begin{subfigure}{.5\textwidth}
		\includegraphics[width=\textwidth]{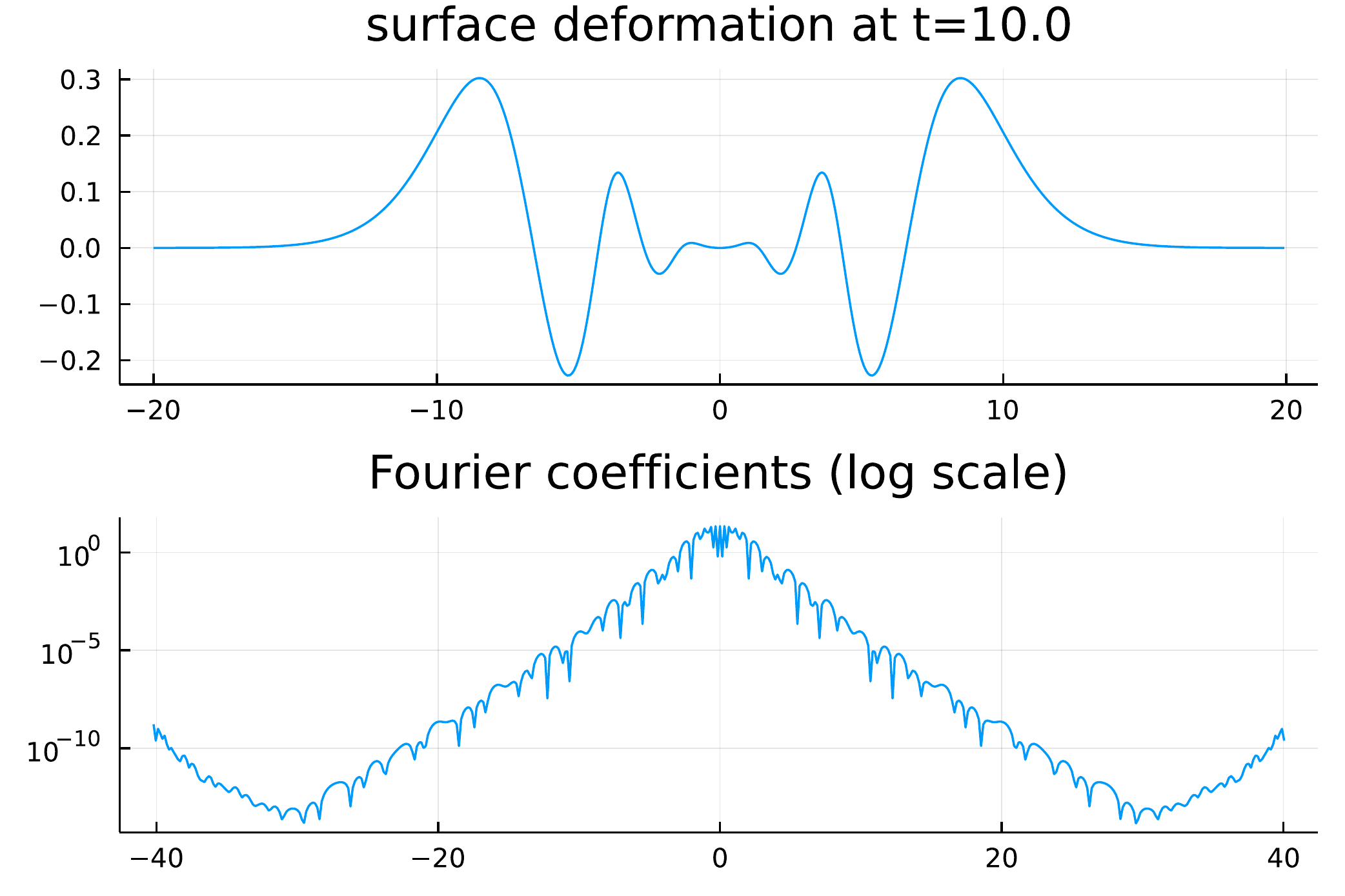}
		\caption{$N=2^9$ modes, at time $t=10$.}
		\label{F.Exp1}
	\end{subfigure}%
	\begin{subfigure}{.5\textwidth}
		\includegraphics[width=\textwidth]{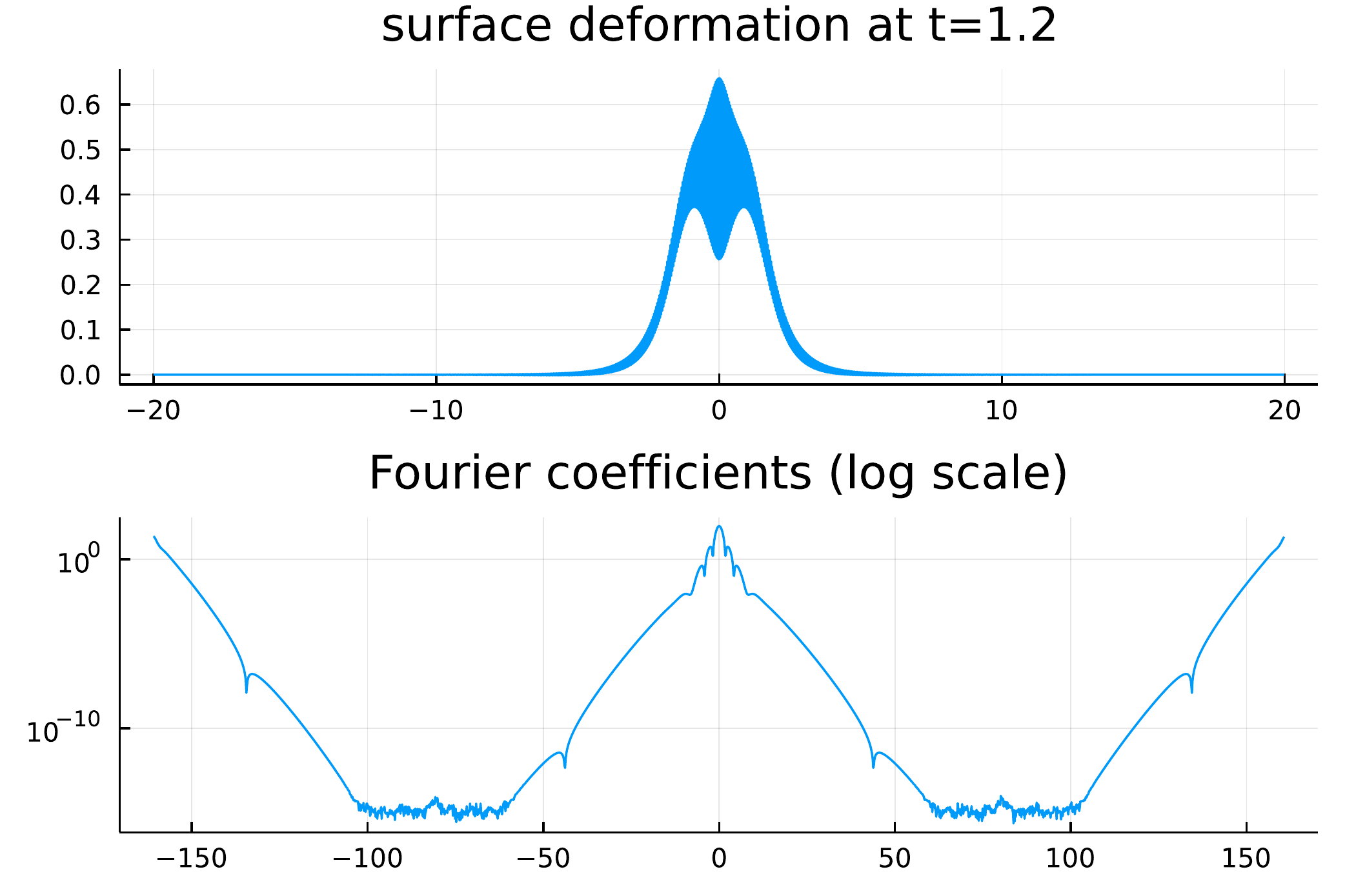}
		\caption{$N=2^{11}$ modes, at time $t=1.2$.}
		\label{F.Exp2}
	\end{subfigure}
	\caption{Time evolution of smooth initial data~\eqref{num-init}, without dealiasing and rectifiers. }
	\label{F.Exp1-2}
\end{figure}
\begin{figure}[htb]
	\begin{subfigure}{.5\textwidth}
		\includegraphics[width=\textwidth]{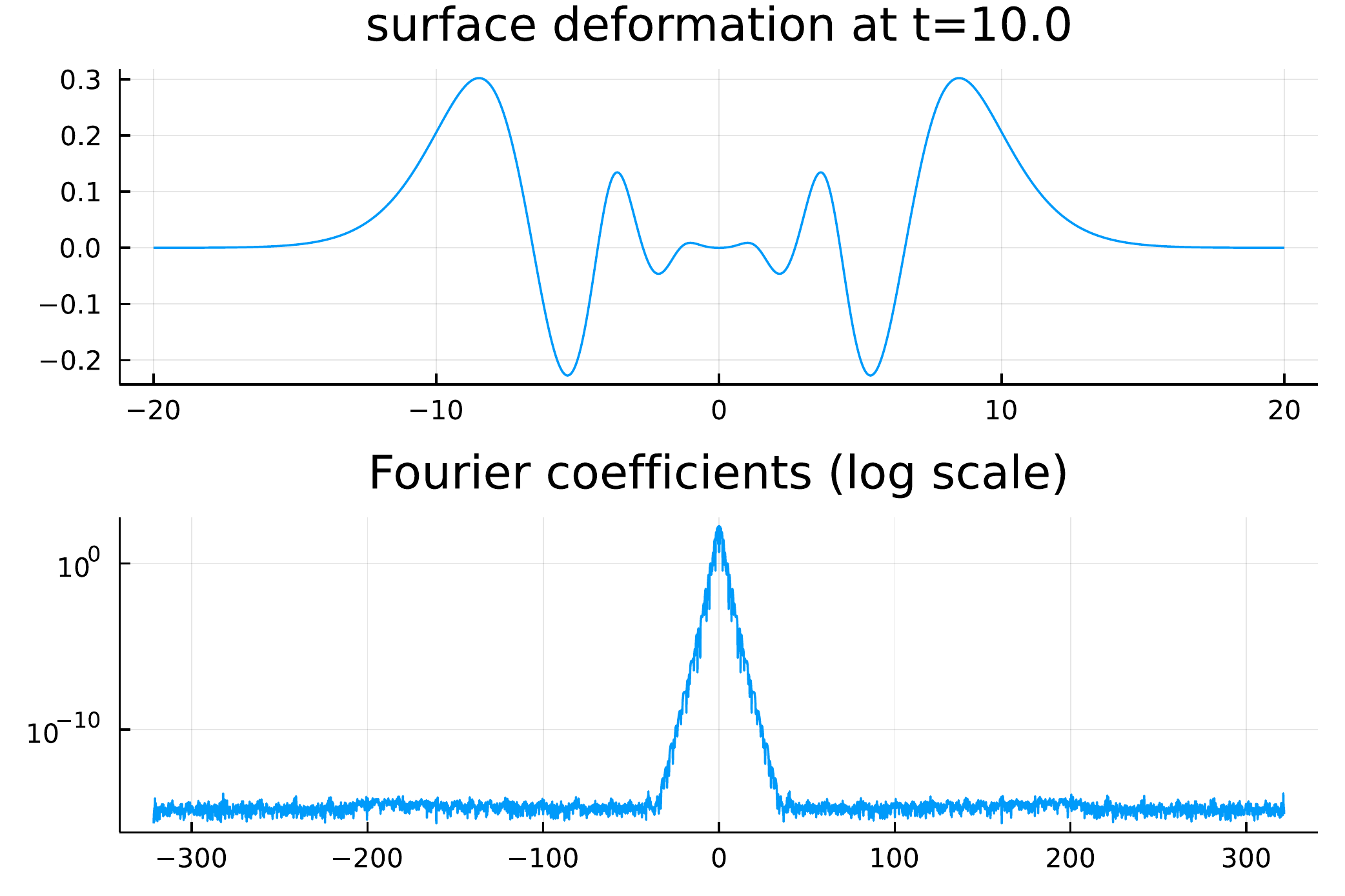}
		\caption{$N=2^{12}$ modes, at time $t=10$.}
		\label{F.Exp3}
	\end{subfigure}%
	\begin{subfigure}{.5\textwidth}
		\includegraphics[width=\textwidth]{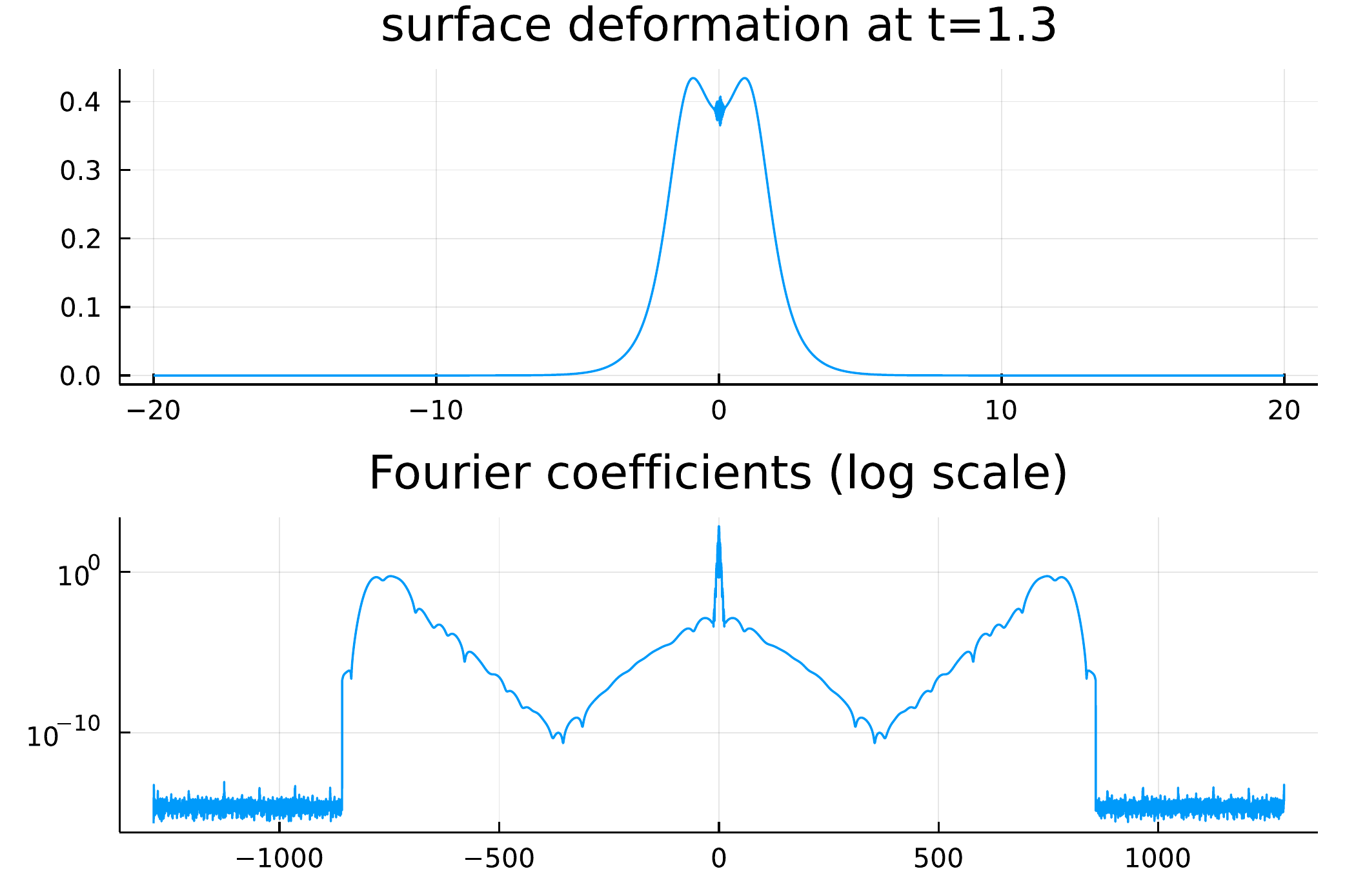}
		\caption{$N=2^{14}$ modes, at time $t=1.3$.}
		\label{F.Exp4}
	\end{subfigure}
	\caption{Time evolution of smooth initial data~\eqref{num-init}, with dealiasing and without regularization. }
	\label{F.Exp3-4}
\end{figure}

In \Cref{F.Exp1-2} (without dealiasing) and \Cref{F.Exp3-4} (with dealiasing) we plot the surface deformation, $(\zeta(t,x_n))_{n=0,\dots,N-1}$ at the final time of computation and at collocation points in the top panels, as well as the associated semi-log graph of the (modulus of) discrete Fourier coefficients, $ (\widehat \zeta_k(t))_{k=-N/2,\dots,N/2-1}$ in the bottom panels. In the left panels we report situations where the number of Fourier modes is too small to notice the rise of discrete Fourier coefficients with large wavenumbers, despite an inflection about extreme values, at least up to the final computation time $t=10$. In the right panel, we add more modes (with all other parameters kept identical) and observe that the large-wavenumbers component quickly grows to plotting accuracy and, as a consequence, is clearly visible on the top-right panel. The solution breaks down after just a few more time steps. The instability occurs more rapidly as more modes are added and does not depend on sufficiently small values of the time step. The presence of the dealiasing operator, $\fPi$, only augments the threshold on the number of modes above which the amplification of the high-wavenumber component of the numerical solution becomes visible. These observations are consistent with the ones already reported in~\cite{DommermuthYue87,numeric_Guyenne_moving_bott} for instance, and with the conjecture that the initial-value problem associated with the continous (\ie before discretization) system is ill-posed in finite-regularity spaces. Incidentally, let us clarify that the initial-value problem associated with the continuous problem is well-posed in the analytic framework (see~\cite{shkoller_etal2019} in the infinite-layer case) covering our initial data, and that what we observe in \Cref{F.Exp1-2} and \Cref{F.Exp3-4} is the growth of the spurious large-wavenumbers component generated by machine-precision rounding errors.

\paragraph{Stabilization through rectifiers}
Next we experiment the effect of introducing the operator $\J^\delta$ in~\eqref{RWW2-dealias}.
We set $\J^\delta$ as in~\eqref{num-rectifier} with $\delta=0.01$ for now, and $\r$ is a parameter which will vary. 
Through the parameter $\r$ we want to study the properties that rectifiers should satisfy in order to suppress instabilities. 
Following the terminology introduced in  \Cref{D.J-intro}, $\J^\delta$ is a regularizing operator of order $\r$ and is regular when $\r\in(-\infty,-1]$. Hence our well-posedness analysis is valid for all $\r\in(-\infty,-1]$, although we expect that setting to $\r\in(-\infty,-1/2]$ is sufficient to secure a well-posedness analogous to \Cref{T.WP-delta} (see \Cref{R.order-J}). 
As a matter of fact, we observe no sign of instabilities in our numerical experiments as long as $\r\leq -1/2$, while high-wavenumber amplification is clearly visible when $\r=-1/4$. 
Let us comment also that we have {\em not}  witnessed any undesirable high-wavenumber amplification when setting $\J^\delta$ as an ideal low-pass filter, \ie setting $\r=-\infty$, despite the lack of regularity of the corresponding symbol (the results of these numerical experiments are not shown). 

\begin{figure}[htb]
	\begin{subfigure}{.5\textwidth}
		\includegraphics[width=\textwidth]{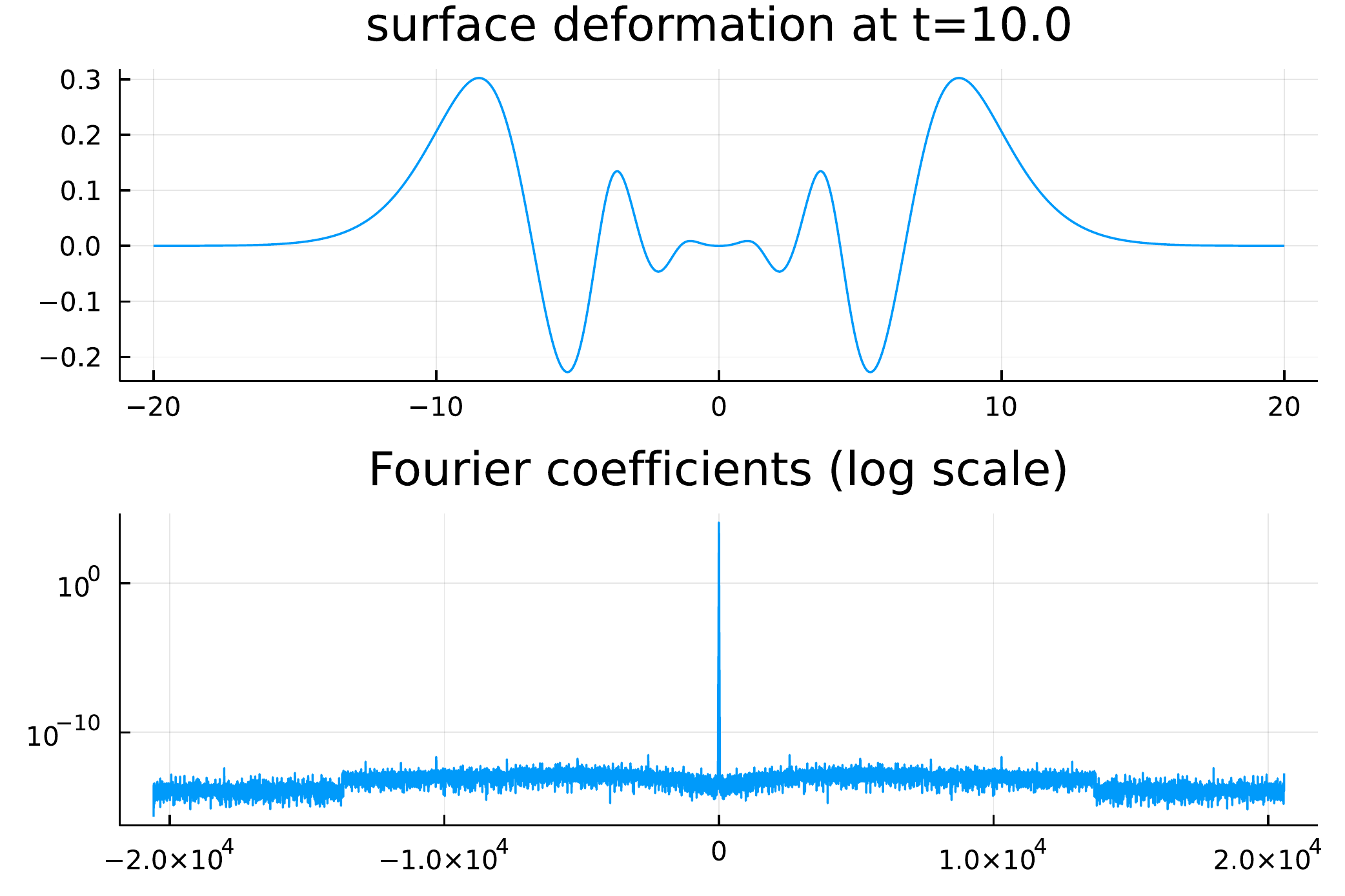}
		\caption{$\J^\delta$ of order $-1$.}
		\label{F.J1}
	\end{subfigure}%
	\begin{subfigure}{.5\textwidth}
		\includegraphics[width=\textwidth]{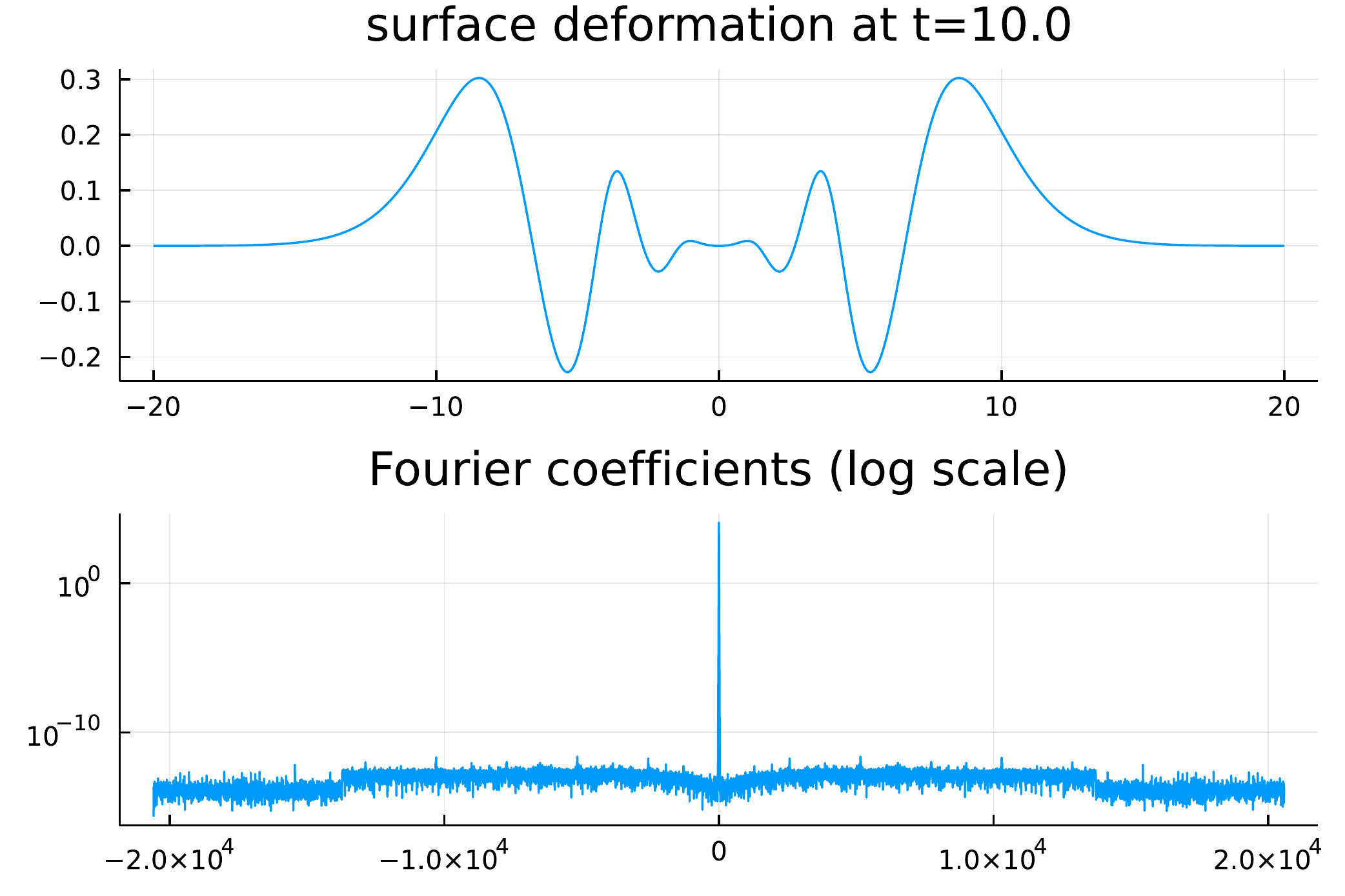}
		\caption{$\J^\delta$ of order $-1/2$.}
		\label{F.J1/2}
	\end{subfigure}
	\caption{Time evolution of smooth initial data~\eqref{num-init}, with a rectifier of order $\r\in\{-1,-1/2\}$ . }
	\label{F.J}
\end{figure}
\begin{figure}[htb]
	\begin{subfigure}{.5\textwidth}
		\includegraphics[width=\textwidth]{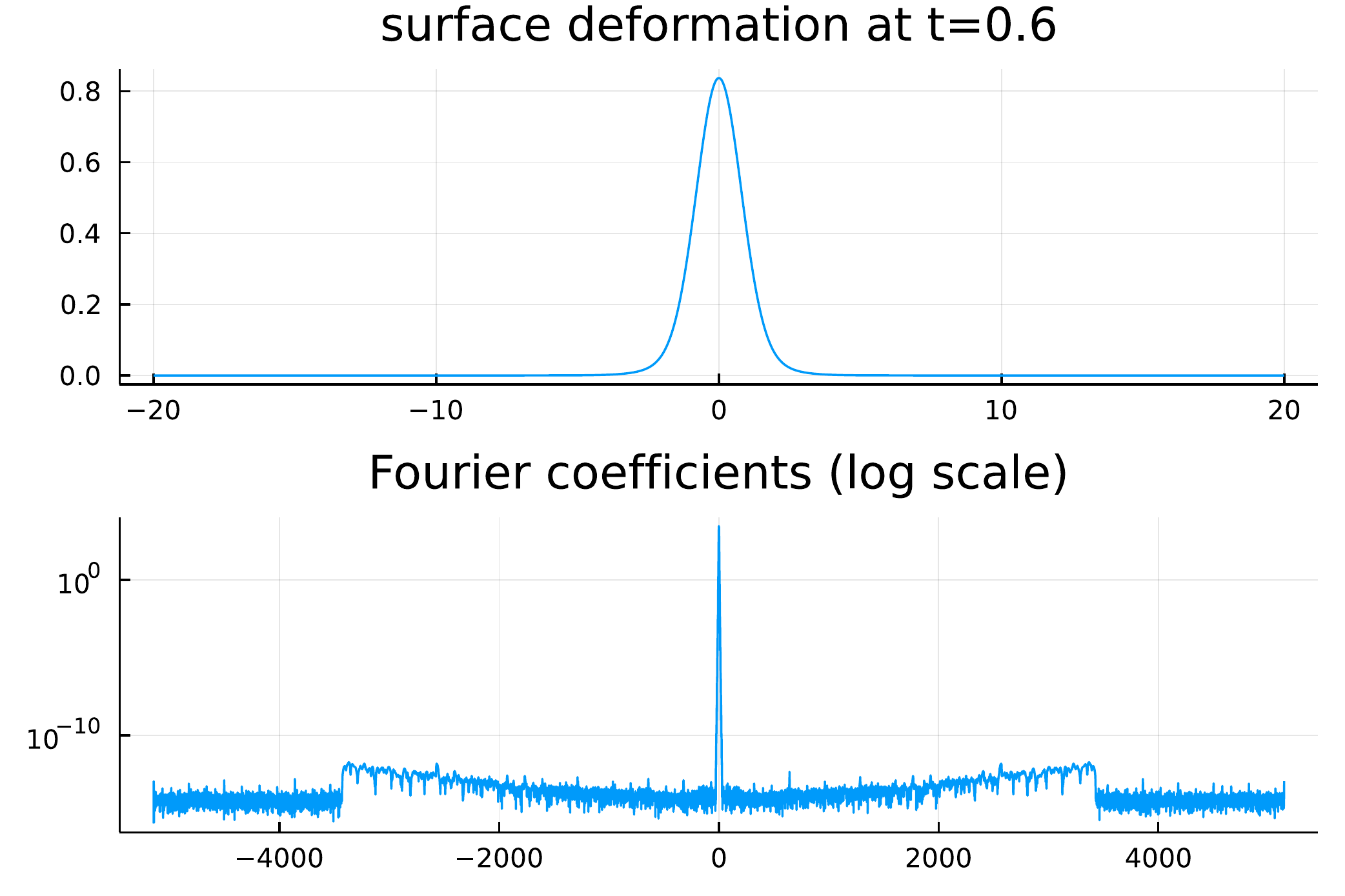}
		\caption{$\J^\delta$ of order $-1/4$, $N=2^{16}$ modes.}
		\label{F.J1/4N16}
	\end{subfigure}%
	\begin{subfigure}{.5\textwidth}
		\includegraphics[width=\textwidth]{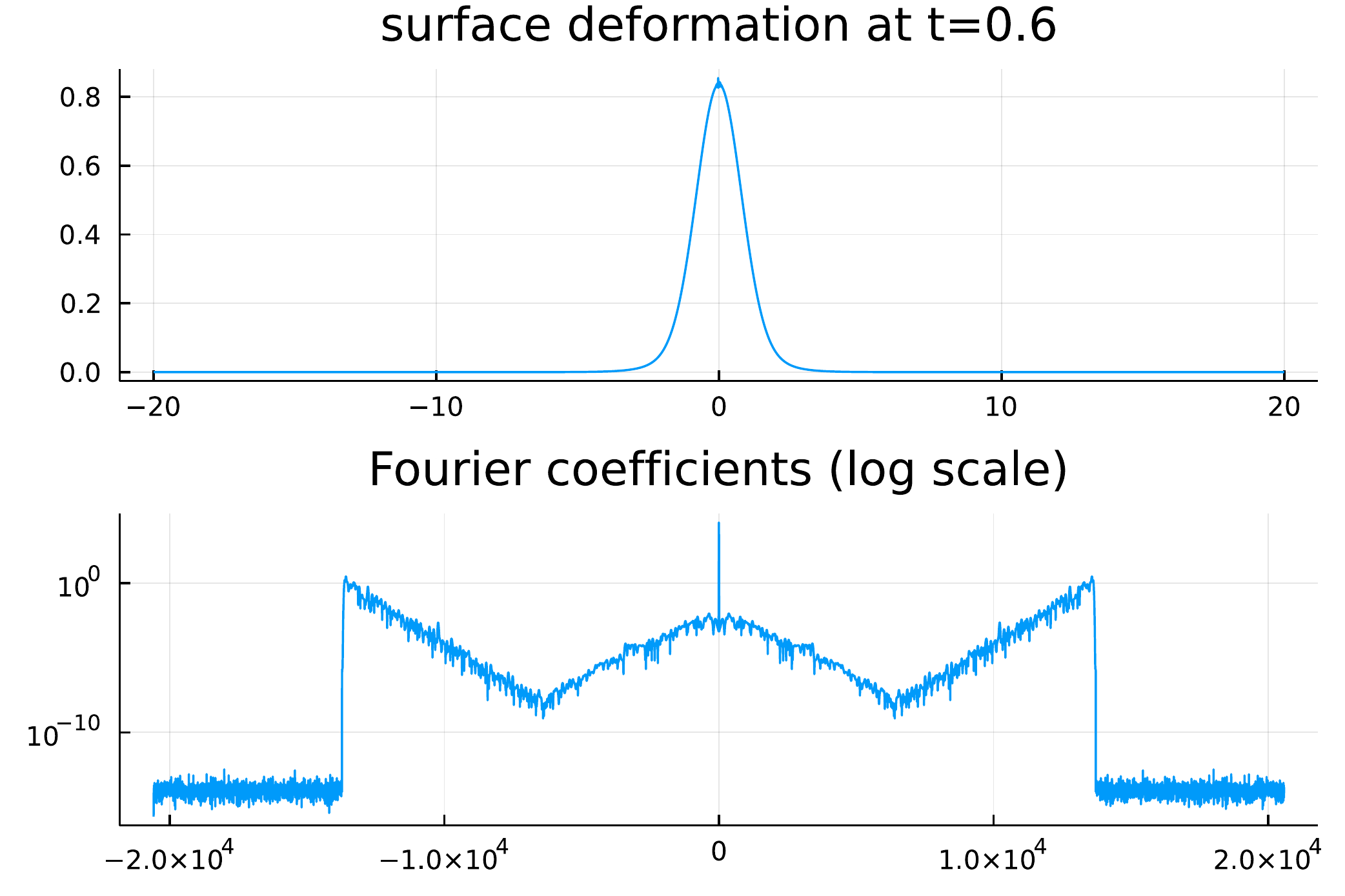}
		\caption{$\J^\delta$ of order $-1/4$, $N=2^{18}$ modes.}
		\label{F.J1/4N18}
	\end{subfigure}
	\caption{Time evolution of smooth initial data~\eqref{num-init}, with a rectifier of order $\r=-1/4$.}
	\label{F.J1/4}
\end{figure}

In \Cref{F.J,F.J1/4} we reproduce the experiment of \Cref{F.Exp3-4} ---that is setting initial data as in~\eqref{num-init}, $\mu=1$, $\epsilon=0.1$, and half-length $L=20$--- but this time with $\J^\delta$ as in~\eqref{num-rectifier}, $\delta=0.01$, and $\r\in\{-1,-1/2,-1/4\}$. We have seen in \Cref{F.Exp3-4} that setting $\r=0$, the time-evolution amplifies large-wavenumbers rounding errors. In \Cref{F.J} we plot the surface deformation in the top panels and the semi-log graph of the modulus of discrete Fourier coefficients in the bottom panel at time $t=10$,
with $\r=-1$ (left) and $\r=-1/2$ (right), and using $N=2^{18}$ modes (since the experiments are computationally more demanding, we use the time step $\triangle\! t=0.01$). There is no sign of instability. In \Cref{F.J1/4} we show the same results for $\r=-1/4$. On the right panel we show the result with $N=2^{18}$ modes, and clear signs of large-wavenumber amplification is visible at time $t=0.6$ (the solutions breaks after a few more time steps). We use only $N=2^{16}$ modes on the left panel and instabilities are tamed (yet still present and more easily witnessed at time $t=1$, not shown). This shows that the phenomenon strongly depends on the number of computed modes, $N$, and suggests that~\eqref{RWW2-dealias} suffers from high-frequency instabilities when $\J^\delta$ is regularizing operator of order $\r=-1/4$.
Hence these numerical experiments fully support our analysis as far as the order of the rectifier $\J^\delta$ as a regularizing operator is concerned.

\paragraph{The role of the strength of rectifiers}
Now we shall study the effect of the strength of rectifiers, by considering $\J^\delta$ as in~\eqref{num-rectifier} with $\r=-1$
and varying the parameter $\delta>0$. We observe a threshold value for $\delta$, above which the solution appears stable for all times, and below which the high-wavenumber amplification occurs rapidly. We later elucidate this threshold value in views of our theoretical results, and find that the latter are not only qualitatively but also quantitatively consistent with the numerical experiments.
	
	\begin{figure}[htb]
		\begin{subfigure}{.5\textwidth}
			\includegraphics[width=\textwidth]{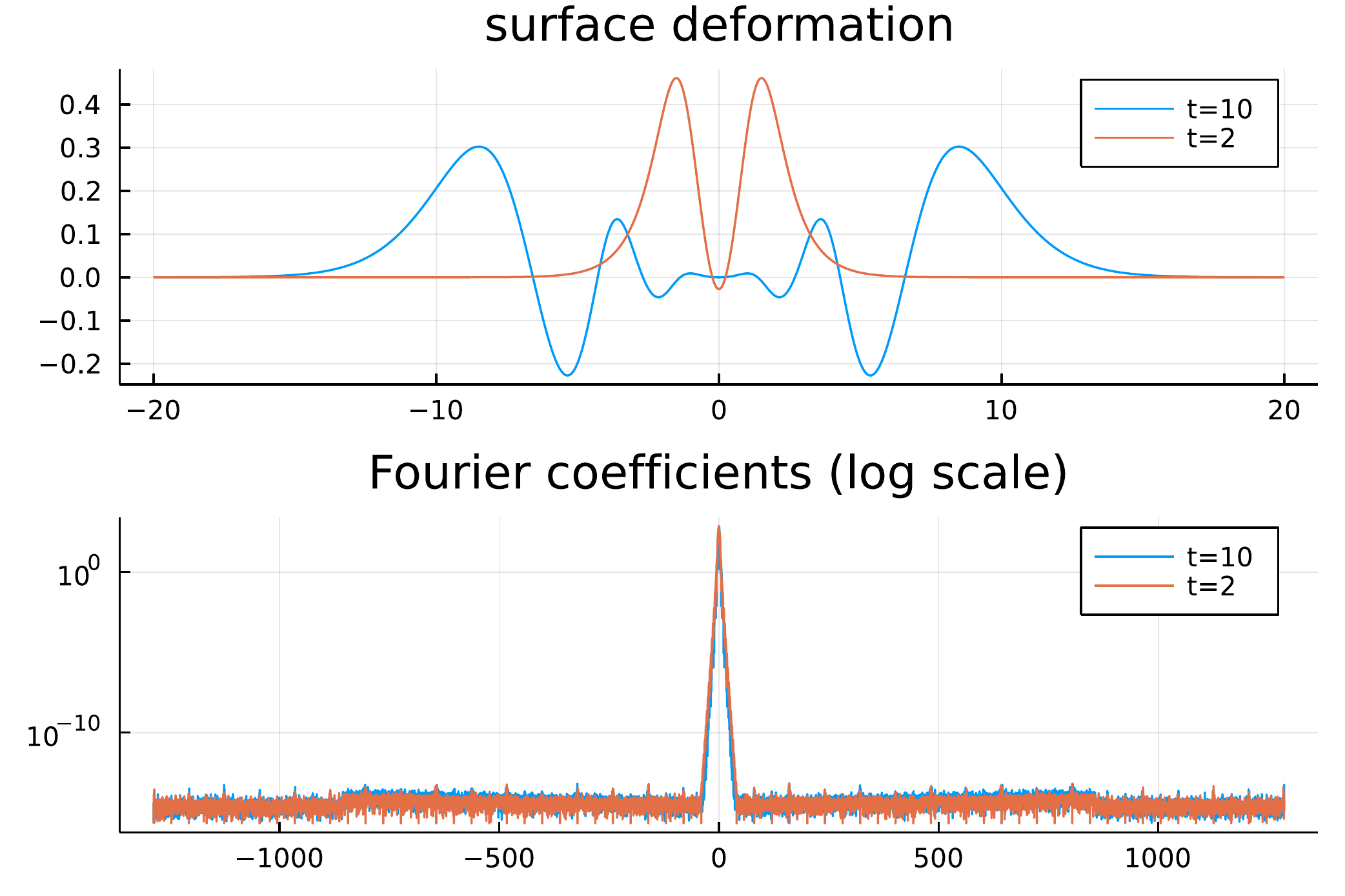}
			\caption{$N=2^{14}$ modes, $\delta = 0.01$.}
			\label{F.Exp7d01}
		\end{subfigure}%
		\begin{subfigure}{.5\textwidth}
			\includegraphics[width=\textwidth]{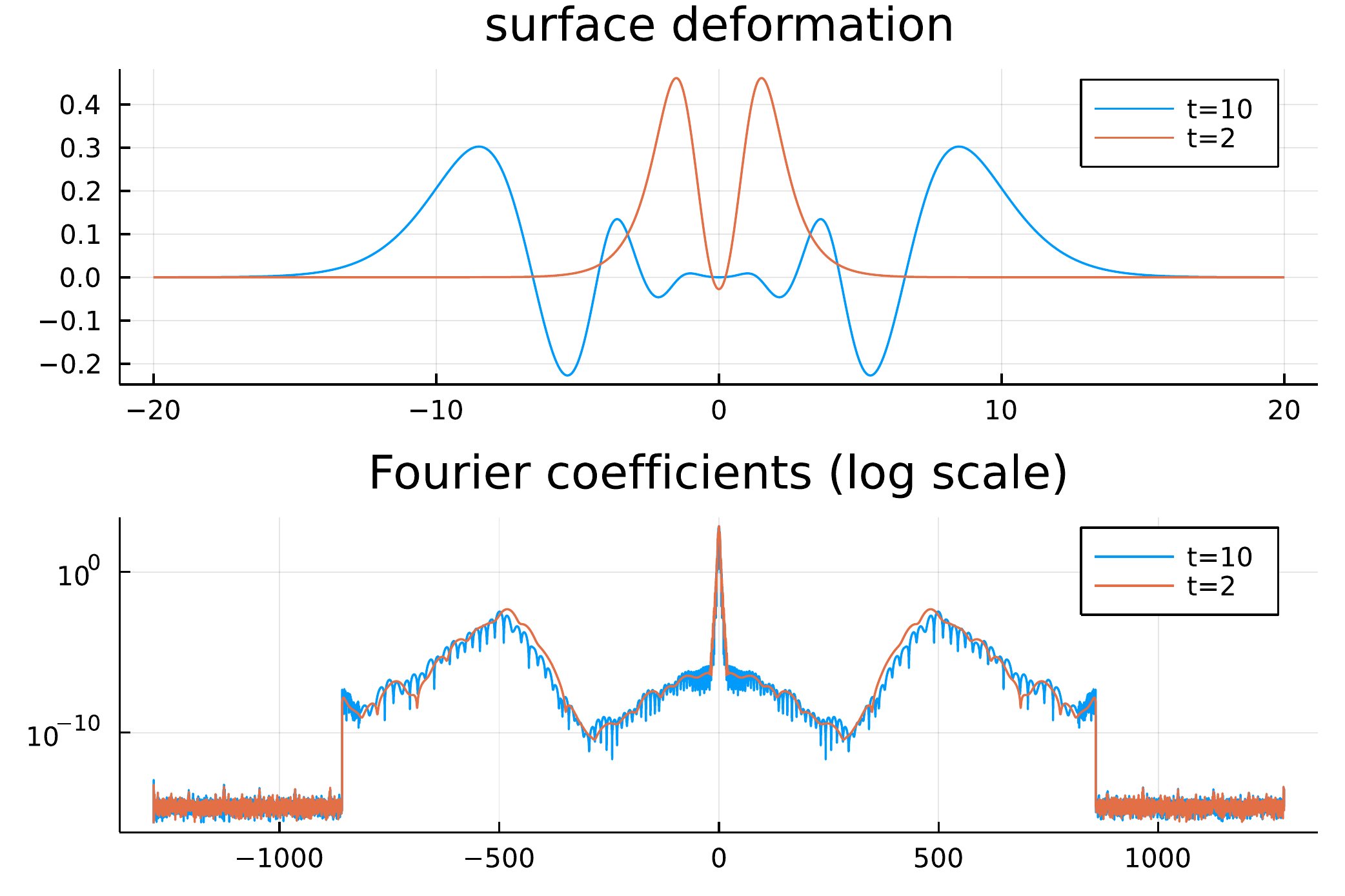}
			\caption{$N=2^{14}$ modes, $\delta = 0.002$.}
			\label{F.Exp7d002}
		\end{subfigure}
		\caption{Time evolution of smooth initial data %\eqref{num-init}, 
			varying the strength of an admissible rectifier.}
		\label{F.Exp7}
	\end{figure}

In \Cref{F.Exp7} we use again the initial data~\eqref{num-init}
and set $\mu=1$ and $\epsilon=0.1$, half-length $L=20$ and $N=2^{14}$ modes. We plot again the surface deformation in the top panels and the semi-log graph of the modulus of discrete Fourier coefficients in the bottom panel, at times $t=2$ and $t=10$. The left panel corresponds to the case $\delta=0.01$ and the right panel to $\delta=0.002$. The former shows no sign of instability, despite a minor amplification of machine epsilon rounding errors. Adding more modes does not change the picture. In the latter we see a clear amplification of large-wavenumber modes, with a maximum about the wavenumber $ \tfrac{2\pi}{2L}k \approx 1/\delta$. Yet the amplification arises at early times, and apparently remains stable for all times. Notice the amplification of intermediate wavenumbers is very sensitive to the parameter $\delta$: for $\delta=0.0025$ the amplification is barely noticeable (not shown). It is, however, stable with respect to parameters of the numerical scheme: augmenting the number of modes up to at least $N=2^{20}$ does not generate additional instabilities. For values below $\delta=0.001$, the amplification does not reach a stable regime, and the solution breaks before $t=2$. 
% Without dealiasing, the solution breaks even for large values of $\delta$.
We reproduced the phenomenon using other values of $\mu$, namely $\mu=10$ and $\mu=\infty$ as well as with less regular initial data, specifically~\eqref{num-init-p} with  $p=1$ and $p=3$. 
We do not display the results as they are very similar.

The outcome of these numerical experiments is again consistent with our analysis, and especially \Cref{P.WP-large-time}. It can be elucidated as follows: for sufficiently small values of $\delta$, the solution quickly violates the Rayleigh--Taylor condition~\eqref{RTpositive} and enters a regime where its large-wavenumbers component is amplified. We study further on the threshold value for $\delta$ and its dependence with respect to $\epsilon$ in the following paragraph.

\paragraph{Critical strength of rectifiers}
In the foregoing numerical experiments, we observed mainly two scenarios depending on values of the strength $\delta>0$ when the rectifier $\J^\delta$ is chosen to be sufficiently regularizing. For large values of $\delta$ the numerical solution seems to exist and remain regular for all times, while for small values the numerical solution rapidly breaks. We now investigate the transition between these two scenarios, conjecturing that for fixed initial data and time $T>0$,  there exists a unique ``critical value'' $\delta_{\rm c}(T)\geq 0$ such that for any $\delta>\delta_{\rm c}$ (resp. $\delta<\delta_{\rm c}$), the maximal time of existence of the solution is greater than $T$ (resp. smaller than $T$).
Computing numerical blowup times as the last computed time for which the produced solution does not involve NaN (as Not a Number) values, and using the dichotomy method, one may provide a range for this critical value, $\delta_{\rm c}$. 

\begin{figure}[!htb]
	\begin{center}\includegraphics[width=0.8\textwidth]{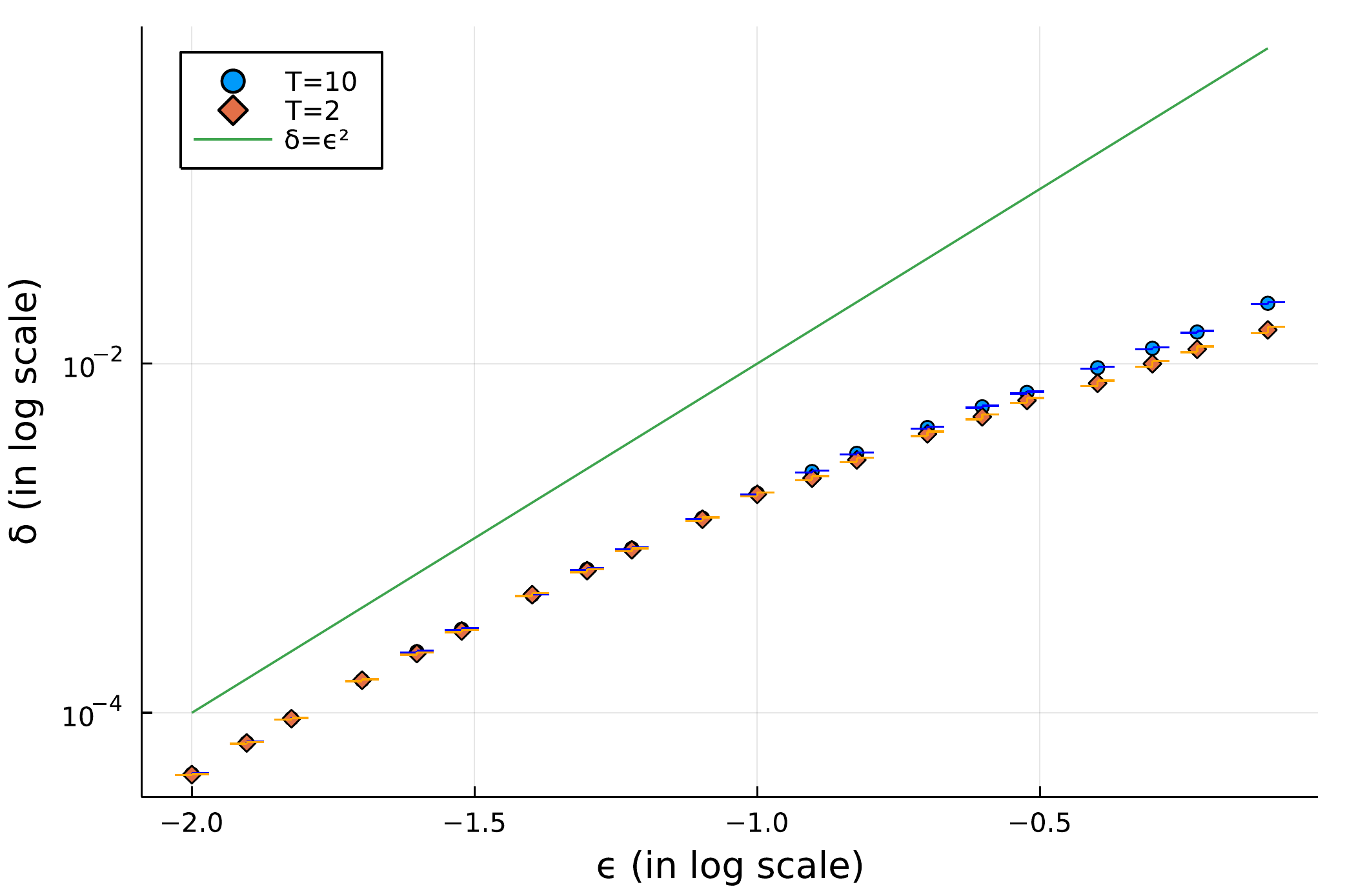}\end{center}
	\caption{   Critical value $\delta_{\rm c}(T)$ ($T\in\{2,10\}$) as a function of $\epsilon$. \\
		Horizontal bars frame the range, while markers locate the geometric mean. }
	\label{F.delta_critic}
\end{figure}

The result of such experiments for $T=10$ and $T=2$, using different values of $\epsilon$ (while $\mu=1$) is reproduced in \Cref{F.delta_critic}. We used $N= 2^{20}$ modes and $\triangle\! t=0.005$. We observe that the transition zone is extremely narrow, in particular for small values of $\epsilon$, and that the critical value behaves asymptotically proportionally to $\epsilon^2$. This behavior is fully consistent with our result in \Cref{P.WP-large-time} and the instability mechanism described by our toy model in \Cref{S.toy}, noticing that the Rayleigh--Taylor condition~\eqref{RTpositive} is satisfied 
for sufficiently regular data when both $\epsilon$ and $\epsilon^2\delta^{-1}$ are sufficiently small; see~\eqref{eq.sufficient-condition-a-positive}. In our opinion, the behavior $\delta_c \approx \epsilon^2$ for small values of $\epsilon$ displayed in \Cref{F.delta_critic} strongly supports our diagnosis of the instability mechanism.

\paragraph{The cost of rectifiers} Lastly we turn to the study of the accuracy of our rectified model, depending on the strength of the involved rectifier, $\delta$. To this aim, we compare the solutions to~\eqref{RWW2-dealias} with the solution to the water waves system, for initial data chosen as
\begin{equation}\label{num-init-p}
	\zeta_0(x)=\zeta(t=0,x)=\exp(-|x|^p)  \quad \text{ and } \quad \psi_0'(x)=\psi'(t=0,x)=0\ \quad (x\in(-L,L)) ,
\end{equation}
with $p\in\{1,2,3\}$, and rectifier $\J^\delta$ set as~\eqref{num-rectifier} with $\r=-1$ and varying $\delta\in(0.01,1)$ as well as $\epsilon\in\{0.05,0.1,0.2\}$.
We set $\mu=1$. 
	As discussed below, the outcome of these numerical experiments fully agree with our convergence result, \Cref{T.Convergence}, both in terms of convergence rate and regularity aspects.

\begin{figure}[!htb]
	\begin{subfigure}{.5\textwidth}
		\includegraphics[width=\textwidth]{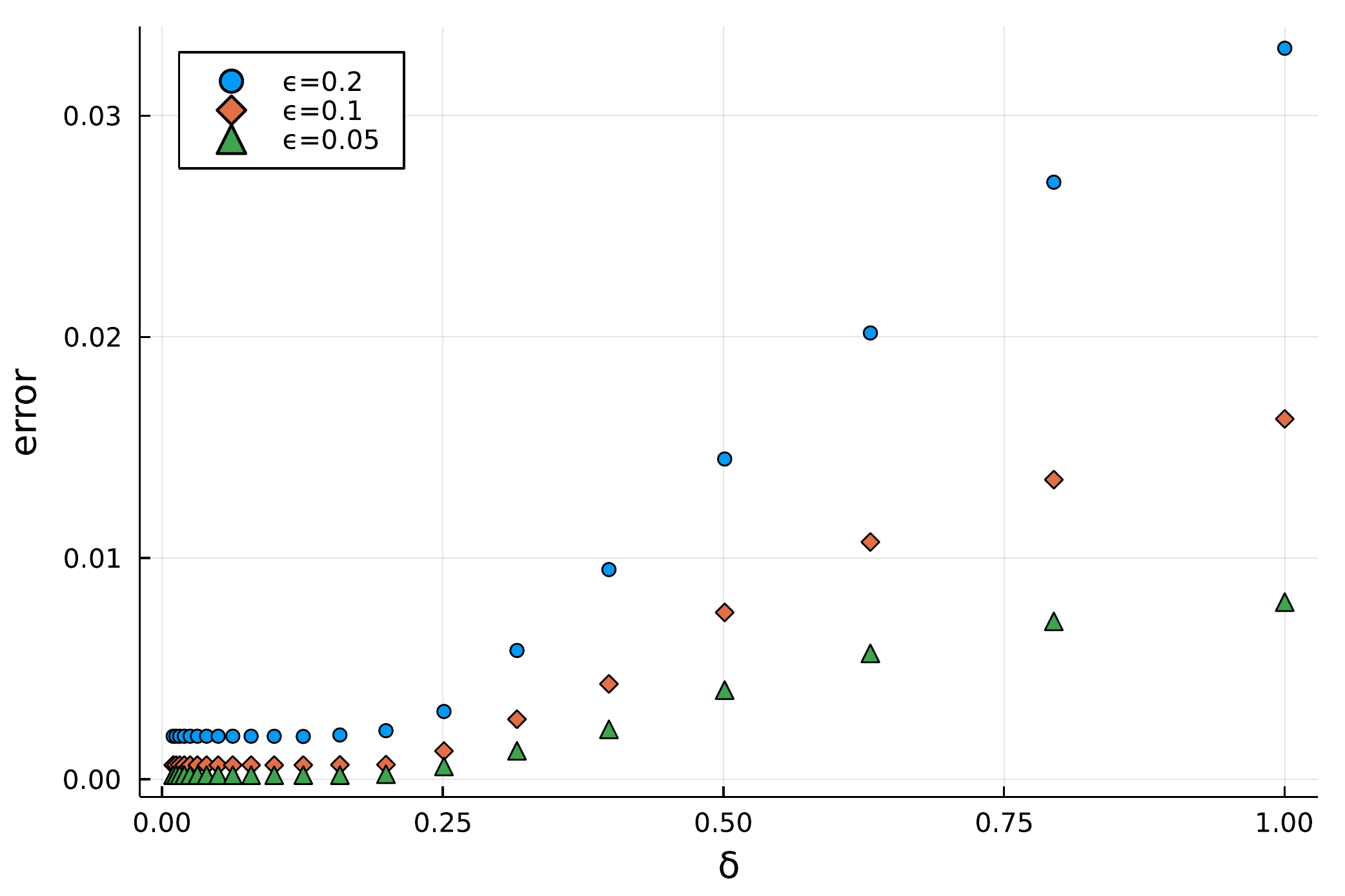}
		\caption{Linear scale.}
	\end{subfigure}%
	\begin{subfigure}{.5\textwidth}
		\includegraphics[width=\textwidth]{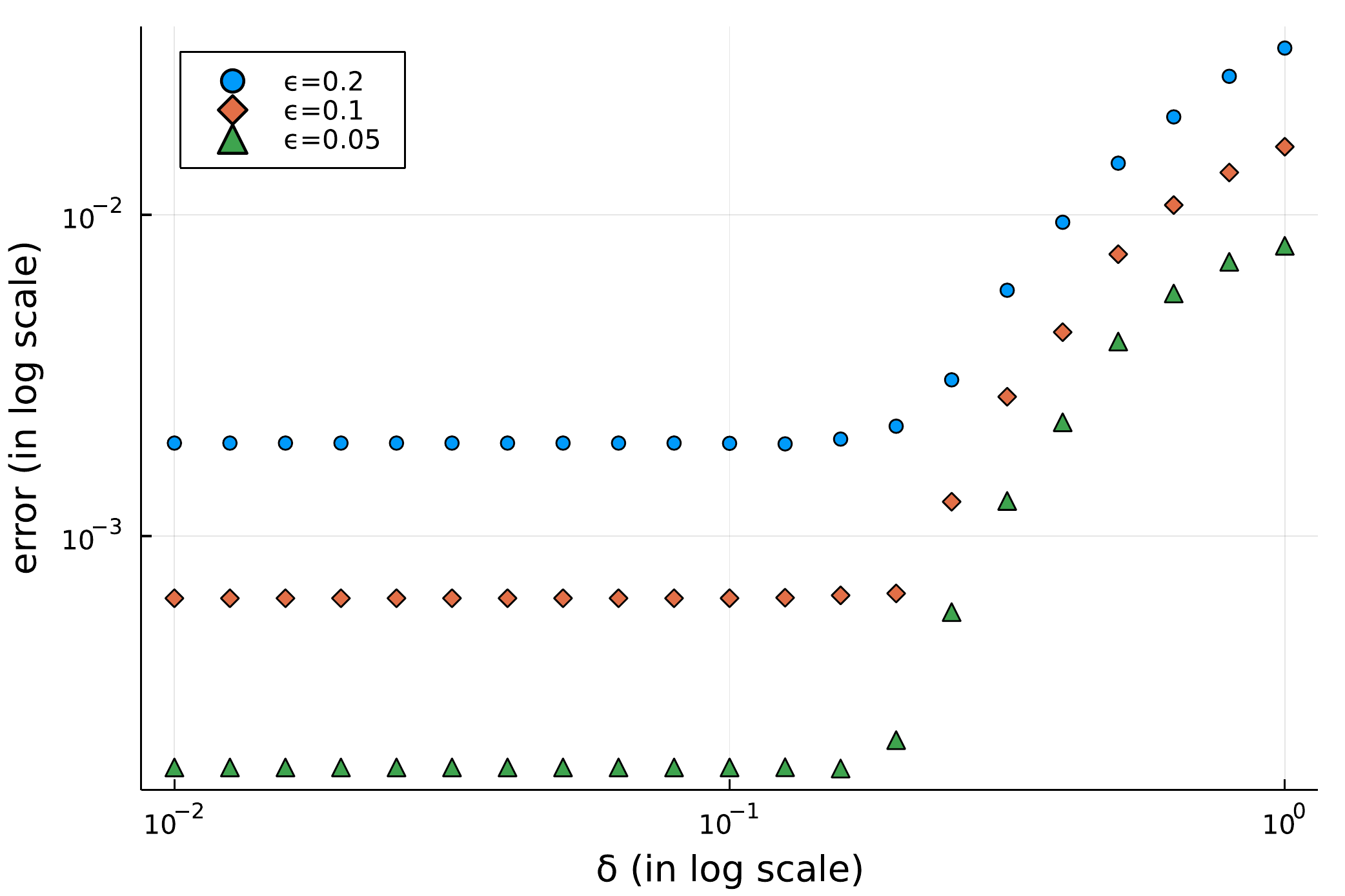}
		\caption{Log scale.}
	\end{subfigure}
	\caption{Error of the model as a function of $\delta$ for smooth initial data,~\eqref{num-init}. }
	\label{F.precision2}
\end{figure}

\begin{figure}[!htb]
	\begin{subfigure}{.5\textwidth}
		\includegraphics[width=\textwidth]{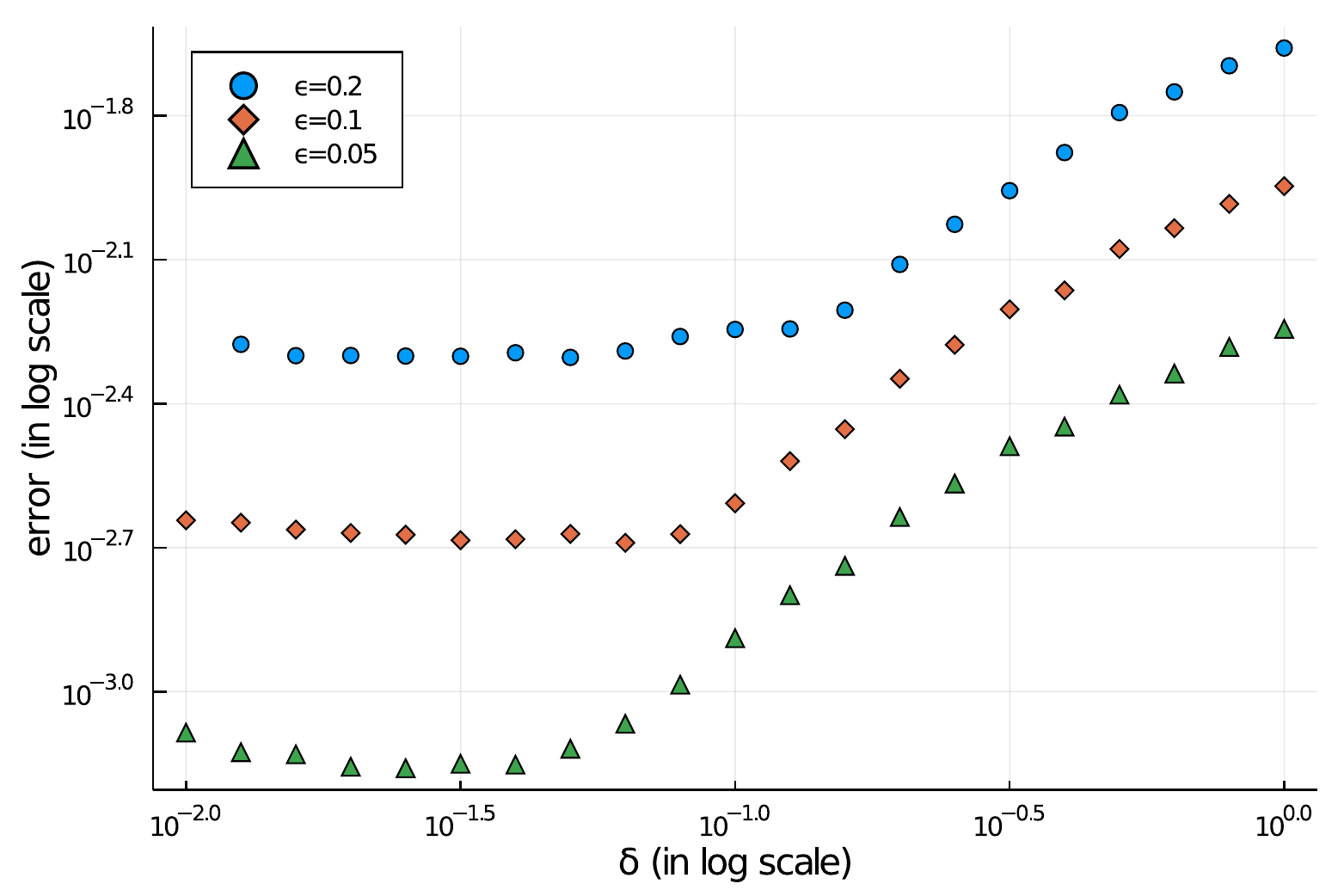}
		\caption{$p=1$.}
	\end{subfigure}%
	\begin{subfigure}{.5\textwidth}
		\includegraphics[width=\textwidth]{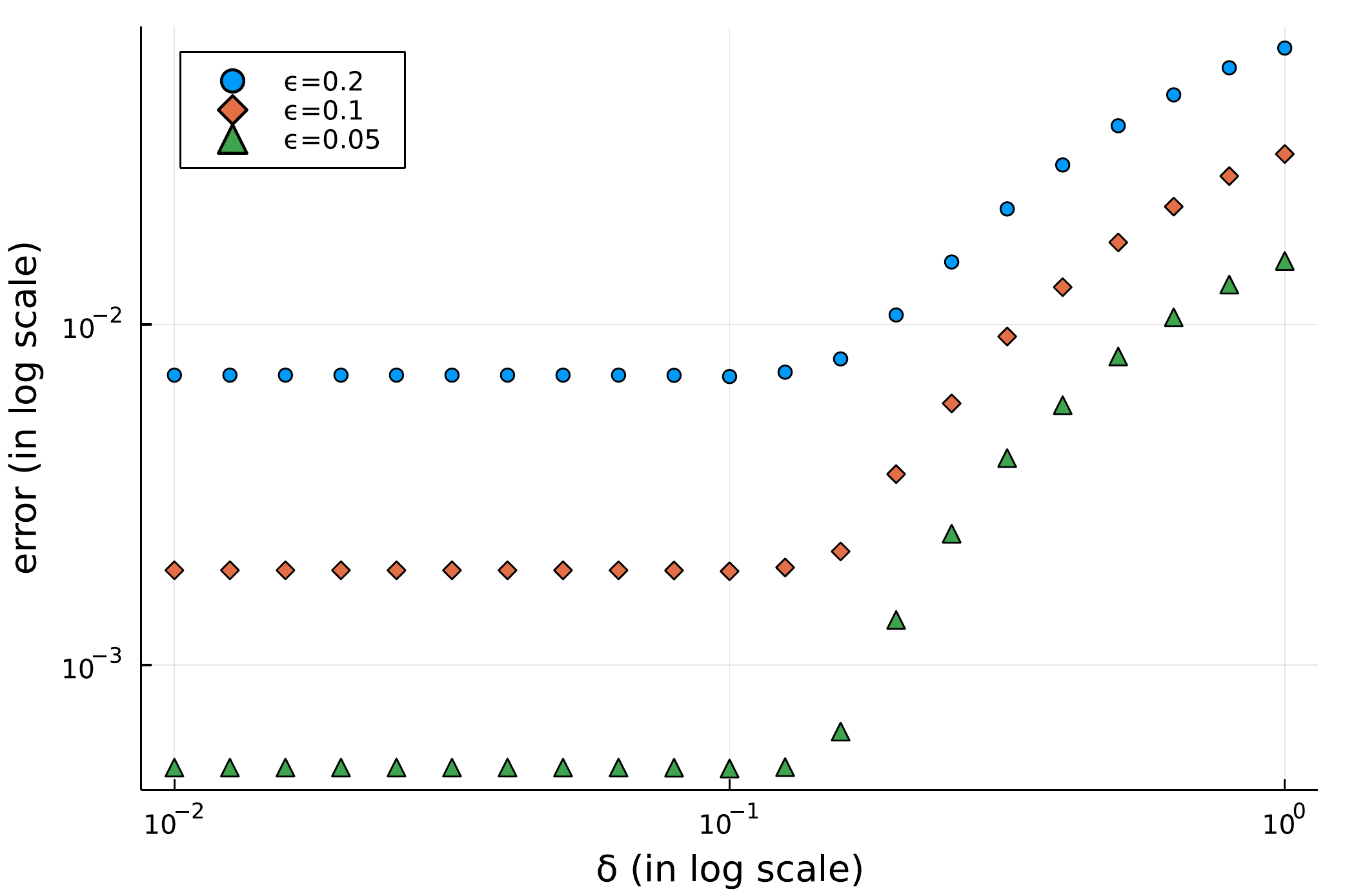}
		\caption{$p=3$.}
	\end{subfigure}
	\caption{Error of the model as a function of $\delta$ for non-smooth initial data,~\eqref{num-init-p}. }
	\label{F.precision13}
\end{figure}

In \Cref{F.precision2} (concerning the situation $p=2$) and \Cref{F.precision13} (concerning the situations $p=1$ and $p=3$) we plot the ``error'' defined as the $\ell^\infty$ norm of the difference between the produced solution (more precisely the surface elevation) to the model and the corresponding numerical solution to the water-waves system, at time $t=10$. The solution to the water waves system is computed following the strategy based on conformal mapping described in~\cite{DyachenkoKuznetsovSpectorEtAl96,ChoiCamassa99} among others, and implemented in the aforementioned Julia package. We use $N=2^{12}$ modes, so that the numerical scheme for~\eqref{RWW2-dealias} would converge (when $p=2$) even without regularizing. However, augmenting the number of modes modifies the error only after several significant digits, showing that the error originates from the model, and not from the spatial discretization. Similarly, diminishing the time step (which is set to $\triangle\! t=0.01$) does not modify the first significant digits.

We observe that for $\delta$ sufficiently small, and in fact way above the ``critical value'' determined in the preceding paragraph, the error stagnates at a value which scales proportionally to $\epsilon^2$. This means that the source of the error, as predicted in \Cref{T.Consistency} and \Cref{T.Convergence}, originates mostly from the consistency of the original model~\eqref{WW2-intro} (without rectification) with respect to the water-waves system. On the contrary, for larger values of $\delta$, the main error originates from the presence of the rectifier $\J^\delta$. Consistently, for fixed value of $\delta$ in this regime, the error scales proportionally to $\epsilon$. 

It is interesting to notice that the behavior with respect to $\delta$ depends strongly on the regularity of the initial data (and hence of the solutions): one can observe that for smooth data ($p=2$) the threshold above which the contribution of the rectifier $\J^\delta$ becomes dominant is almost independent of $\epsilon$, while it strongly depends on $\epsilon$ for less regular data, when $p=1$. This is again consistent with our results: since $\J^\delta$ set as~\eqref{num-rectifier}  is near-identity of arbitrarily large order $\a\geq0$ (recall the terminology in \Cref{D.J-intro}), the smallness of the remainder terms in \Cref{T.Consistency} (and consequently the error terms in \Cref{T.Convergence}) caused by the rectifier $\J^\delta$ is driven by the regularity of the data, compared with the regularity involved when measuring the error.
\medskip

As a general conclusion for this section let us observe that in all these numerical experiments, and even in situations where non-regular data are involved,  it is possible to choose $\delta>0$ sufficiently large so that~\eqref{RWW2-dealias} provides a stable solution for all computed times and at the same time sufficiently small so that the presence of $\J^\delta$ does not induce any noticeable additional errors. {\em In other words, our rectification strategy is fully validated in that introducing appropriate rectifiers is able to suitably regularize the system while being harmless from the point view of the accuracy of the model.}

\section{Main results}\label{S.main}

Before stating our main results, let us introduce a few notations used throughout this work.
\begin{itemize}
	\item We denote $C(\lambda_1,\dots,\lambda_N)$ a positive ``constant'' depending non-decreasingly on its variables. We write $a\lesssim b$ and sometimes $a=\cO(b)$ if $a\leq C b$ where $C>0$ is a universal constant or its dependencies are obvious from the context. We write $a\approx b$ when $a\lesssim b$ and $b\lesssim a$.
	\item We denote the ceiling function as $\lceil x \rceil$ and the floor function as $\lfloor x \rfloor$ for any $x \in \RR$.
	\item  The space $L^\infty(\RR^d)$ consists of all real-valued, essentially bounded, Lebesgue-measurable functions. We denote
	\[
	\norm{f}_{L^\infty}\eqdef {\rm ess\,sup}_{\bx\in\RR^d} | f(\bx)|.
	\]
	\item $L^2(\RR^d)$ denotes the real-valued square-integrable functions, endowed with the topology associated with the inner product
	\[ \forall f,g\in L^2(\RR^d), \quad \big(f,g\big)_{L^2}\eqdef \int_{\RR^d} fg\dd\bx .\]
	\item For any $s\in \RR$, the Sobolev space $H^s(\RR^d)$ is the space of tempered distributions such that  
	\[\norm{f}_{H^s}\eqdef \norm{\langle\cdot\rangle^s \widehat f}_{L^2} <\infty \]
	where $\langle\bxi\rangle\eqdef (1+|\bxi|^2)^{1/2}$ and $\widehat f$ is the Fourier transform of $f$. Of course $H^s(\RR^d)$ is endowed with the norm $|\cdot|_{H^s}$. We use without clarification that when $s=n\in\NN$,
	\[ \norm{f}_{H^n}^2\approx \sum_{\alpha\in\NN^d,\ |\alpha|\leq n} \norm{\partial^\alpha f}_{L^2}^2,\]
	where for multi-indices $\alpha=(\alpha_1,\dots,\alpha_d)\in\NN^d$, we denote $|\alpha|=\sum_{i=1}^d \alpha_i$ and $\partial^\alpha=\partial_{x_1}^{\alpha_1}\cdots \partial_{x_d}^{\alpha_d}$. 
	\item For any $s\in\RR$, the Beppo Levi space $\Hdot^{s}(\RR^d) $ denotes the tempered distributions such that  $\nabla f \in H^{s-1}(\RR^d)^d$, endowed with the semi-norm $\norm{f}_{\Hdot^{s}}\eqdef \norm{\nabla f}_{H^{s-1}}$.
	\item We use the notation $G(D)$ with $D=\frac1{\i}\partial_x$ for the Fourier multiplier operator defined by
	\[ \widehat{G(D)f}(\bxi) = G(\bxi)\widehat f(\bxi).\]
	\item For any $\mu>0$, we denote $G^\mu_0=|D|\tanh(\sqrt{\mu} |D|)$.
\end{itemize}

Let us recall the terminology concerning the operators $\J^\delta$ in~\eqref{RWW2-intro} considered in this work. 
\begin{Definition}[ Rectifiers ]\label{D.J}
	Let $\J^\delta=J(\delta D)$ with $J\in L^{\infty}(\RR^d)$, real-valued and even, and $\delta>0$.
	\begin{itemize}
		\item We say that $\J^\delta$ is {\em regularizing} of order $\r\leq0$ if $\langle\cdot\rangle^{-\r} J \in L^\infty(\RR^d)$, recalling $\langle\cdot\rangle\eqdef(1+|\cdot|^2)^{1/2}$.
		\item We say that $\J^\delta$ is  {\em regular} if $\J^\delta$ is regularizing of order $-1$ and, additionally,  $\langle\cdot\rangle \nabla J \in L^\infty(\RR^d)$.
		\item We say that $\J^\delta$ is {\em near-identity} of order $\a\geq0$ if $|\cdot|^{-\a}(1-J) \in L^\infty(\RR^d)$.
	\end{itemize}
\end{Definition}

We can now state our main results. Our first main result is the local-in-time well-posedness on a relevant timescale for the regularized system~\eqref{RWW2-intro}, under the assumption that rectifiers $\J^\delta$ are sufficiently regularizing and their strength $\delta$ sufficiently large.

\begin{Theorem}[Well-posedness]\label{T.WP-delta}
	Let $d\in\{1,2\}$, $t_{0}> \frac{d}{2}$, $N \in \NN$ with $N \geq t_{0}+2$, $C>1$ and $M>0$. 
	Set $J\in L^\infty(\RR^d)$ such that it defines regular rectifiers. 
	There exists $T_0>0$ such that for any $\mu\geq 1$ and $\epsilon>0$, for any $(\zeta_{0},\psi_{0})\in H^N(\RR^d)\times \Hdot^{N+\frac12}(\RR^d)$ such that
	\begin{equation*}
		0<\epsilon M_0 \eqdef \epsilon  \big( \norm{ \zeta_0 }_{H^{ \lceil t_{0} \rceil  +2}} + \norm{  \mfP  \psi_0 }_{H^{ \lceil t_{0} \rceil +2  }} \big) \leq M,
	\end{equation*}
	and for any $\delta \geq  \epsilon M_0$, the following holds.
	
	There exists a unique $(\zeta,\psi)\in\cC([0,T_0/(\epsilon M_0)];H^N(\RR^d)\times \Hdot^{N+\frac12}(\RR^d))$ solution to~\eqref{RWW2-intro} with $\J^\delta=J(\delta D)$ and initial data $(\zeta,\psi)\id{t=0}=(\zeta_{0},\psi_{0})$, and it satisfies 
	\[ \underset{t \in [0,T_0/(\epsilon M_0)]}{\sup} \big(\norm{\zeta(t,\cdot)}_{H^N}^{2}+\norm{\mfP\psi(t,\cdot)}_{H^N}^{2} \big)\leq C \, \big( \norm{\zeta_0}_{H^N}^{2}+\norm{\mfP\psi_0}_{H^N}^{2} \big) \, . \]
\end{Theorem}
\Cref{T.WP-delta} is the corollary of two more precise and complementary results, namely
\begin{itemize}
	\item \Cref{P.WP-short-time}, an unconditional well-posedness result on ``small times'', \ie up to a maximal time $T\approx \min(1/\epsilon,\delta/\epsilon)$, which we use for large values of $\epsilon$;
	\item \Cref{P.WP-large-time}, a conditional (fulfilled when $\epsilon$ is sufficiently small) result on ``large times'', \ie up to a maximal time $T\approx \min(1/\epsilon,\delta/\epsilon^2)$.
\end{itemize}   

Additionally, we prove in this work that if the rectifier $\J^\delta$ is regularizing of order $\r$ with ${\r<-(\frac32+\frac d 2)}$, then the energy preservation provides the global-in-time well-posedness for sufficiently small initial; see the precise statement in \Cref{P.WP-global-in-time}.  
\medskip

Our second main result describes the precision in the sense of consistency of~\eqref{RWW2-intro} as an asymptotic model for the fully nonlinear water waves system, that is~\eqref{WW}  displayed in \Cref{S.WW}.

\begin{Theorem}[Consistency]\label{T.Consistency}
	Let $d\in\{1,2\}$, $t_0>d/2$, $s \geq 0$, $h_\star>0$ and $M>0$. There exists ${C>0}$ such that for any $\epsilon>0$, $\mu\geq 1$,  $\delta\geq 0$ and any $\J^\delta=J(\delta D)$ near-identity rectifier of order $\a\geq0$ and for any $(\zeta,\psi)\in \cC([0,T];H^{\max(s+\a+1,s+2,t_{0}+2)}(\RR^d)\times \Hdot^{\max(s+\a+\frac32,t_{0}+1)}(\RR^d))$ solution to~\eqref{RWW2-intro} with the property that for all $t\in[0,T]$
	\[  1+\tfrac{\epsilon}{\sqrt\mu} \zeta(t,\cdot) \geq h_\star, \qquad \epsilon M_0\eqdef \epsilon\big(\norm{\zeta(t,\cdot)}_{H^{t_0+2}}+\norm{\mfP\psi(t,\cdot)}_{H^{t_0+\frac12}} \big)\leq M, \] 
	one has
	\[
	\left\{\begin{array}{l}
		\partial_t\zeta-G^\mu[\epsilon\zeta]\psi= \epsilon \delta^\a R_1+ \epsilon^2 \widetilde R_1,\\[1ex]
		\partial_t\psi+\zeta+\frac\epsilon{2} |\nabla\psi|^2-\frac{\epsilon}{2} \frac{ \left(G^\mu[\epsilon\zeta]\psi+\epsilon \nabla\zeta\cdot\nabla\psi \right)^2}{1+\epsilon^2|\nabla\zeta|^2}= \epsilon\delta^\a R_2+ \epsilon^2 \widetilde R_2,
	\end{array}\right.
	\]
	with $G^\mu$ the Dirichlet-to-Neumann operator defined and discussed in \Cref{S.WW}, and
	\begin{align*} \norm{R_1}_{H^s}+ \norm{R_2}_{H^{s+\frac12}} &\leq  C \, M_0\, \dotN{-\a}{(1-J)} \,  \,\big( \norm{\zeta}_{H^{s+\a+1}}+\norm{ \nabla\psi}_{H^{s+\a+\frac12}}   \big),\\
		\norm{\widetilde R_1}_{H^s} + \norm{\widetilde R_2}_{H^{s+\frac12}}   &\leq   C\,  M_0^2\,  \big(  \norm{ \zeta}_{H^{s+2}} + \norm{ \mfP \psi }_{H^{s+1}} \big).
	\end{align*}
\end{Theorem}

This result distinguishes two contributions: the remainders caused by introducing the rectifiers $\J^\delta$ in~\eqref{RWW2-intro} are of magnitude $\epsilon \delta^\a$, while the remainders originating from the truncation of the expansion of the Dirichlet-to-Neumann operator are of magnitude $\epsilon^2$. Only the latter remain if we consider solutions to~\eqref{WW2-intro} instead of~\eqref{RWW2-intro}, or equivalently set $\delta=0$ or $\J^\delta=\Id$. The main observation is that the remainders due to rectifiers can be made asymptotically negligible in front of the remainders associated with the original system~\eqref{WW2-intro} by choosing $\delta$ sufficiently small (depending on $\epsilon$). Moreover, if $\J^\delta$ is near-identity of order $\a\ge 1$, such smallness assumption on $\delta$ is compatible with the constraint $\delta\gtrsim \epsilon$ in \Cref{T.WP-delta}.

Combining the above well-posedness and consistency results and making use of the well-posedness and stability results on the water waves system obtained in~\cite{Alvarez_Lannes}, we infer the full justification of~\eqref{RWW2-intro}. Specifically, a consequence of the result below is that the regularized system~\eqref{RWW2-intro} produces $\cO(\epsilon^2)$ approximations to exact solutions of the water waves system~\eqref{WW} in the relevant timescale, provided that the rectifiers $\J^\delta$ and in particular their strength $\delta$ are appropriately chosen.

\begin{Theorem}[Convergence]\label{T.Convergence}
	There exists $p\in\NN$ such that the following holds.
	
	Let $d\in\{1,2\}$, $t_0>d/2$, $s \geq 0$, $h_\star>0$, $M>0$, and  $J\in L^\infty(\RR^d)$ be such that it defines regular rectifiers near-identity of order $\a\geq0$. There exists ${C>0}$ and $T>0$ such that for any $\epsilon>0$ and $\mu\geq 1$, any $(\zeta_0,\psi_0)\in H^{s+\a+p}(\RR^d)\times \Hdot^{s+\a+p+\frac12}(\RR^d)$ such that
	\[ 1+\tfrac{\epsilon}{\sqrt\mu} \zeta \geq h_\star, \qquad \epsilon M_0\eqdef \epsilon \big( \norm{\zeta_0}_{H^{s+\a+p}}+\norm{\mfP\psi_0}_{H^{s+\a+p}}\big)\leq M, \] 
	and for any $\delta \geq  \epsilon M_0$, there exists
	\begin{itemize}
		\item  a unique $(\zeta,\psi)\in\cC([0,T/(\epsilon M_0)];H^{\lfloor s+\a+p\rfloor }(\RR^d)\times \Hdot^{\lfloor s+\a+p\rfloor+\frac12 }(\RR^d))$ solution to~\eqref{RWW2-intro} where ${\J^\delta= J(\delta D)}$  with initial data $(\zeta,\psi)\id{t=0}=(\zeta_{0},\psi_{0})$,
		\item   a unique $(\zeta_{\rm ww},\psi_{\rm ww})\in\cC([0,T/(\epsilon M_0)];H^{s}(\RR^d)\times \Hdot^{s+\frac12}(\RR^d))$ solution to the water waves system~\eqref{WW} with initial data $(\zeta_{\rm ww},\psi_{\rm ww})\id{t=0}=(\zeta_{0},\psi_{0})$;
	\end{itemize}
	and moreover one has 
	\[ \underset{t \in [0,T/(\epsilon M_0)]}{\sup} \big(\norm{(\zeta-\zeta_{\rm ww})(t,\cdot) }_{H^s}+\norm{(\mfP\psi-\mfP\psi_{\rm ww})(t,\cdot)}_{H^s} \big)\leq C \, M_0\,  t\, \big(  \delta^\a(\epsilon M_0) +(\epsilon M_0)^2\big)\, . \]
\end{Theorem}

The proof of \Cref{T.Consistency,T.Convergence}, assuming that \Cref{T.WP-delta} holds, is provided in \Cref{S.consistency}. The proof of \Cref{T.WP-delta} is postponed to \Cref{S.WP}. 
The upcoming \Cref{S.technical} collects some important tools used in the latter.

\section{Technical ingredients}\label{S.technical}

Recall the notation $G^\mu_0\eqdef |D|\tanh(\sqrt\mu|D|)$.
\begin{Lemma}\label{L.control_P}
Let $s\in\RR$. For any $\mu \geq 1$ and any functions $f\in \Hdot^{s+1/2}(\RR^d)$, 
\begin{equation*}
  \norm{\nabla f}_{H^{s-1/2}} \leq \norm{ \mfP f}_{H^s} .
\end{equation*}
Conversely, $\mfP:\Hdot^{s+1/2}(\RR^d)\to H^s(\RR^d)$ is well-defined and bounded, yet not uniformly with respect to $\mu\geq 1$:
\[ \norm{\mfP f}_{H^{s}} \leq (2\mu)^{\frac14} \norm{ \nabla f }_{H^{s-1/2}}. \]
Moreover, for any $f\in H^{s+1/2}(\RR^d)$, one has the uniform bound
\begin{equation*}
 \norm{\mfP f}_{H^{s}} \leq \norm{  f}_{H^{s+1/2}} .
\end{equation*}
For any $f\in \Hdot^{s+1}(\RR^d)$
\[\norm{G^\mu_0 f}_{H^s}\leq \norm{\nabla f}_{H^s}.\]
The operators $\mfP$ and  $G^\mu_0$ are symmetric
 for the $L^2(\RR^d)$ inner product, in particular for  any $f\in H^{1/2}(\RR^d)$, $g\in \Hdot^1(\RR^d)$,
\[ \big( G^\mu_0 g , f \big)_{L^2} =\big( \mfP g , \mfP f \big)_{L^2}  .\]
\begin{proof}
Results follow from Parseval's theorem and the fact that  for any $\xi\geq 0$ and $\mu\geq 1$ 
\[\xi/\sqrt{1+\xi^2} \leq \tanh(\xi) \leq \tanh(\sqrt{\mu} \xi)\leq \min(1,\sqrt\mu\xi).\]
Only the first inequality requires explanation. It follows from $\tanh(\xi)=\tfrac{\sinh(\xi)}{\sqrt{1+\sinh(\xi)^2}}$, $\sinh(\xi)\geq \xi$ and the fact that $\xi\mapsto \xi/\sqrt{1+\xi^2} $ is increasing.
\end{proof}
\end{Lemma}

The following product estimate is proved for instance in~\cite[Th.~8.3.1]{Hormander97}. 
\begin{Lemma}\label{L.product}
Let $d \in \NN^\star$. Let $s,s_{1},s_{2} \in \RR$ such that $s \leq s_{1}$, $s \leq s_{2}$, $s_{1}+s_{2}\geq 0$ and $ s_{1} + s_{2} >s+ \frac{d}{2}$. Then, there exists a constant $C > 0$ such that for all $f \in H^{s_{1}}(\RR^{d})$ and for all $g \in H^{s_{2}}(\RR^{d})$, we have $fg \in H^{s}(\RR^{d})$  and 
\footnote{
When $s<0$ the product $fg$ is well-defined as a tempered distribution as soon as $s_1+s_2\geq 0$. Indeed, we have
\[ \forall \varphi\in\mathcal{S}(\RR^d), \qquad \langle u_1 u_2,\varphi\rangle := \langle u_1,u_2\varphi\rangle,  \quad \text{ and } \quad \big|\langle u_1 u_2,\varphi\rangle| \lesssim \norm{u_1}_{H^{s_1}} \norm{u_2}_{H^{s_2}}  \norm{\varphi} ,\]
for $H^{s_1}(\RR^d)'=H^{-s_1}(\RR^d)\supset H^{s_2}(\RR^d)$ and $\norm{u_2\varphi}_{H^{s_2}} \lesssim \norm{u_2}_{H^{s_2}}  \norm{\varphi}$ where the norm on $\varphi$ involves a finite number (depending on $s_2$) of semi-norms. Another statement of \Cref{L.product} is that the pointwise multiplication, which is well-defined from $L^2(\RR^d)\times L^2(\RR^d)$ to $L^1(\RR^d)$ (say), extends as a continuous bilinear map from $H^{s_{1}}(\RR^{d})\times H^{s_{2}}(\RR^{d})$ to $H^s(\RR^d)$. Henceforth, we will use without mention the standard identification between functions and (tempered) distributions.
}
\begin{equation*}
\norm{fg}_{H^{s}} \leq C \norm{f}_{H^{s_{1}}} \norm{g}_{H^{s_{2}}}.
\end{equation*}
In particular, for any $t_0>d/2$, and $s\in[ -t_0,t_0]$, there exists $C>0$ such that for all $f\in  H^{s}(\RR^{d})$ and $g\in H^{t_0}(\RR^{d})$, $fg\in H^s(\RR^d)$ and
\begin{equation*}
\norm{fg}_{H^{s}} \leq C \, \norm{f}_{H^{s}} \norm{g}_{H^{t_0}}.
\end{equation*}
\end{Lemma}
We infer the following useful ``tame'' product estimates.
\begin{Lemma}\label{L.product-tame}
Let $d \in \NN^\star$. Let $s,s_1,s_2,s_1',s_2' \in \RR$ be such that \[s_1+s_2=s_1'+s_2'>s+d/2,\quad s_1+s_2\geq 0 , \quad s\leq s_1, \ s\leq s_2'.\]
Then there exists a constant $C > 0$ such that for all $f  \in H^{s_1}(\RR^{d})\cap H^{s_1'}(\RR^{d})$, $g\in  H^{s_2}(\RR^d) \cap H^{s_2'}(\RR^d)$, we have
 $fg \in H^{s}(\RR^{d})$ and 
 \begin{equation*}
 \norm{ fg }_{H^{s}} \leq C \big( \norm{ f }_{H^{s_1}} \norm{  g }_{H^{s_2}} + \norm{  f }_{H^{s_1'}} \norm{ g }_{H^{s_2'}} \big).
 \end{equation*}
 In particular, for any $t_0>d/2$, $s\geq -t_0$, there exists $C>0$ such that for all $f,g\in  H^{s}(\RR^{d})\cap H^{t_0}(\RR^{d})$, $fg\in H^s(\RR^d)$ and
\begin{equation*}
\norm{fg}_{H^{s}} \leq C \big( \norm{f}_{H^{s}} \norm{g}_{H^{t_0}} + \norm{f}_{H^{t_0}} \norm{g}_{H^{s}}\big).
\end{equation*}
\end{Lemma}
\begin{proof}
We consider several cases. If $s\leq s_2$, then \Cref{L.product} yields immediately
\[\norm{ fg}_{H^{s}} \lesssim  \norm{  f }_{H^{s_1}} \norm{  g }_{H^{s_2}} .\]
Symmetrically, if $s\leq s_1'$, then 
\[\norm{ fg}_{H^{s}} \lesssim  \norm{  f }_{H^{s_1'}} \norm{  g }_{H^{s_2'}} .\]
Otherwise $s_1'< s\leq s_1 $ and  $s_2<s\leq s_2'$. Assume moreover that $ s\leq d/2$. Then denoting $s_1^\star=s$ and $s_2^\star=s_1+s_2-s=s_1'+s_2'-s$, we have  $s_1'<s_1^\star=s\leq s_1$, hence $s_2\leq s_2^\star<s_2'$, and $s_2^\star>d/2\geq s$.  \Cref{L.product} yields
\[\norm{ fg }_{H^{s}} \lesssim \norm{  f }_{H^{s_1^\star}} \norm{ g }_{H^{s_2^\star}} .\]
We conclude by the Sobolev interpolation 
\[ \norm{  f }_{H^{s_1^\star}} \norm{  g }_{H^{s_2^\star}} \leq \norm{  f }_{H^{s_1}}^{\theta}\norm{  f }_{H^{s_1'}}^{1-\theta}\norm{  g }_{H^{s_2}}^{\theta}\norm{  g }_{H^{s_2'}}^{1-\theta} \]
with $\theta=\frac{s_1^\star-s_1'}{s_1-s_1'}=\frac{s_2'-s_2^\star}{s_2'-s_2}$, and then by Young's inequality.
The first statement is proved when $ s\leq d/2$. When $s>d/2$, we set $n\in\NN$ such that $-d/2\leq s-n\leq d/2$, and use that
\[\norm{ fg }_{H^{s}} \lesssim  \norm{ fg }_{L^2} + \sum_{\alpha\in\NN^d,\,|\alpha|=n} \norm{ \partial^\alpha (fg) }_{H^{s-n}} .\]
From the previously proved estimate we have
\[ \norm{ fg }_{L^2} \lesssim   \norm{  f }_{H^{s_1}} \norm{  g }_{H^{s_2}} + \norm{  f }_{H^{s_1'}} \norm{  g }_{H^{s_2'}}  \]
and for any $\beta,\gamma\in\NN^d$ such that $\beta+\gamma=\alpha$,
\[\norm{( \partial^\beta f)(\partial^\gamma g) }_{H^{s-n}} \lesssim \norm{ \partial^\beta f }_{H^{s_1-|\beta|}} \norm{ \partial^\gamma g }_{H^{s_2-|\gamma|}} + \norm{ \partial^\beta f }_{H^{s_1'-|\beta|}} \norm{ \partial^\gamma g }_{H^{s_2'-|\gamma|}} .\]
The first statement then follows by Leibniz rule and triangular inequality. The second one is a particular case with $s_1=s_2'=s$ and $s_2=s_1'=t_0$.
\end{proof}

We now turn to commutator estimates involving operators of order $1$. 

\begin{Lemma}\label{L.commut_order1}
Let $d\in\NN^\star$, $t_{0}>\frac{d}{2}$ and $-t_0\leq s\leq t_0$. There exists a constant $C>0$ such that for any $\sigma\in  L^\infty_{\rm loc}(\RR^d)$ such that $\nabla\sigma\in L^{\infty}(\RR^d)$, then for any $f\in \Hdot^{t_0+1}(\RR^d)$ and $g\in H^s(\RR^d)$, one has 
\[\norm{ \big[ \sigma(D),f\big] g}_{H^s} \leq C \norm{\nabla\sigma}_{L^\infty}\norm{\nabla f}_{H^{t_0}}\norm{g}_{H^{s}}.\]
\end{Lemma}
\begin{proof}
We have 
\[ \norm{ \big[ \sigma(D),f\big] g}_{H^s} \lesssim \norm{\int_{\RR^d} \langle \cdot\rangle^s  | \sigma(\cdot)-\sigma(\bta)|  |\widehat f(\cdot-\bta)||\widehat g(\bta)| \dd\bta}_{L^2}.\]
The result follows immediately from the inequality
\[  | \sigma(\bxi)-\sigma(\bta)| \leq \norm{\nabla\sigma}_{L^\infty} |\bxi-\bta| \]
valid for all $\bxi,\bta\in\RR^d$, and an application of \Cref{L.product}, with $s_1=t_0$ and $s_2=s$.
\end{proof}

We now consider commutator estimates involving operators of order $0$.
\begin{Lemma}\label{L.commut_order0}
Let $d\in\NN^\star$, $t_{0}>\frac{d}{2}$ and $-t_0\leq s\leq t_0+1$. There exists a constant $C>0$ such that the following holds. For any $\sigma\in L^{\infty}(\RR^d)$ such that
\[  C_\sigma \eqdef \dotNi{\sigma} 
<\infty \]
and for any $f\in H^{t_0+1}(\RR^d)$ and $g\in H^{s-1}(\RR^d)$, one has $\big[ \sigma(D),f\big] g \in H^s(\RR^d)$ and
\[\norm{ \big[ \sigma(D),f\big] g}_{H^s} \leq C  C_\sigma\norm{ f}_{H^{t_0+1}}  \norm{g}_{H^{s-1}}.\]
If moreover $\nabla\sigma\in L^{\infty}(\RR^d)$ then 
\[\norm{ \big[ \sigma(D),f\big] g}_{H^s} \leq C  C_\sigma'\norm{\nabla f}_{H^{t_0}}  \norm{g}_{H^{s-1}}\]
with $C_\sigma'\eqdef\Ni{\sigma}$.
\end{Lemma}
\begin{proof}
We have 
\[ \norm{ \big[ \sigma(D),f\big] g}_{H^s} \lesssim \norm{ \langle \cdot\rangle^s \int_{\RR^d}  \langle\bta\rangle | \sigma(\cdot)-\sigma(\bta)|  |\widehat f(\cdot-\bta)| \langle \bta\rangle^{-1} |\widehat g(\bta)| \dd\bta}_{L^2}.\]
\begin{itemize}
\item If $|\bta|\geq 2|\bxi|$ and $|\bta|\geq 2$, then $\langle\bta\rangle\leq 2|\bta|\leq 4|\bxi-\bta|$ and hence
\[ \langle\bta\rangle  | \sigma(\bxi)-\sigma(\bta)|  \leq 8\norm{\sigma}_{L^\infty} | \bxi-\bta| .\] 
\item If $ |\bxi|\geq |\bta|/2$ and $|\bta|\geq 2$, suitably selecting an integration path $\gamma$ (with endpoints $\bxi$ and $\bta$)  taking values in $\{\zeta\in\RR^d, \ |\zeta|\geq |\bta|/2\}$ we find that 
\[  \langle\bta\rangle | \sigma(\bxi)-\sigma(\bta)| = \langle\bta\rangle\left|  \int_\gamma   \nabla\sigma \cdot \dd \br \right |   \leq  \langle\bta\rangle |\gamma|\esssup_{\zeta\in\gamma}|\nabla\sigma(\zeta)|\leq 2\pi|\bxi-\bta|\,  \esssup_{|\zeta|\geq 1}|\zeta||\nabla\sigma(\zeta)|. \]
\item Finally, if $|\bta|\leq 2$, then we have immediately (almost everywhere)
\[ \langle\bta\rangle  | \sigma(\bxi)-\sigma(\bta)|  \leq 2\sqrt5\norm{\sigma}_{L^\infty}  \] 
in the first case and
\[ \langle\bta\rangle  | \sigma(\bxi)-\sigma(\bta)|  \leq \sqrt5\norm{\nabla\sigma}_{L^\infty} | \bxi-\bta| \] 
in the second case.
 \end{itemize}
The desired result when $-t_0\leq s\leq t_0$ follows from an application of \Cref{L.product}, with $s_1=t_0$ and $s_2=s$.
Moreover, by symmetry considerations, we have (almost everywhere)
\[ \langle\bxi\rangle  | \sigma(\bxi)-\sigma(\bta)| \lesssim \min(\langle\bxi-\bta\rangle C_\sigma ,|\bxi-\bta| C_\sigma')\]
 which yields the desired result when $-t_0\leq s-1\leq t_0$, using \Cref{L.product} with $s_1=t_0$ and $s_2=s-1$. Since $t_0\geq1/2$, the proof is complete.
\end{proof}
The following Lemma is a direct application of \Cref{L.commut_order1,L.commut_order0}. 
\begin{Lemma}\label{L.commut_Tmu}
Let $d\in\NN^\star$ and $t_{0}>\frac{d}{2}$. There exists a constant $C>0$ such that for any $\mu \geq 1$, any $g\in H^{-1/2}(\RR^d)$ and any $f  \in \Hdot^{t_0+1}(\RR^d)$,
\begin{align*}
\norm{ [|D| ,f] g }_{H^{-1/2}} &\leq C  \norm{ \nabla f }_{H^{t_{0}}} \norm{ g }_{H^{-1/2}},\\
\norm{ [\T^\mu ,f] g }_{H^{1/2}} &\leq C  \norm{  f }_{H^{t_{0}+1}} \norm{ g }_{H^{-1/2}},
\end{align*}
where we denote $\T^\mu\eqdef -\frac{\tanh(\sqrt\mu|D|)}{|D|}\nabla$.
\end{Lemma}
We now provide improved commutator estimates for specific operators.
\begin{Lemma}\label{L.commutator-|D|}
Let $d\in\{1,2\}$ and $t_0>d/2$. If $d=1$, there exists $C>0$ such that  for any $s\geq 0$ and any $r\geq s-t_0$, 
\[\norm{|D| (f |D| g) + \partial_x( f \partial_x g)}_{H^s}\leq C \norm{\partial_x f}_{H^{t_0+r}}\norm{\partial_x g}_{H^{s-r}}.\]
If $d=2$,  for  any $0\leq s\leq t_0+1$ and $s-t_0\leq r\leq 1$, there exists $C>0$ such that
\[\norm{|D| (f |D| g) + \nabla\cdot( f \nabla g)}_{H^s}\leq C\norm{\nabla f}_{H^{t_0+r}}\norm{\nabla g}_{H^{s-r}}.\]
Moreover, for any $s\geq 0$ and $r\leq 1$, there exists $C>0$ such that
\[\norm{|D| (f |D| g) + \nabla\cdot( f \nabla g)}_{H^s}\leq C\norm{\nabla f}_{H^{t_0+r}}\norm{\nabla g}_{H^{s-r}} \, + \, C \norm{\nabla f}_{H^{s}}\norm{\nabla g}_{H^{t_0}}.\]
All the constants $C$ above are independent of $f,g$ such that the right-hand side is finite.
\end{Lemma}
\begin{proof}
The case $d=1$ follows from~\cite[Lemma 3.1]{saut_xu} and the identity 
\[ |D|(f |D| g) + \partial_x( f \partial_x g)=[\cH, \partial_xf] (|D| g )+ [\cH,f ](\partial_x|D| g),\]
where $\cH\eqdef-\frac{\partial_{x}}{|D|}$ is the Hilbert transform.
We however provide a short proof for the sake of completeness.
Let us denote $a = |D| (f |D| g) + \partial_x( f \partial_x g)$. One has for almost any $\xi\geq 0$,
\begin{align*}\widehat{a}(\xi)&=\int_\RR   \big(|\xi||\xi-\eta|-\xi(\xi-\eta)\big) \widehat f(\eta)\widehat g(\xi-\eta)\dd\eta\\
&= 2\int_\xi^\infty \xi(\eta-\xi) \widehat f(\eta)\widehat g(\xi-\eta) \dd\eta,
\end{align*}
and hence, using $|\xi|\leq|\eta| $ and $|\eta-\xi|\leq|\eta| $, one has for any $r'\geq 0$
\[\langle \xi\rangle^s |\widehat{a}(\xi)|\leq 2 \int_\RR \langle \eta\rangle^{s+r'} |\eta| |\widehat f(\eta)|\langle \xi- \eta\rangle^{-r'}|\xi-\eta| |\widehat g(\xi-\eta)|\dd\eta.\]
We conclude by Young's inequality and the fact that $\langle \cdot\rangle^{-t_0}\in L^2(\RR)$, setting $r'=r+t_0-s$.

When $d=2$,~\cite[Lemma 3.3]{saut_xu} is not sufficient to our purpose due to the restriction $r+s\leq 1$. However its proof may be adapted as follows.
Let us denote $a = |D| (f |D| g) + \nabla\cdot( f \nabla g)$, and set $s\geq 0$. One has (almost everywhere)
\begin{align*} \widehat{a}(\bxi)&=\int_{\RR^2}   \big(|\bxi||\bxi-\bta|-\bxi\cdot(\bxi-\bta)\big) \widehat f(\bta)\widehat g(\bxi-\bta) \dd\bta\\
&= \int_{|\bta|\geq |\bxi|/2}   \big(|\bxi||\bxi-\bta|-\bxi\cdot(\bxi-\bta)\big) \widehat f(\bta)\widehat g(\bxi-\bta) \dd\bta\\
&\qquad + \int_{|\bta|< |\bxi|/2}   \big(|\bxi||\bxi-\bta|-\bxi\cdot(\bxi-\bta)\big) \widehat f(\bta)\widehat g(\bxi-\bta)\dd\bta\\
& = I_1+I_2.
\end{align*}
For $I_1$, using $|\bxi|\leq 2|\bta| $, there exists $C>0$ depending uniquely on $s\geq 0$ such that
\[\langle \bxi\rangle^s |I_1|  \leq C   \int_{\RR^2} \langle \bta\rangle^s |\bta| |\widehat f|(\bta)|\bxi-\bta||\widehat g|(\bxi-\bta)\dd\bta.\]
and it follows by Young's inequality, and the fact that $\langle \cdot\rangle^{-t_0}\in L^2$, that
\[\norm{\langle \bxi\rangle^s I_1}_{L^2}\leq C \norm{\nabla f}_{H^{s}}\norm{\nabla g}_{H^{t_0}}.\]

We consider now $I_2$. We use that $|\bxi||\bxi-\bta|-\bxi\cdot(\bxi-\bta)=|\bxi||\bxi-\bta|(1-\cos\alpha)$ with $\alpha\to 0$ as $|\bta|/|\bxi|\to 0$. In fact, $|\tan\alpha|\leq  \frac{2}{\sqrt{3}} |\bta|/|\bxi|$ so there exists $c>0$ such that $(1-\cos\alpha)\leq c|\bta|^2/|\bxi|^2$ as long as $|\bta|/|\bxi| \leq 1/2$. 
Hence
\[\langle \bxi\rangle^s |I_2|  \lesssim \langle \bxi\rangle^s  \int_{ |\bta|< |\bxi| /2}  \frac{ |\bta|^2}{|\bxi|} |\widehat f|(\bta)|\bxi-\bta||\widehat g|(\bxi-\bta)\dd\bta.\]
 When $|\bxi| \leq 1$ we have (since $|\bta|^2/|\bxi|^2\leq \frac12 |\bta|/|\bxi|$) 
\[\langle \bxi\rangle^s |I_2|  \lesssim  \int_{ |\bta|< |\bxi| /2}  |\bta| |\widehat f|(\bta)|\bxi-\bta||\widehat g|(\bxi-\bta)\dd\bta,\]
and we have the same estimate as for $I_1$.
When $|\bxi| \geq 1$, $1/|\bxi| \leq 2/\langle\bxi\rangle$, and we have
\[\langle \bxi\rangle^s |I_2|  \lesssim  \langle \bxi\rangle^{s-1}\int_{ |\bta|< |\bxi| /2}  |\bta|^2 |\widehat f|(\bta)|\bxi-\bta||\widehat g|(\bxi-\bta)\dd\bta.\]
Since $s-1 \geq -t_{0}$ ($d=2$ and $s \geq 0$), 
and  using the first estimate of  \Cref{L.product-tame} with $s_1=s-1$, $s_2=t_0$, $s_1'=t_0-1+r$, $s_2'=s-r$, we have, for any $r\leq 1$,
\[\norm{\langle \bxi\rangle^s I_2}_{L^2}  \lesssim \norm{\nabla f}_{H^{s}} \norm{\nabla g}_{H^{t_0}} + \norm{\nabla f}_{H^{t_0+r}} \norm{\nabla g}_{H^{s-r}}.\] 
We have proved the last estimate. 
When $0\leq s\leq t_0+1$, we use that by the above analysis we have (almost everywhere)
\[\langle \bxi\rangle^s |\widehat a|(\bxi)\lesssim  \langle \bxi\rangle^{s-1}\int_\RR \langle \bta\rangle |\bta| |\widehat f|(\bta)|\bxi-\bta||\widehat g|(\bxi-\bta)\dd\bta.\]
The result follows by \Cref{L.product}, since $s_1:=t_0+r-1\geq s-1$, $s_2:=s-r\geq s-1$, and $s_1+s_2=t_0+s-1> 0$ (recall $t_0> d/2=1$).
This concludes the proof.
\end{proof}

\begin{Remark}
The result of \Cref{L.commutator-|D|} exhibits a regularizing effect of order $2$, and in particular improves the naive result obtained when using \Cref{L.commut_order1,L.product} on the identity 
\[|D| (f |D| g) + \nabla\cdot ( f \nabla g)=[|D| ,f] |D| g+ (\nabla f)\cdot (\nabla g).\]
This turns out to be crucial in our analysis.
\end{Remark}
We deduce the following ``finite-depth'' version of  \Cref{L.commutator-|D|}.
\begin{Lemma}\label{L.commutator-tanh}
Let $d\in\{1,2\}$ and $t_{0}>\frac{d}{2}$. Let $s,r \in \RR$ such that $0\leq s \leq t_{0}+1$ and $s-t_0 \leq r \leq 1$. There exists $C>0$ such that for any $\mu\geq 1$, 
\begin{equation*}
\norm{ G^\mu_0 ( f G^\mu_0 g ) + \nabla \cdot (f \nabla g ) }_{H^s} \leq  C\,\big(  \tfrac1{\sqrt\mu}\norm{f}_{L^2}+\norm{ \nabla f }_{H^{t_{0}+r}} \big)\norm{ \nabla g }_{H^{s-r}}.
\end{equation*}
Moreover, for any $s\geq 0$ and  $r\leq 1$, there exists a constant $C>0$ such that
\begin{equation*}
\norm{ G^\mu_0 ( f G^\mu_0 g ) + \nabla \cdot (f \nabla g ) }_{H^s} \leq  C \norm{ \nabla f }_{H^{t_{0}+r}} \norm{ \nabla g }_{H^{s-r}} \, + \, C\, \big(  \tfrac1{\sqrt\mu}\norm{f}_{L^2}+\norm{ \nabla f }_{H^{s}}  \big)\norm{ \nabla g }_{H^{t_{0}}} .
\end{equation*}
All the constants $C$ above are independent of $f,g$ such that the right-hand side is finite.
\end{Lemma}
\begin{proof} We start with the first estimate. We first note that 
\begin{align*}
|D| \tanh(\sqrt{\mu} |D|) ( f |D| \tanh(\sqrt{\mu} |D|) g ) = &|D|( \tanh(\sqrt{\mu} |D|) - 1) ( f |D| \tanh(\sqrt{\mu} |D|) g )\\
 &+ |D| ( f |D| ( \tanh(\sqrt{\mu} |D|) -1) g ) + |D| ( f |D| g ).
\end{align*}
We have for any $r' \geq 0$ and $\mu\in [1,+\infty)$
\begin{align*}
\norm{\sqrt\mu |D|( \tanh(\sqrt{\mu} |D|) - 1) ( f |D| \tanh(\sqrt{\mu} |D|) g ) }_{H^s} 
&\leq C_{r',s} \norm{ f |D| \tanh(\sqrt{\mu} |D|) g }_{H^{\min(t_{0}-r',0)}}\\
&\leq C_{r',s}' \norm{f }_{H^{r'}} \norm{\nabla g }_{t_{0}-r'},
\end{align*}
where we used \Cref{L.product} and $C_{r',s}$, $C_{r',s}'$ depend uniquely on $t_0$, $r'\geq 0$ and $s$.
Furthermore, one has for any $\mu\in [1,+\infty)$ and any $r'\in\RR $
\begin{align*}
\norm{ \sqrt\mu|D|(f (\tanh(\sqrt{\mu} |D|)-1) |D| g  ) }_{H^s} 
& \leq C_{t_0,s}\sqrt\mu \norm{\nabla f}_{H^{s}} \norm{(\tanh(\sqrt{\mu} |D|)-1) |D| g }_{H^{\max(t_0,s)}} \\
&\quad + C_{t_0,s} \norm{f}_{H^{s}}  \norm{\sqrt\mu|D|(\tanh( \sqrt\mu|D|)-1) |D| g }_{H^{\max(t_0,s)}}\\
&\leq C_{t_0,s} C_{r',s}'' \big( \norm{f}_{H^{s}}  +\sqrt\mu \norm{\nabla f}_{H^{s}}\big)  \norm{\nabla g }_{H^{t_0-r'}} 
\end{align*}
where we used  \Cref{L.product} and $C_{t_0,s},C_{r',s}''$ depend uniquely on $t_0$, $r'$ and $s$. The desired estimates now follow from the triangular inequality, the fact that for any $\theta\in\RR$, $\norm{f}_{H^\theta}\lesssim  \norm{f}_{L^2}+\norm{\nabla f}_{H^{\theta-1}}$, the above estimates and \Cref{L.commutator-|D|}.
\end{proof}

By a similar compensation mechanism as in \Cref{L.commutator-|D|}, we infer the following result  that allows us to define a quantity at low regularity that could not be defined if one only use the product estimate of \Cref{L.product}. 
\begin{Lemma}\label{L.compensation-|D|}
Let $d\in\{1,2\}$ and $t_0>d/2$. Set $\sigma\in[\frac12,1]$.
The bilinear map
\[B:(f,g)\in \Hdot^1(\RR^d)\times \Hdot^1(\RR^d) \mapsto (|D| f)( |D| g) -( \nabla f )\cdot(  \nabla g) \in L^1(\RR^d) \subset H^{-t_0}(\RR^d)\]
extends as a continuous bilinear map from $\Hdot^\sigma(\RR^d)\times \Hdot^\sigma(\RR^d)$ to $H^{2\sigma-2-t_0}(\RR^d)$.
\end{Lemma}
\begin{proof}
The fact that $B$ is defined follows from \Cref{L.product}. We assume below that $f,g\in \Hdot^1(\RR^d)$, and the result follows by a density argument from the desired estimate. 
One has (almost everywhere)
\begin{align*}\widehat{B(f,g)}(\bxi)&=\int_{\RR^d}   \big(|\bta||\bxi-\bta|+\bta\cdot(\bxi-\bta)\big) \widehat f(\bta)\widehat g(\bxi-\bta) \dd\bta\\
&= \int_{\Omega}   \big(|\bta||\bxi-\bta|+\bta\cdot(\bxi-\bta)\big) \widehat f(\bta)\widehat g(\bxi-\bta) \dd\bta\\
&\qquad + \int_{\RR^d\setminus\Omega}   \big(|\bta||\bxi-\bta|+\bta\cdot(\bxi-\bta)\big) \widehat f(\bta)\widehat g(\bxi-\bta)\dd\bta\\
& = I_1+I_2.
\end{align*}
where we define $\Omega\eqdef \big\{(\bxi,\bta)\in\RR^{d}\times\RR^d,\ |\bta|\leq 2\langle\bxi\rangle\big\}$. Using that on  $\Omega$, $\langle \bxi\rangle^{-s_0} \leq 3^{s_0} \langle \bta\rangle^{-s_0}$ for any $s_0\geq 0$, we have by \Cref{L.product} (with $s=-t_0$)
\[\norm{\langle \bxi\rangle^{-s_0-t_0} I_1}_{L^2}\leq C \norm{\nabla f}_{H^{s_1-s_0}}\norm{\nabla g}_{H^{s_2}}\]
as soon as $s_1+s_2\geq 0$ and $s_1,s_2\geq -t_0$. We choose $s_1=-s_2=s_0/2$ with $s_0=2(1-\sigma)\in [0,2t_0]$.

 On $\RR^d\setminus\Omega$, we have $|\bta|\geq 2 \langle\bxi\rangle\geq 2\max(1,|\bxi|)$. Proceeding as in the proof of \Cref{L.commutator-|D|}, we infer that there exists $c>0$ such that  for any $s_0\in[0,2]$ when $d=2$  (and for any $s_0\geq 0$ when $d=1$)
 \[ \big||\bta||\bxi-\bta|-\bxi\cdot(\bxi-\bta)\big|\lesssim |\bta||\bxi-\bta|\frac{\langle\bxi\rangle^{s_0}}{\langle\bta\rangle^{s_0}}.\]
\Cref{L.product} yields once again 
\[\norm{\langle \bxi\rangle^{-s_0-t_0} I_1}_{L^2}\leq C \norm{\nabla f}_{H^{-s_0/2}}\norm{\nabla g}_{H^{-s_0/2}}  \]
 when $s_0\in [0,2t_0]$ (and $s_0\in[0,2]$ when $d=2$). Setting one again $s_0=2(1-\sigma)$ proves the result. 
\end{proof}

We deduce the following ``finite-depth'' version of  \Cref{L.compensation-|D|}, which allows (among other things) to define the second equations in~\eqref{WW2-intro} and~\eqref{RWW2-intro} when $\psi\in \Hdot^{1/2}(\RR^d)$, corresponding to the natural energy space defined by the corresponding Hamiltonians,~\eqref{eq.Hamiltonian} and~\eqref{eq.Hamiltonian-mod}.
\begin{Lemma}\label{L.compensation-tanh}
Let $d\in\{1,2\}$ and $t_0>d/2$. Set $\sigma\in[1/2,1]$.  For any $\mu\geq 1$, the bilinear map
\[B^\mu:(f,g)\in \Hdot^1(\RR^d)\times \Hdot^1(\RR^d) \mapsto (G^\mu_0 f)( G^\mu_0 g) -( \nabla f )\cdot(  \nabla g) \in L^1(\RR^d)\subset H^{-t_0}(\RR^d)\]
extends as a continuous bilinear map from $\Hdot^\sigma(\RR^d)\times \Hdot^\sigma(\RR^d)$ to $H^{2\sigma-2-t_0}(\RR^d)$. Moreover there exists a constant $C>0$, depending only on $t_0$ and $\sigma$, such that for any $f,g\in \Hdot^\sigma(\RR^d)$,
\[ \norm{B^\mu(f,g)}_{H^{2\sigma-2-t_0}} \leq C  \norm{\nabla f}_{H^{\sigma-1}}\norm{\nabla g}_{H^{\sigma-1}}.\]
\end{Lemma}
\begin{proof} By \Cref{L.compensation-|D|}, we need to prove the corresponding result on 
\[ B^\mu-B:(f,g)\in \Hdot^1(\RR^d)\times \Hdot^1(\RR^d) \mapsto (G^\mu_0 f)( G^\mu_0 g) -(|D|f)(|D|g).\]
We rewrite
\[(B^\mu-B)(f,g) = (|D| (\tanh(\sqrt{\mu} |D|-1)f)  (G^\mu_0 g ) + (|D|f) (|D| (\tanh(\sqrt{\mu} |D|-1)g)  \]
Yet since for any $s\in \RR$ there exists $C>0$ such that
\[ \norm{ (|D| (\tanh(\sqrt{\mu} |D|-1)f) }_{H^s}\leq  \norm{ (|D| (\tanh(|D|-1)f) }_{H^s}\leq C \norm{\nabla f}_{H^{\sigma-1}}\]
and $\norm{G^\mu_0 g }_{H^{\sigma-1}} \leq \norm{\nabla g}_{H^{\sigma-1}}$ by \Cref{L.control_P},
we infer immediately from \Cref{L.product} that
\[ \norm{ (|D| (\tanh(\sqrt{\mu} |D|-1)f)  (G^\mu_0 g ) }_{H^{2\sigma-2-t_0 }} \lesssim C \norm{\nabla f}_{H^{\sigma-1}}\norm{\nabla g}_{H^{\sigma-1}}.\]
By similar considerations on the second term and triangular inequality, we obtain the desired estimate, and the proof is complete.
\end{proof}

\section{Full justification; proof of \texorpdfstring{\Cref{T.Consistency,T.Convergence}}{Sec}}\label{S.consistency}

In this section we complete the full justification of \eqref{RWW2-intro} as an asymptotic model for water waves by proving \Cref{T.Consistency,T.Convergence}. 

\begin{proof}[Proof of \Cref{T.Consistency}] The remainder terms $R_1$, $R_2$, $\widetilde R_1$ and $\widetilde R_2$ in the statement may be explicitly defined as
	\begin{align*}
	\epsilon\delta^\a	R_1 &=  \epsilon G^\mu_0 ((\zeta-\J^\delta\zeta ) G^\mu_0 \psi)+\epsilon \nabla\cdot((\zeta-\J^\delta \zeta)\nabla\psi),\\
	\epsilon\delta^\a	R_2 &= \frac\epsilon2 (\Id-\J^\delta) \left( |\nabla\psi|^2- (G^\mu_0 \psi )^2 \right),
	\end{align*}
	and
	\begin{align*}
		\epsilon^2\widetilde R_1 &= -\frac1{\sqrt\mu}G^\mu[\epsilon\zeta]\psi + G^\mu_0 \psi  -  \epsilon G^\mu_0 (\zeta G^\mu_0 \psi)-\epsilon\nabla\cdot(\zeta\nabla\psi),\\
		\epsilon^2\widetilde R_2 &= -\frac{\epsilon}{2} \frac{ \left(\frac{1}{\sqrt{\mu}} G^\mu[\epsilon\zeta]\psi+\epsilon \nabla\zeta\cdot\nabla\psi \right)^2}{1+\epsilon^2|\nabla\zeta|^2} + \frac\epsilon{2 } (G^\mu_0 \psi )^2\\
		&= -\frac{\epsilon}{2} \frac{ \left(\frac{1}{\sqrt{\mu}} G^\mu[\epsilon\zeta]\psi+\epsilon \nabla\zeta\cdot\nabla\psi -G^\mu_0 \psi \right)\left(\frac{1}{\sqrt{\mu}} G^\mu[\epsilon\zeta]\psi+\epsilon \nabla\zeta\cdot\nabla\psi +G^\mu_0 \psi \right) }{1+\epsilon^2|\nabla\zeta|^2}\\
		&\hspace{0.4cm}+ \frac{\epsilon^{3}}{2} \frac{|\nabla\zeta|^2(G^\mu_0 \psi )^2 }{1+\epsilon^2|\nabla\zeta|^2}.
	\end{align*}
	We now estimate each of these terms.
	
	By the second estimate of
	\Cref{L.commutator-tanh}  with $r=0$, we find
	\[ \delta^\a\norm{R_1}_{H^s}\lesssim \norm{\zeta-\J^\delta\zeta}_{H^{s+1}} \norm{\nabla \psi }_{H^{t_0}} +   \norm{\zeta-\J^\delta\zeta}_{H^{t_0+1}} \norm{\nabla \psi }_{H^{s}} . \]
	Now, by Parseval's theorem, we have for any $\sigma\in\RR$, $\a\geq0$ and $\delta\geq 0$,
	\[\norm{\zeta-\J^\delta\zeta}_{H^{\sigma}}  \leq \delta^\a \dotN{-\a}{(1-J)}\norm{\zeta }_{H^{\sigma+\a}}.\]
	For any $s>t_0$, since $\a\geq0$, we can put $\theta=\frac{s-t_0}{s+\a+\frac12-t_0}\in(0,1)$ and infer by interpolation and Young's inequalities,
	\begin{align*} \norm{\zeta}_{H^{t_0+\a+1}} \norm{\nabla \psi }_{H^{s}} &\lesssim  \norm{\zeta}_{H^{t_0+\frac12}}^\theta  \norm{\zeta}_{H^{s+\a+1}}^{1-\theta} \norm{\nabla \psi }_{H^{s+\a+\frac12}}^\theta  \norm{\nabla \psi }_{H^{t_0}}^{1-\theta} \\
		&\lesssim  \norm{\zeta}_{H^{t_0+\frac12}} \norm{\nabla \psi }_{H^{s+\a+\frac12}} + \norm{\zeta}_{H^{s+\a+1}} \norm{\nabla \psi }_{H^{t_0}} .
	\end{align*}
	This provides the desired estimate for $\norm{R_1}_{H^s}$ thanks to \Cref{L.control_P}.
	Then, we have as above
	\[\delta^\a\norm{R_2}_{H^{s+\frac12}}\leq \frac12  \delta^\a\dotN{-\a}{(1-J)}  \norm{ |\nabla\psi|^2- (G^\mu_0 \psi )^2 }_{H^{s+\a+\frac12}},\]
	and tame product estimates (\Cref{L.product-tame}) as well as \Cref{L.control_P} yield
	\[\norm{ |\nabla\psi|^2- (G^\mu_0 \psi )^2 }_{H^{s+\a+\frac12}} \lesssim \norm{ \nabla\psi}_{H^{t_0}}\norm{ \nabla\psi}_{H^{s+\a+\frac12}} .\]
	This provides the desired estimate for $\norm{R_2}_{H^{s+\frac12}}$.
	
	By \Cref{asymptotic_DN}, we have immediately
	\[\norm{\widetilde R_1}_{H^s}  \leq C\big(\tfrac1{h_\star} , \norm{\epsilon\zeta}_{H^{t_0+2}} \big) \norm{ \zeta }_{H^{t_{0}+1}} \left( \norm{ \zeta }_{H^{t_{0}+1}}  \norm{ \mfP \psi }_{H^{s+\frac12}}  +  \norm{ \zeta }_{H^{s+\frac{3}{2}}} \norm{ \mfP \psi }_{H^{t_{0}}}  \right). \]
	This provides the desired estimate for $\norm{\widetilde R_1}_{H^s}$.
	Then, by Moser tame estimates (see for instance~\cite[Prop.~B.4]{Lannes_ww}) we have for any $F\in H^{s+\frac12}(\RR^d)$,
	\[\norm{\tfrac{F}{1+\epsilon^2|\nabla\zeta|^2}  }_{H^{s+1/2}} \leq C \left( \epsilon \norm{ \nabla\zeta}_{H^{t_{0}}} \right) \left( \norm{ F }_{H^{s+\frac{1}{2}}} +  \epsilon \norm{F}_{H^{t_{0}}} \norm{ \nabla\zeta }_{H^{s+\frac{1}{2}}} \right) \]
	and, using both \Cref{asymptotic_DN} and \Cref{L.commutator-tanh} (with $r=0$), and the triangular inequality, we get for any $\sigma\geq0$
	\begin{multline*}\norm{\tfrac{1}{\sqrt{\mu}} G^\mu[\epsilon\zeta]\psi - G^\mu_0 \psi   }_{H^{\sigma}}\leq 
		\epsilon^2 C\big(\tfrac1{h_\star} , \norm{\epsilon\zeta}_{H^{t_0+2}}\big) \norm{ \zeta }_{H^{t_{0}+1}} \left( \norm{ \zeta }_{H^{t_{0}+1}}  \norm{ \mfP \psi }_{H^{\sigma+\frac12}}  +  \norm{ \zeta }_{H^{\sigma+\frac32}} \norm{ \mfP \psi }_{H^{t_{0}}}  \right)\\
		+\epsilon \norm{ \zeta }_{H^{\sigma+1}}  \norm{\nabla\psi}_{H^{t_0}} +\epsilon \norm{\zeta}_{H^{t_0+1}}\norm{ \nabla\psi }_{H^{\sigma}} .
	\end{multline*}
	The desired estimate for $\norm{\widetilde R_2}_{H^{s+\frac12}}$ follows from the above with $\sigma\in\{s+\frac12,t_0\}$, \Cref{L.control_P}, the triangular inequality and tame product estimates, \Cref{L.product-tame}.
\end{proof}

\Cref{T.Convergence} is a direct consequence of \Cref{T.Consistency}, \Cref{T.WP-delta} and the well-posedness and stability of the water waves system provided in~\cite{Alvarez_Lannes}.
The proof is almost identical to~\cite[Theorem~6.5]{Alvarez_Lannes} (in fact~\eqref{WW2-intro} arises explicitly in (6.9) therein). Specifically, one infers the existence and uniqueness of a solution to~\eqref{WW} on the appropriate time interval from Theorem~5.1 and Proposition~5.1 in~\cite{Alvarez_Lannes}, solutions to~\eqref{RWW2-intro} on the same timescale are provided by \Cref{T.WP-delta}, and the control of the difference stems from \Cref{T.Consistency} and stability estimates on the water waves system which again are provided in \cite[Remark~5.4]{Alvarez_Lannes}, referring to \cite[Corollary~3.13]{Alvarez_Lannes_green_naghdi}.

\section{Well-posedness results and proof of \texorpdfstring{\Cref{T.WP-delta}}{sec}} \label{S.WP}

In this section we prove several well-posedness results on the initial-value problem for~\eqref{RWW2-intro}, 
 which we rewrite below for the sake of readability :
 \begin{equation}\label{RWW2}\tag{RWW2}\left\{\begin{array}{l}
\partial_t\zeta- G^\mu_0\psi+ \epsilon G^\mu_0 ((\J^\delta \zeta) G^\mu_0 \psi)+\epsilon\nabla\cdot((\J^\delta \zeta)\nabla\psi)=0,\\[1ex]
\partial_t\psi+\zeta+\frac\epsilon{2 } \J^\delta \left( |\nabla\psi|^2- (G^\mu_0 \psi )^2 \right)=0.
\end{array}\right.\end{equation}
We prove in particular \Cref{T.WP-delta}, as a consequence of the following two results.

We start with the ``small-time'' existence and uniqueness of solutions expressing the semilinear nature of the system (for sufficiently regular data) as soon as $\J$ is regularizing of order at least $-1$.
\begin{Proposition}\label{P.WP-short-time}
Let $d\in\NN^\star$, $t_0>d/2$, $s\geq 0$, and $C>1$. There exists a constant $T_0>0$ such that for any  ${\mu \geq 1}$, any  $\delta>0$ and $J\in L^\infty(\RR^d)$ such that $\langle \cdot\rangle^{\max(1,t_0+\frac32-s)} J\in L^\infty(\RR^d)$, and for any  $(\zeta_0,\psi_0)\in  H^s(\RR^d)\times \Hdot^{s+1/2}(\RR^d)$ 
the following holds. There exists $T^\star,T_\star\in(0,+\infty]$ and a unique $(\zeta,\psi)\in \cC((-T_\star,T^\star);H^s(\RR^d)\times \Hdot^{s+1/2}(\RR^d))$ maximal solution to~\eqref{RWW2} with $\J^\delta=J(\delta D)$ and initial data ${(\zeta,\psi)\id{t=0}=(\zeta_0,\psi_0)}$. Moreover, $(\partial_t\zeta,\partial_t\psi) \in \cC((-T_\star,T^\star);H^{s-1/2}(\RR^d) \times H^{s}(\RR^d))$ and one has $\min(T_\star,T^\star)> T_1$ with
\[ T_1\eqdef \frac{T_0 }{\epsilon\big( \norm{\zeta_0}_{H^{\min(s,t_0+\frac12)}}+\norm{\nabla\psi_0}_{H^{ \min(s-\frac12,t_0) }}\big)\, }\frac{\min(1,\delta^{\max(1,t_0+\frac32-s)})}{\N{\max(1,t_0+\frac32-s)}{J}} \]
 and 
%\footnote{
%We also have
% \begin{equation*}
% \max_{t\in[-T_{1},T_{1}]}\big(\norm{(\mfP)^{-1}\nabla\zeta(t,\cdot)}_{H^{s+\frac12}}^2+\norm{\nabla\psi(t,\cdot)}_{H^{s+\frac12}}^2\big)\leq C \, \big(\norm{(\mfP)^{-1}\nabla\zeta_0}_{H^{s+\frac12}}^2+\norm{\nabla\psi_0}_{H^{s+\frac12}}^2\big),
% \end{equation*}
%which is not used afterward but provides an additional control when $\mu=\infty$.
%} 
\begin{equation*}
\max_{t\in[-T_1,T_1]}\big(\norm{\zeta(t,\cdot)}_{H^{s}}^2+\norm{\mfP\psi(t,\cdot)}_{H^{s}}^2\big)\leq C \, \big(\norm{\zeta_0}_{H^{s}}^2+\norm{\mfP\psi_0}_{H^{s}}^2\big),
\end{equation*}
where we recall that $\mfP\eqdef\left(|D|\tanh(\sqrt\mu|D|) \right)^{1/2}$.\end{Proposition}

Then, we give a ``large-time'' result under some hyperbolic-type condition. First, we define the Rayleigh--Taylor operator $\mfa_{\J^\delta}$ as
\begin{equation}\label{eq.def-RT}
\mfa_{\J^\delta}[\epsilon \zeta, \epsilon\nabla\psi] f \eqdef f - \epsilon (G^\mu_0 \J^\delta \zeta) \J^\delta  f - \epsilon^2\J^\delta \Big((G^\mu_0 \psi) |D| \big\{ (G^\mu_0 \psi) \J^\delta f \big\}\Big).
\end{equation}
Then, we introduce the energy, for $N \in \NN^\star$,
\begin{align}\label{eq.def-E}
\cE^{N}(\zeta,\psi) &= \underset{\alpha \in \NN^{d}, |\alpha| \leq N-1}{\sum} \left( \norm{ \partial^{\alpha} \zeta }_{L^2}^2 + \norm{ \mfP \partial^{\alpha} \psi }_{L^2}^2 \right)\\
&\hspace{5em} + \underset{\alpha \in \NN^{d}, |\alpha| = N}{\sum} \left( \zetaa , \mfa_{\J^\delta}[\epsilon \zeta, \epsilon\nabla\psi]  \zetaa \right)_{L^2} + \norm{ \mfP \psia }_{L^2}^2\nonumber
\end{align}
with  $\zetaa \eqdef \partial^{\alpha} \zeta$ and $\psia \eqdef \partial^{\alpha} \psi - \epsilon (G^\mu_0 \psi) (\J \partial^{\alpha} \zeta)$.

\begin{Proposition}\label{P.WP-large-time}
Let $d\in\{1,2\}$, $t_{0}> d/2$ and $N \in \NN$ with $N \geq t_{0}+2$. Let $M_J>0$, $M_U>0$, $\mfa_\star>0$ and $C>1$. 
There exists $T_0>0$ such that for any $\epsilon>0$, $\mu\geq 1$, $\delta\in(0,1]$ and any regular rectifier $\J^\delta$ (see \Cref{D.J}) satisfying $\Ni{J}\leq M_J$, the following holds.  For any $(\zeta_{0},\psi_{0})\in H^N(\RR^d)\times \Hdot^{N+\frac12}(\RR^d)$ satisfying
\[ 0<\epsilon M_0 \eqdef  \epsilon \sqrt{ \cE^{\lceil t_0 \rceil +2}(\zeta_0,\psi_0)} \leq M_U\]
and the Rayleigh--Taylor condition  
\begin{equation}\label{RTpositive}
\forall f\in L^2(\RR^d) \text{ , } \left( f , \mfa_{\J^\delta}[\epsilon \zeta_0,\epsilon\nabla \psi_0]  f\right)_{L^2} \geq \mfa_\star \norm{ f }_{L^2}^{2},
\end{equation}
the maximal solution to~\eqref{RWW2-intro} with initial data $(\zeta,\psi)\id{t=0}=(\zeta_{0},\psi_{0})$ satisfies $\min(T_\star,T^\star)> T_2$ where
\begin{equation*}
T_2 \eqdef \frac{T_{0}}{\epsilon M_0 + \epsilon^{2} \delta^{-1} M_0^2 \dotN{\frac12}{J}^2} ,
\end{equation*}
and for any $0\leq |t|\leq T_2$ and any $N_\star \in\{ \lceil t_0 \rceil +2,N\}$,
\[\cE^{N_\star}(\zeta(t,\cdot),\psi(t,\cdot)) \leq C \, \cE^{N_\star}(\zeta_{0},\psi_{0}). \]
\end{Proposition}

As $\delta\searrow 0$, the lower bound on the time of existence guaranteed by \Cref{P.WP-short-time} is of magnitude $T_1\approx \delta\epsilon^{-1}$ (when $s$ is sufficiently large), while it is of magnitude $T_2 \approx \min( \epsilon^{-1}, \epsilon^{-2}\delta)$ in  \Cref{P.WP-large-time}, hence the ``small-time'' {\em vs} ``large-time'' terminology. Specifically, \Cref{P.WP-short-time} provides the unconditional small-time existence (and control) of solutions while \Cref{P.WP-large-time} provides a conditional large-time existence (and control) of solutions in the small steepness framework ($\epsilon\ll 1$) and for weak rectification ($\delta\ll 1$).

The proof of \Cref{P.WP-short-time} is provided on \Cref{S.WP-short-time}. It follows from standard techniques on semilinear dispersive equations. Incidentally, the analysis is completed by a blow-up criterion in \Cref{C.blowup}, and the well-posedness of the initial-value problem in the sense of Hadamard is discussed in \Cref{R.WP}. 

An additional global-in-time well-posedness (for small initial data) result is stated and proved in \Cref{S.WP-global}, based on the low-regularity well-posedness result provided in \Cref{P.WP-short-time}, and the fact that the Hamiltonian function,~\eqref{eq.Hamiltonian-mod}, is an invariant quantity.

The difficult and technical part consists in proving  \Cref{P.WP-large-time}. The proof relies on careful {\em a priori} energy estimates satisfied by smooth solutions, and is provided on \Cref{S.WP-large-time}.. We divide the proof in several parts: first in \Cref{S.quasi} we extract simple sets of equations satisfied by smooth solutions and their derivatives. Based on these equations, and assuming that a certain hyperbolicity criterion holds, we obtain in \Cref{S.energy-estimates} energy estimates, that is a differential inequality satisfied by the functional $\cE^N$. The completion of the proof of \Cref{P.WP-large-time} is provided in \Cref{S.WP-large-completion}.

We conclude this introduction with several remarks, followed by the completion of the proof of \Cref{T.WP-delta}, as a consequence of \Cref{P.WP-short-time,P.WP-large-time}.

\begin{Remark}\label{R.RT}
The operator $\mfa_{\J^\delta}$ defined in~\eqref{eq.def-RT} is a key ingredient of our analysis. Notice that the function
\[ \mfa_{(2)}\eqdef 1 - \epsilon (G^\mu_0  \zeta) \]
is a $\cO(\epsilon^2)$ approximation of the Rayleigh--Taylor coefficient appearing for instance in~\cite[(4.20), see also (4.27) and discussion in \S4.3.5]{Lannes_ww}. We claim that the last term in~\eqref{eq.def-RT}, which has no counterpart in the fully nonlinear water waves system, is responsible for the observed spurious oscillations in numerical simulations, as the consequence of the ill-posedness of the initial-value problem. Indeed, without any regularization (that is setting $\J^\delta=\Id$), the Rayleigh--Taylor condition~\eqref{RTpositive} can never be satisfied unless $\epsilon\nabla\psi_0=\bz$.

Another key ingredient of our ``large-time'' analysis is the use of $\psia$ instead of $\partial^\alpha\psi$ when defining the energy functional $\cE^{N}(\zeta,\psi)$. Setting $\J^\delta=\Id$, we recognize in $\psia$ a $\cO(\epsilon^2)$ approximation of Alinhac's good unknowns for the water waves system; see discussion in~\cite[\S4.1]{Lannes_ww}. Notice that if the rectifier operator $\J^\delta$ is regularizing of order $-1/2$, we infer (in contrast with the water waves situation) that under the Rayleigh--Taylor condition,
\[ \cE^{N}(\zeta,\psi) \approx \norm{\zeta}_{H^N}+\norm{\mfP \psi}_{H^N}.\]
However the equivalence between the energy functional and the standard Sobolev norms is not uniform with respect to $\delta \in (0,1]$; see \Cref{P.psi_controls}.
\end{Remark}

\begin{Remark}\label{R.order-J}
We claim  \Cref{P.WP-large-time} holds assuming only that rectifiers $\J^\delta$ are regularizing of order $-1/2$  (and not $-1$). Yet in that case~\eqref{RWW2-intro} is of quasilinear nature and the well-posedness theory requires additional arguments which we decided to avoid for simplicity.
\end{Remark}

\begin{Remark}\label{R.energy_noninteger}
With a small adaptation of this work, one can consider non integer regularities, that is $N=s\in\RR$ with $s > \frac{d}{2}+2$. Then we need to replace the functional~\eqref{eq.def-E} with
\begin{align*}
\cE^{s}(U) &= \underset{\alpha \in \NN^{d}, |\alpha| \leq \lceil s \rceil - 1}{\sum} \left( \norm{ \partial^{\alpha} \zeta }_{L^2}^2 + \norm{ \mfP \partial^{\alpha} \psi }_{L^2}^2 \right)\\
&\hspace{5em} +  \left(   |D|^s\zeta , \mfa_{\J^\delta}[\epsilon \zeta, \epsilon\nabla\psi] |D|^s\zeta \right)_{L^2} + \norm{ \mfP \psi_{(s)} }_{L^2}^2\nonumber
\end{align*}
with $\psi_{(s)}\eqdef |D|^s \psi - \epsilon (G^\mu_0 \psi) (\J |D|^{s} \zeta)$. 
\end{Remark}

Let us now show how \Cref{T.WP-delta}, which we rewrite below for the convenience of the reader, follows from \Cref{P.WP-short-time} and \Cref{P.WP-large-time}.
\begin{Theorem}\label{T.WP-delta-bis}
Let $d\in\{1,2\}$, $t_{0}> \frac{d}{2}$, $N \in \NN$ with $N \geq t_{0}+2$, $C>1$ and $M>0$. 
Set $J\in L^\infty(\RR^d)$ such that it defines regular rectifiers.
There exists $T_0>0$ such that for any $\mu\geq 1$ and $\epsilon>0$, for any $(\zeta_{0},\psi_{0})\in H^N(\RR^d)\times \Hdot^{N+\frac12}(\RR^d)$ such that
\begin{equation*}
0<\epsilon M_0 \eqdef \epsilon  \big( \norm{ \zeta_0 }_{H^{\lceil t_{0} \rceil +2}} + \norm{  \mfP \psi_0  }_{H^{\lceil t_{0} \rceil +2}} \big) \leq M,
\end{equation*}
and for any $\delta \geq  \epsilon M_0$, the following holds. 

There exists a unique $(\zeta,\psi)\in\cC([0,T_0/(\epsilon M_0)];H^N(\RR^d)\times \Hdot^{N+\frac12}(\RR^d))$ solution to~\eqref{RWW2} with $\J^\delta=J(\delta D)$ and initial data $(\zeta,\psi)\id{t=0}=(\zeta_{0},\psi_{0})$, and it satisfies 
\[ \underset{t \in [0,T_0/(\epsilon M_0)]}{\sup} \big(\norm{\zeta(t,\cdot)}_{H^N}^2+\norm{\mfP\psi(t,\cdot)}_{H^N}^2 \big)\leq C \, \big( \norm{\zeta_0}_{H^N}^2+\norm{\mfP\psi_0}_{H^N}^2\big) \, . \]
\end{Theorem}

\begin{proof}[Proof of \Cref{T.WP-delta}/\ref{T.WP-delta-bis}]
For any $\delta>0$,  let$(\zeta^\delta,\psi^\delta)\in\cC([0,T_\delta^\star);H^N(\RR^d)\times \Hdot^{N+\frac12}(\RR^d))$ be the maximal solution for positive times (the result for negative times following from time-reversibility of the equations) to~\eqref{RWW2} with $\J^\delta=J(\delta D)$, with initial data $(\zeta^\delta,\psi^\delta)\id{t=0}=(\zeta_{0},\psi_{0})$, as provided by \Cref{P.WP-short-time}.

We start with some preliminary estimates, restricting to the case $\delta \in (0,1]$. First by \Cref{L.control_P}, the continuous Sobolev embedding $H^{t_0}(\RR^d)\subset L^\infty(\RR^d)$ and \Cref{L.product}, there exists $C_1>0$, depending only on $t_0>d/2$ such that, for all $\zeta\in H^{t_0+1}(\RR^d)$ and $f\in L^2(\RR^d)$, 
\[
 \norm{ (G^\mu_0 \J^\delta \zeta) (\J^\delta  f) }_{L^\infty}\leq C_1 \No{J}^2\norm{\zeta}_{H^{t_0+1}} \norm{f}_{L^2},\]
and, for all $\psi \in \Hdot^{t_0+1}(\RR^d)$ and $f\in L^2(\RR^d)$, 
 \begin{align*}
\norm{\J^\delta \Big((G^\mu_0 \psi) |D| \big\{ (G^\mu_0 \psi) \J^\delta f \big\}\Big)}_{L^2} &\leq \delta^{-1/2} \N{\frac12}{J}  \norm{(G^\mu_0 \psi) |D| \big\{ (G^\mu_0 \psi) \J^\delta f \big\}}_{H^{-1/2}} \\
&\leq C_1 \delta^{-1/2}\N{\frac12}{J}  \norm{G^\mu_0 \psi}_{H^{t_0}} \norm{ (G^\mu_0 \psi)( \J^\delta f) }_{H^{1/2}} \\
& \leq C_1^2 \delta^{-1}\N{\frac12}{J}^2 \norm{G^\mu_0 \psi}_{H^{t_0}}^2 \norm{f}_{L^2}.
 \end{align*}
 From the above we infer that there exists $M_1>0$ depending only on $t_0>d/2$ and $\N{\frac12}{J} $ such that for any $(\zeta,\psi)\in H^{t_0+1}(\RR^d)\times \Hdot^{t_0+1}(\RR^d)$ satisfying
 \begin{equation}\label{eq.sufficient-condition-a-positive}
 \epsilon  \norm{\zeta}_{H^{t_0+1}}  + \epsilon^2\delta^{-1}\norm{G^\mu_0 \psi}_{H^{t_0}}^2 \leq M_1,
 \end{equation}
the operator $\mfa_{\J^\delta}[\epsilon \zeta,\epsilon\nabla \psi]$ satisfies the Rayleigh--Taylor condition~\eqref{RTpositive} with $\mfa_\star=1/2$. This holds in particular, using that $\delta\geq  \epsilon M_0$ and \Cref{L.control_P}, if
 \[ 0<\epsilon M_0 = \epsilon  \big( \norm{ \zeta_0 }_{H^{ \lceil t_{0} \rceil +2}} + \norm{ \mfP \psi_0 }_{H^{ \lceil t_{0} \rceil +2}} \big) \leq M_1.\]

Then, using as above \Cref{L.control_P,L.product}, we define $C_2>0$, depending only on $t_0>d/2$, such that for all $\psi \in \Hdot^{t_0+1}(\RR^d)$ and $f\in L^2(\RR^d)$, 
 \[ \norm{  (G^\mu_0 \psi) (\J^\delta f) }_{H^{1/2}} \leq C_2  \delta^{-1/2}\N{\frac12}{J}   \norm{\mfP\psi }_{H^{t_0+1/2}} \norm{f}_{L^2}.\] 
Using this estimate on the second term of $\psia \eqdef \partial^{\alpha} \psi - \epsilon (G^\mu_0 \psi) (\J \partial^{\alpha} \zeta)$, \Cref{L.control_P} as well as the above analysis on $\mfa_{\J^\delta}[\epsilon \zeta, \epsilon\nabla\psi]$, we infer that for any $C'>1$, there exists $M'>0$ depending only on $t_0>d/2$, $\N{\frac12}{J} $ and $C'$ such that  for any $(\zeta,\psi)\in H^{N}(\RR^d)\times \Hdot^{N+1/2}(\RR^d)$ satisfying
\begin{equation}\label{control_M'}
 \epsilon \norm{\zeta}_{H^{t_0+1}} \leq M'  \qquad \text{ and } \qquad \epsilon^2\delta^{-1}\norm{\mfP\psi}_{H^{t_0+1/2}}^2 \leq M',
\end{equation}
 then for any $N_\star  \in \{ \lceil t_{0} \rceil +2, N \}$, we have
 \begin{equation}\label{equiv_energy}
\frac{1}{C'} \cE^{ N_\star }(\zeta,\psi)  \leq \norm{\zeta}_{H^{N_\star}}^2 + \norm{ \mfP \psi }_{H^{N_\star}}^2 \leq C' \cE^{ N_\star }(\zeta,\psi).
\end{equation}

We can now prove the proposition. We define $C'=C^{1/3}>1$ and we introduce $M'>0$ accordingly so that~\eqref{control_M'} yields~\eqref{equiv_energy}. We consider two cases. Firstly, if $\epsilon M_0 \geq \min(M_1,\tfrac{M'}{C})$ or $\delta>1$, then $\delta\geq \min(M_1,\tfrac{M'}{C},1)\eqdef \delta_0$ where $\delta_0$ depends only on $t_0>d/2$, $\N{\frac12}{J} $,
and $C>1$. Therefore we can simply use \Cref{P.WP-short-time} with $s=N$, using that
\[ \min(1,\delta^{\max(1,t_0+\frac32-N)})  \geq \delta_0.\]
Secondly, we assume that $\epsilon M_{0} < \min(M_1,\tfrac{M'}{C})$ and $\delta \in (0,1]$. Thanks to the preliminary estimates we can apply \Cref{P.WP-large-time} with $M_U=\sqrt{C'}M$, $\mfa_\star=1/2$ and amplification factor $C''\in(1,\frac{C}{(C')^2})$. Hence there exists $\widetilde T_0$, depending only on  $t_0$, $N$, 
 $\Ni{J}$, $M$ and $C$ such that
 $T^\star_\delta> T_2$ where (using again that $\delta\geq  \epsilon M_0$ and~\eqref{equiv_energy})
 \begin{equation*}
 T_2 \eqdef \frac{\widetilde T_{0}}{\epsilon \sqrt{\cE^{\lceil t_{0} \rceil +2 }(\zeta_0,\psi_{0})} + \epsilon^{2}\delta^{-1} \cE^{\lceil t_{0} \rceil +2 }(\zeta_0,\psi_{0}) \dotN{\frac12}{J}^2}  \geq \frac{\widetilde T_0}{\epsilon M_0( \sqrt{C'}+ C' \dotN{\frac12}{J}^2   )} ,
 \end{equation*}
 and for any $0\leq t\leq T_2$ and any $N_\star \in\{ \lceil t_0 \rceil +2,N\}$,
 \begin{equation}\label{est_energy}
 \cE^{N_\star}(\zeta^\delta(t,\cdot),\psi^\delta(t,\cdot)) \leq C'' \, \cE^{N_\star}(\zeta_{0},\psi_{0})  .
 \end{equation}
The above estimate provides the desired control provided that~\eqref{control_M'} and hence~\eqref{equiv_energy} holds. 
Using that $C''(C')^2<C$, $\epsilon M_0 C\leq M'$ and $\delta\geq \epsilon M_0$, we infer from~\eqref{est_energy} with $N_\star =\lceil t_0 \rceil +2$ and the  continuity of $(\zeta^\delta(t,\cdot),\mfP\psi^\delta(t,\cdot))$ in $H^{\lceil t_0 \rceil +2}(\RR^d)^2$ that
\[I\eqdef \Big\{ t\in[0,T_2] \ : \  \norm{\zeta^\delta(t,\cdot)}_{H^{\lceil t_0 \rceil +2}}^2+\norm{\mfP\psi^\delta(t,\cdot)}_{H^{\lceil t_0 \rceil +2}}^2 \leq C\big( \norm{\zeta_0}_{H^{\lceil t_0 \rceil +2}}^2+\norm{\mfP\psi_0}_{H^{\lceil t_0 \rceil +2}}^2\big)\Big\} \]
is an open subset of $[0,T_2]$. Since it is also closed and non-empty, we have $I=[0,T_2]$. In particular we have that~\eqref{equiv_energy} holds on $[0,T_2]$, and hence~\eqref{est_energy} with $N_\star=N$ provides the desired control, which concludes the proof.
\end{proof}

\subsection{Short-time well-posedness. Proof of \texorpdfstring{\Cref{P.WP-short-time}}{Sec}}\label{S.WP-short-time}

In this section we prove \Cref{P.WP-short-time}. Let us first notice that we can specialize the study to the case $\delta=1$. Indeed, if \Cref{P.WP-short-time} holds when $\delta=1$, then we can apply this result with $\J^\delta = J^\delta(D)$ where $J^\delta=J(\delta\cdot)$ and $\delta>0$ is arbitrary. Noticing that for any $\theta\geq 0$,
\[\frac{1}{\N{\theta}{J^\delta}} \geq \frac{\min(1,\delta^{\theta})}{\N{\theta}{J}},\]
we infer that \Cref{P.WP-short-time} holds for arbitrary $\delta>0$. Consequently in the following we fix $\delta=1$.

 We write~\eqref{RWW2} as
\begin{equation}\label{eq.RWW2-compact}
\partial_t U + \L U + N(U) \ = \ 0
\end{equation}
where $U=(\zeta,\psi)^\top$, 
\[ \L \eqdef \begin{pmatrix}
0 &  -G^\mu_0 \\ 1 & 0
\end{pmatrix} , \quad \text{ and } \quad N(U)=\begin{pmatrix} \epsilon G^\mu_0 ((\J^1 \zeta) G^\mu_0 \psi)+ \epsilon \nabla\cdot((\J^1 \zeta)\nabla\psi)\\
\frac\epsilon{2} \J^1 \left( |\nabla\psi|^2- (G^\mu_0 \psi )^2 \right)
\end{pmatrix}.\]
Let $X^s\eqdef H^s(\RR^d)\times (\Hdot^{s+1/2}(\RR^d)/\RR)$ be the Hilbert space (see ~\cite[Proposition~2.3]{Lannes_ww} concerning the quotient space $\Hdot^{s+1/2}(\RR^d)/\RR$) endowed with the inner product
\[ \big(\ (\zeta_1,\psi_1) \ , \  (\zeta_2,\psi_2) \ \big)_{X^s} \ \eqdef \ \int_{\RR^d} (\Lambda^s\zeta_1)(\Lambda^s\zeta_2)+ (\Lambda^s\mfP\psi_1)(\Lambda^s\mfP\psi_2) \dd\bx.\]
We denote $|\cdot|_{X^s}$ the norm associated with the inner product $(\cdot,\cdot)_{X^s}$. 
Using \Cref{L.control_P}, one easily checks that the (unbounded) operator $\i\L$ with domain $X^{s+1/2}$ is self-adjoint on $X^s$.  Hence by Stone's theorem (see \eg~\cite[Theorem~10.8]{pazy_semigroup}) $\L$ generates a strongly continuous group of unitary operators on  $\big(X^s,\norm{\,\cdot\,}_{X^s}\big)$, which we denote $e^{t \L}$. 

By \Cref{L.commutator-tanh}, \Cref{L.compensation-tanh} (when $0 \leq s < \frac{1}{2}$) and \Cref{L.product-tame},
we have that there exists $C_1>0$ and $C_2>0$ depending only on $s$ and $t_0>d/2$ such that for any $U\in H^s(\RR^d)\times \Hdot^{s+1/2}(\RR^d)$,
\[ \norm{N(U)}_{H^s\times H^{s+1/2}} \leq 
\begin{cases}
	\epsilon C_{1} \N{t_0+\frac32-s}{J}  \norm{U}_{H^s\times \Hdot^{s+1/2}}^2& \text{ if } 0\leq s\leq t_0+\frac12,\\
	\epsilon C_{1} \N{1}{J} \norm{U}_{H^{t_{0}+\frac12}\times \Hdot^{t_{0}+1}} \norm{U}_{H^s\times \Hdot^{s+1/2}}& \text{ if }  s\geq t_0+\frac12,
\end{cases}\]
and for any $U,V\in H^s(\RR^d)\times  \Hdot^{s+1/2}(\RR^d)$
\[ \norm{N(U)-N(V)}_{H^s\times H^{s+1/2}} \leq \epsilon  C_{2} \N{\max(1,t_0+\frac32-s)}{J} \big( \norm{U+V}_{H^{s}\times \Hdot^{s+\frac12}}  \big) \norm{U-V}_{H^s\times \Hdot^{s+1/2}}.\]
Recall (see \Cref{L.control_P}) that $\big(X^s,|\cdot|_{X^s}\big)$ is equivalent to $H^s(\RR^d)\times \Hdot^{s+1/2}(\RR^d)/\RR$ 
 and that for any $\mu \geq 1$ and any $(f,g)\in H^s(\RR^d)\times H^{s+1/2}(\RR^d)$,
\[  \norm{f}_{H^s}^2 + \norm{\nabla g}_{H^{s-1/2}}^2  \ \leq\ \norm{(f,g)}_{X^s}^2 = \norm{f}_{H^{s}}^2+\norm{\mfP g}_{H^{s}}^2\ \leq\  \  \norm{f}_{H^s}^2 + \norm{g}_{H^{s+1/2}}^2 .\]
We shall then apply Banach fixed-point theorem on the Duhamel formula 
\begin{equation}\label{eq.Duhamel}
U = G_{U(0)}(U), \qquad G_{U_0}:U \ \mapsto \ \big( t\mapsto e^{-t\,\L }U_0-\int_0^t e^{-(t-\tau)\,\L} N(U(\tau,\cdot))\dd \tau \big).
\end{equation}

To this aim, we define for any $M>0$ and any $T>0$  (determined later on) the set
\[\cX^s_{T,M}:=\Big\{U\in \cC([-T,T];X^s) \ : \ \max_{t\in[-T,T]}\norm{U(t,\cdot)}_{X^{s}}^2\leq M\Big\}.\]
From the above, we find that for any $U\in \cX^s_{T,M}$ and $U_0\in X^s$, $G_{U_0}(U) \in \cC([-T,T];X^s )$ and 
\[\norm{\int_0^t e^{-(t-\tau)\,\L} N(U(\tau,\cdot))\dd \tau}_{X^s} \leq \epsilon |t|\,    C_1\, C_J \, {M}, \]
where we denote here and thereafter $C_J= \N{\max(1,t_0+\frac32-s)}{J} $; and for any $U,V\in \cX^s_T$, one has (by the triangular inequality)
 \[\norm{\int_0^t e^{-(t-\tau)\,\L} \big(N(U)-N(V)\big)(\tau,\cdot)\dd \tau}_{X^s} \leq 2\epsilon |t|\,    C_2\, C_J\, \sqrt{M} \norm{(U-V)(\tau,\cdot)}_{X^s} .\]
 Hence, choosing  $M$ and $M'$ such that $0<M' < M$ and defining $T$ as
 \begin{equation}\label{eq.def-T}
  T=\min\Big( \tfrac{ \sqrt{M}-\sqrt{M'} }{ \epsilon\,  C_1 \, C_J \, M }\ , \ \tfrac1{3\epsilon\,   C_2 \, C_J\, \sqrt{M}} \Big)\,,
 \end{equation}
we find that $G_{U_0}$ defines a contraction mapping in $\cX_{T,M}$ for any $U_0$ satisfying $\norm{U_0}_{X^{s}}^2\leq M$. 

This proves the existence and uniqueness of a solution  in $\cX_{T,M}$ to~\eqref{eq.Duhamel} with $U(0)=(\zeta_0,\psi_0)^\top$. We deduce the uniqueness in $\cX_T\eqdef \cC([-T,T];X^s) $ from a standard continuity argument, and one easily checks the equivalence between $U\in \cX_T$ satisfying~\eqref{eq.Duhamel} and $U\in \cX_T$ satisfying~\eqref{eq.RWW2-compact}. Up to now the second component of the solution as well as the second equation of~\eqref{eq.RWW2-compact} are defined up to an additive constant. Requiring additionally that~\eqref{eq.RWW2-compact} holds in  $\cC([-T,T];H^{s-1/2}(\RR^d) \times H^{s}(\RR^d))$ uniquely determines these constants ---and hence  $U\in\cC([-T,T];H^{s}(\RR^d) \times \Hdot^{s+\frac12}(\RR^d))$--- and we have $\partial_t U\in \cC([-T,T];H^{s-1/2}(\RR^d) \times H^{s}(\RR^d))$ by the above estimates for $N(\cdot)$ and straightforward bounds on $\L$.

There remains to prove the lower bound on the maximal time of existence. We focus first on the case $s>t_{0}+\frac12$. From the above (and in particular uniqueness) we can define a maximal time of existence and maximal solutions $U=(\zeta,\psi) \in \cC((-T_\star,T^\star);H^{s}(\RR^d) \times \Hdot^{s+\frac12}(\RR^d))$. On $(-T_\star,T^\star)$, we have by the above estimates
\begin{equation}\label{est-tame.duhamel}
 \norm{U(t,\cdot)}_{X^s} \leq \norm{U(0,\cdot)}_{X^s} +|t|\, (\epsilon \, C_1\, C_J)\, \norm{U(t,\cdot)}_{H^{t_0+\frac12}\times \Hdot^{t_0+\frac12}} \norm{U(t,\cdot)}_{X^s}
 \end{equation}
with $C_J= \N{1}{J}$, and (augmenting $C_1>0$ if necessary) the same estimate replacing $s$ with $t_0+\frac12$. Hence defining $T_1>0$ such that
\[ 1+ C\, T_1\, (\epsilon \, C_1\, C_J) \norm{U(0,\cdot)}_{X^{t_{0}+\frac12}} =C^{1/3},\] 
and
\[I_s\eqdef \Big\{ t\in  [-T_1,T_1]\cap (-T_\star,T^\star) \ : \ \forall \tau\in [0,t],\ \norm{U(\tau,\cdot)}_{X^{s}}^2\leq C \, \norm{U(0,\cdot)}_{X^{s}}^2 \Big\} ,\]
we infer (since $C^{2/3}<C$) that 
$I_s\cap I_{t_{0}+\frac12}$  is an open subset of $[-T_1,T_1]\cap (-T_\star,T^\star) $. Since it is also closed and non-empty, $I_s\cap I_{t_{0}+\frac12}=[-T_1,T_1]\cap (-T_\star,T^\star) $, and hence (arguing as in \Cref{C.blowup}) $\min(T_\star,T^\star)> T_1$. If now $s \leq t_{0}+\frac12$, taking  $M = C \norm{U_{0}}^{2}_{X^{s}}$ and $M' = \norm{U_{0}}^{2}_{X^{s}}$ (if $U_{0} \neq (0,0)$ in which case the result is trivial) in~\eqref{eq.def-T} provides immediately the corresponding lower bound for $T_\star$ and $T^\star$. Gathering the two previous results, we find that there exists $T_{0}>0$ depending only on $C$ and $C_{1}$, such that
\[ \min(T_\star,T^\star) \geq \frac{T_0}{\epsilon\big( \norm{\zeta_0}_{H^{\min (s,t_0+\frac12)}}+\norm{\mfP\psi_0}_{H^{\min(s,t_0+\frac12)}}\big)\,\N{\max(1,t_0+\frac32-s)}{J} }. \]

 A subtlety arises as $\norm{\nabla\psi_0}_{H^{s-\frac12}}$ does not control $\norm{\mfP\psi_0}_{H^s}$ with a uniform bound with respect to $\mu \geq 1$ (see \Cref{L.control_P}). Yet the desired result follows from the following additional ingredient. Applying $\sqrt{\frac{|D|}{\tanh(\sqrt\mu|D|)}}\langle D\rangle^{-1/2}$ to both equations in~\eqref{RWW2} and following the above arguments but with a careful use of \Cref{L.commutator-tanh}, we infer that there exists $\widetilde C_1>0$ depending only on $s$ and $s_\star\eqdef \min(s,t_{0}+\frac12)$ such that for any  $t\in [0,T^\star)$ 
 \begin{align*} 
 \norm{U(t,\cdot)}_{X^s} &\leq \norm{U(0,\cdot)}_{X^s} +|t|\, (\epsilon \, \widetilde C_1\, C_J)\, \norm{U(t,\cdot)}_{\widetilde X^{s_\star}} \norm{U(t,\cdot)}_{X^s},\\
\norm{ U(t,\cdot)}_{\widetilde X^{s_\star}} &\leq \norm{ U(0,\cdot)}_{\widetilde X^{s_\star}}  + |t|\, (\epsilon \, \widetilde C_1\, C_J)\,  \norm{U(t,\cdot)}_{\widetilde X^{s_\star}}^2,
\end{align*}
where we define $\norm{ U}_{\widetilde X^{s_\star}}\eqdef \norm{\sqrt{\frac{|D|}{\tanh(\sqrt\mu|D|)}}\langle D\rangle^{-1/2} U}_{X^{s_\star}}$, and notice that
\[ \frac1{\sqrt\mu}\norm{\zeta}_{H^{s_\star-\frac12}}^2 +\norm{\nabla\zeta}_{H^{s_\star-1}}^2+\norm{\nabla\psi}_{H^{s_\star-\frac12}}^2\lesssim \norm{(\zeta,\psi)}_{\widetilde X^{s_\star}}^2 \lesssim \norm{\zeta}_{H^{s_\star}}^2+\norm{\nabla\psi}_{H^{s_\star-\frac12}}^2.\]
Hence proceeding as previously we infer first the control of $\norm{ U(t,\cdot)}_{\widetilde X^{s_\star}}$, then the control of $\norm{U(t,\cdot)}_{X^s}$ (with the amplification factor $C>1$), for $t\in [-\widetilde T_1,\widetilde T_1]$ with $\widetilde T_1=\widetilde T_0/(\varepsilon \norm{ U}_{\widetilde X^{s_\star}}\N{t_0+\frac32-s_\star}{J}  ) $ where $\widetilde T_0$ depends only on $C$ and $\widetilde C_1$.
The proof of  \Cref{P.WP-short-time} is now complete.

\begin{Corollary}\label{C.blowup}
Under the assumptions of \Cref{P.WP-short-time}, denote $T^\star(\zeta_0,\psi_0;s)\in(0,+\infty]$ the maximal time of existence associated with initial data $(\zeta_0,\psi_0)\in H^s(\RR^d)\times \Hdot^{s+1/2}(\RR^d)$ and index $s \geq 0$, defined as  the supremum of $T>0$ such that there exists $(\zeta,\psi)\in \cC([0,T]; H^s(\RR^d)\times \Hdot^{s+1/2}(\RR^d))$ solution to~\eqref{RWW2} with initial data ${(\zeta,\psi)\id{t=0}=(\zeta_0,\psi_0)}$.

If $(\zeta_0,\psi_0)\in  H^s(\RR^d)\times \Hdot^{s+1/2}(\RR^d)$ with $s>\frac d 2 +\frac12$, then $T^\star(\zeta_0,\psi_0;s)=T^\star(\zeta_0,\psi_0;s')$ for any $s' \in \left( \frac{d}{2} + \frac12,s \right]$ and one has the blowup criterion 
\[ T^\star(\zeta_0,\psi_0;s) <\infty \quad \Longrightarrow \quad \forall s' \in \left( \frac{d}{2} + \frac12,s \right], \quad  \norm{\zeta(t,\cdot)}_{H^{s'}}+\norm{  \mfP  \psi(t,\cdot)}_{H^{s'-1/2}}\to\infty \ \text{ as $t \nearrow T^\star$}\,.\]
The corresponding result also holds for the (negative) minimal time of existence.
\end{Corollary}
\begin{proof}
Let $s$ and $s'$ such that $s\geq s'> \frac d 2+\frac12$, and $(\zeta_0,\psi_0)\in  H^s(\RR^d)\times \Hdot^{s+1/2}(\RR^d)$. Let denote for simplicity $T^\star_{s_\star}\eqdef T^\star(\zeta_0,\psi_0;s_\star)$ for $s_\star\in\{s,s'\}$. From \Cref{P.WP-short-time} and \Cref{L.control_P} we have (reasoning by contradiction and using a suitable sequence of times approaching $T^\star_s$)
\[ T^\star_s<\infty \quad \Longrightarrow \quad \norm{\zeta(t,\cdot)}_{H^{s}}+\norm{\mfP \psi(t,\cdot)}_{H^{s-1/2}}\to\infty \ \text{ as $t \nearrow T_s^\star$}\,.\]
 From the uniqueness in \Cref{P.WP-short-time}, we have obviously $T^\star_s\leq  T^\star_{s'}$ and there remains to prove that $T^\star_s\geq  T^\star_{s'}$. We argue by contradiction and assume $T^\star_s<  T^\star_{s'}$ (and in particular $T^\star_s<\infty$). Thus $(\zeta,\psi)\in \cC([0,T^\star_s]; H^{s'}(\RR^d)\times \Hdot^{s'+1/2}(\RR^d))$. Set $M'\eqdef\max_{t\in [0,T^\star_s]}\big( \norm{\zeta(t,\cdot)}_{H^{s'}}+\norm{ \mfP \psi(t,\cdot)}_{H^{s'}} \big)$. 
We use once again the tame estimates~\eqref{est-tame.duhamel} obtained in the proof of \Cref{P.WP-short-time}, and \Cref{L.control_P}: there exists $C>0$, depending on $\epsilon,s',s$ and $\J^\delta$ such that for any $t\in [0,T^\star_s)$,
\[ \norm{(\zeta,\mfP\psi)(t,\cdot)}_{H^s\times H^s} \leq \norm{(\zeta,\mfP\psi)(0,\cdot)}_{H^s\times H^s} +|t|\, C \, M'\,   \norm{(\zeta,\mfP\psi)(t,\cdot)}_{H^s\times H^s} .\]
 Grönwall's inequality provides the desired contradiction.
\end{proof}

 \begin{Remark}\label{R.WP}
 In order to certify that the initial-value problem for~\eqref{RWW2} is (unconditionally and locally-in-time) well-posed in $H^s(\RR^d)\times \Hdot^{s+1/2}(\RR^d)$ in the sense of Hadamard, one should discuss the regularity of the solution map 
 \[ \Phi:(\zeta_0,\psi_0)\in H^s(\RR^d)\times \Hdot^{s+1/2}(\RR^d) \mapsto (\zeta,\psi)\in \cC([-T,T];H^s(\RR^d)\times \Hdot^{s+1/2}(\RR^d)).\]
 The proof of \Cref{P.WP-short-time} readily shows that the solution map is Lipschitz from any ball of $H^s(\RR^d)\times \Hdot^{s+1/2}(\RR^d)$ to $\cC([-T,T];H^s(\RR^d)\times \Hdot^{s+1/2}(\RR^d))$ (with $T$ sufficiently small), and the estimates therein allow to prove that the solution map $\Phi$ is in fact analytic  (and hence infinitely differentiable), in the sense that for any $U_0\eqdef (\zeta_0,\psi_0) \in H^s(\RR^d)\times \Hdot^{s+1/2}(\RR^d)$ in a given ball and restricting to $T>0$ sufficiently small we can write 
 \[\Phi(U_0)=\sum_{k=1}^\infty \Phi_k(U_0,\cdots,U_0)\]
 where the operators $\Phi_k:\big(H^s(\RR^d)\times \Hdot^{s+1/2}(\RR^d)\big)^k\to \cC([-T,T];H^s(\RR^d)\times \Hdot^{s+1/2}(\RR^d))$ are continuous $k$-multilinear and the series is normally convergent.
 \end{Remark}

\subsection{Large-time well-posedness; proof of \texorpdfstring{\Cref{P.WP-large-time}}{Sec}}\label{S.WP-large-time}

In this section we provide the proof of \Cref{P.WP-large-time}. It follows from suitable energy estimates on smooth solutions to~\eqref{RWW2}. Here and henceforth, we refer to $(\zeta,\psi)$ as a smooth (local-in-time) solution to~\eqref{RWW2} when there exists an interval $I\subset \RR$ such that
\[\forall N\in\NN, \quad (\zeta,\psi)\in\cC^1(I;H^N(\RR^d)\times \Hdot^{N+\frac12}(\RR^d))\]
and~\eqref{RWW2} holds for any $t\in I $. The existence of smooth solutions (for smooth initial data) follows from \Cref{P.WP-short-time} and \Cref{C.blowup}. In \Cref{S.quasi} we extract a ``quasilinear structure''\footnote{As mentioned in the introduction and proved in \Cref{P.WP-short-time}, the nature of~\eqref{RWW2} is in fact semilinear. We refer to the structure of~\eqref{eq.quasi_eq1}-\eqref{eq.quasi_eq2} as quasilinear in the sense that we will refuse to make use of the full regularization effects of $\J^\delta$ but will rather obtain improved energy estimates using the skew-symmetry of the leading-order contributions of the system. The system is genuinely quasilinear if $\J^\delta$ is regularizing of order $-\r$ with $\r\in[\frac12,1)$ but not regularizing of order $-1$.  } of the system, which is then used in \Cref{S.energy-estimates} to infer the control of suitable energy functionals. The completion of the proof is postponed to \Cref{S.WP-large-completion}.

As in the previous section, let us notice that we can specialize the study to the case $\delta=1$. Indeed, assume that \Cref{P.WP-large-time} holds with $\delta=1$. Then for any $\delta\in(0,1]$ we can apply the result with $\J^\delta = J^\delta(D)$ where $J^\delta=J(\delta\cdot)$. Since $M_{J^\delta}\leq M_J$ and $ \dotN{\frac12}{J^\delta}^2 = \delta^{-1}\dotN{\frac12}{J}^2$, we find that
 \Cref{P.WP-large-time} holds with arbitrary $\delta\in (0,1]$. 
 {\em Consequently we fix $\delta=1$ in the following, and denote for simplicity $\J\eqdef \J^1$.}

\subsubsection{Quasilinearization}\label{S.quasi}

\begin{Proposition}\label{P.quasi}
Let $d\in\{1,2\}$, $t_{0}>\frac{d}{2}$, $N \in \NN$ with $N \geq t_0+2$. There exists $C>0$ such that for any $\epsilon \geq 0$, $\mu \geq 1$, $\J$ regular rectifier,  and $(\zeta,\psi)$ smooth solution to~\eqref{RWW2},
 the following holds. Let $\alpha \in \NN^{d}$ a multi-index. If $|\alpha| \leq N-1$, we have
\begin{align}\label{eq.quasi_eq1_low}
\partial_t \partial^{\alpha} \zeta-G^\mu_0 \partial^{\alpha} \psi &= \epsilon R_{1}^{(\alpha)} ,\\
\label{eq.quasi_eq2_low}
\partial_t \partial^{\alpha} \psi + \partial^{\alpha} \zeta &= \epsilon  R_{2}^{(\alpha)} ,
\end{align}
with
\begin{align*}
 \norm{ R_{1}^{(\alpha)} }_{L^2} &\leq C\, \No{J}\, \big(\norm{ \nabla \psi }_{H^{t_{0}}} \norm{ \zeta }_{H^{N}} + \norm{ \zeta }_{H^{t_{0}+1}} \norm{ \nabla \psi }_{H^{N-1}}\big),\\
 \norm{ R_{2}^{(\alpha)}}_{L^2}  &\leq C\, \No{J}\, \norm{\nabla\psi}_{H^{t_{0}}} \norm{\nabla\psi}_{H^{N-1}}  ,
 \end{align*}
whereas if $|\alpha|=N$
\begin{align}
\label{eq.quasi_eq1}
\partial_t \zetaa-G^\mu_0 \psia  + \epsilon \nabla\cdot((\J \zetaa)\nabla\psi) &=\epsilon \widetilde R_{1}^{(\alpha)},\\
\label{eq.quasi_eq2}
\partial_{t} \psia + \mfa_{\J}[\epsilon \zeta, \epsilon\nabla\psi] \zetaa + \epsilon\J \left( \nabla \psi \cdot \nabla \psia \right) &= \epsilon \widetilde R_2^{(\alpha)}+ \epsilon^2 \widetilde R_{3}^{(\alpha)},
\end{align}
with $\zetaa \eqdef \partial^{\alpha} \zeta$ and $\psia \eqdef  \partial^{\alpha} \psi - \epsilon (G^\mu_0 \psi) (\J   \partial^\alpha\zeta)$, $ \mfa_{\J}$ defined in~\eqref{eq.def-RT} and
\begin{align*} \norm{ \widetilde R_{1}^{(\alpha)} }_{L^2} &\leq C\,\No{J}\, \big(  \norm{ \nabla \psi }_{H^{t_{0}+1}} \norm{ \zeta }_{H^{N}} + \norm{ \zeta }_{H^{t_{0}+2}} \norm{ \nabla \psi }_{H^{N-1}} \big),\\
\norm{ \widetilde R_{2}^{(\alpha)}}_{H^{\frac{1}{2}}}  &\leq C \, \big( \No{J}\, \norm{ \nabla \psi }_{H^{t_{0}}} \norm{ \nabla \psi }_{H^{N-\frac12}} + \dotNi{J}\norm{ \nabla \psi }_{H^{t_{0}+1}} \norm{ \psia }_{H^{\frac{1}{2}}}  \big) , \\
 \norm{ \widetilde R_{3}^{(\alpha)}}_{H^{\frac{1}{2}}} &\leq C\, \N{\frac12}{J}\, \dotNi{J}\,\norm{ \nabla \psi }_{H^{t_{0}}}  \big( \norm{ \nabla \psi }_{H^{t_{0}+1}} \norm{ \zeta }_{H^{N}} + \norm{ \zeta }_{H^{t_{0}+2}} \norm{\nabla \psi}_{H^{N-1}} \big) .
 \end{align*}

\end{Proposition}
\begin{proof}We first focus on the first equation.
By the second estimate in \Cref{L.commutator-tanh} (with $r=0$) and the fact that $\Norm{\J}_{H^s\to H^s}=\No{J} $ and (see \Cref{L.control_P}) $\Norm{G^\mu_0}_{\Hdot^{s+1}\to H^s}=1$ for any $s\in\RR$ and $\mu>0$,  we get
\begin{align*}
\norm{ G^\mu_0 ((\J \zeta) G^\mu_0 \psi)+\nabla\cdot((\J \zeta)\nabla\psi) }_{H^{N-1}} \lesssim  \No{J} \norm{ \zeta }_{H^{t_{0}+1}} \norm{ \nabla \psi }_{H^{N-1}} + \No{J}\norm{ \zeta }_{H^{N}}  \norm{ \nabla \psi }_{H^{t_{0}}} .
\end{align*}
This provides the estimate for $R_1^{(\alpha)}$ for $|\alpha|\leq N-1$.
We consider now the case $|\alpha|=N$. We differentiate $\alpha$ times the first equation of~\eqref{RWW2}. We get
\begin{equation*}
\partial_t \partial^{\alpha} \zeta-G^\mu_0 \partial^{\alpha} \psi 
+\epsilon G^\mu_0 ((\J \partial^{\alpha} \zeta) G^\mu_0 \psi)+ \epsilon \nabla\cdot((\J \partial^{\alpha} \zeta)\nabla\psi) = \epsilon \underset{\substack{\beta+\gamma=\alpha \\ 0 \leq |\beta| \leq N-1 } }{\sum} A_{(\beta,\gamma)}\eqdef \epsilon\widetilde R_1^{(\alpha)}
\end{equation*}
where
\begin{equation*}
A_{(\beta,\gamma)} =C(\beta,\gamma) \left( G^\mu_0 (( \J \partial^{\beta} \zeta) G^\mu_0 \partial^{\gamma} \psi)+ \nabla\cdot( (\J \partial^{\beta} \zeta) \nabla \partial^{\gamma} \psi) \right).
\end{equation*} 
If $|\beta|=0$ or $|\beta|=1$ using the first estimate in \Cref{L.commutator-tanh} (with $s=0$ and $r=1-|\beta|$), we get
\begin{equation*}
\norm{ A_{(\beta,\gamma)} }_{L^2} \lesssim \norm{\J\partial^\beta\zeta }_{H^{t_{0}+2-|\beta|}} \norm{ \partial^{\gamma}\nabla \psi }_{H^{|\beta|-1}}\lesssim \No{J} \norm{ \zeta }_{H^{t_{0}+2}} \norm{ \nabla \psi }_{H^{N-1}}.
\end{equation*}
If $2 \leq |\beta| \leq N-1$, we obtain by the triangular inequality and the product estimate in \Cref{L.product-tame} with $s_1=N-|\beta|$, $s_2=t_0+1-|\gamma|$, $s_1'=t_0+2-|\beta|$, $s_2'=N-1-|\gamma|$, 
\begin{align*}
\norm{ A_{(\beta,\gamma)} }_{L^2} &\leq \norm{(\J \partial^{\beta} \zeta) (G^\mu_0 \partial^{\gamma} \psi) }_{H^1} + \norm{ (\J \partial^{\beta} \zeta)(\nabla \partial^{\gamma} \psi) }_{H^1}\\
&\lesssim  \norm{\J \zeta }_{H^{N}} \big( \norm{ G^\mu_0  \psi }_{H^{t_{0}+1}} +\norm{ \nabla  \psi }_{H^{t_{0}+1}}   \big) + \norm{\J \zeta }_{H^{t_{0}+2}} \big(\norm{ G^\mu_0 \psi }_{H^{N-1}} +\norm{ \nabla \psi }_{H^{N-1}} \big) \\
&\lesssim \No{J} \norm{ \nabla \psi }_{H^{t_{0}+1}} \norm{ \zeta }_{H^{N}} + \No{J} \norm{ \zeta }_{H^{t_{0}+2}} \norm{ \nabla \psi }_{H^{N-1}}.
\end{align*}
This concludes the estimate for $\widetilde R_1^{(\alpha)}$.
\medskip

We now focus on the second equation of~\eqref{RWW2}. First we notice that, by \Cref{L.product-tame},
\[\norm{\J \left( |\nabla\psi|^2- (G^\mu_0 \psi )^2 \right)}_{H^{N-1}}\lesssim \No{J} \norm{\nabla\psi}_{H^{t_0}}\norm{\nabla\psi}_{H^{N-1}}.\]
This provides the estimate for $R_2^{(\alpha)}$ for $|\alpha|\leq N-1$. Now we consider the case $|\alpha|=N$. 
Differentiating $\alpha$ times the second equation of~\eqref{RWW2}, we get 
\[
\partial_t \partial^{\alpha} \psi+\partial^{\alpha} \zeta +\epsilon \J  B_{(\alpha)}=\epsilon    \underset{\substack{\beta+\gamma=\alpha \\ 1 \leq |\beta| \leq N-1 } }{\sum} \J B_{(\beta,\gamma)}\eqdef \epsilon \widetilde R_{2,i}^{(\alpha)} \]
with
\[B_{(\alpha)}\eqdef  \nabla  \psi \cdot (\nabla \partial^{\alpha} \psi) - (G^\mu_0 \psi ) (G^\mu_0 \partial^{\alpha} \psi ) \]
and
\[ B_{(\beta,\gamma)} \eqdef C(\beta,\gamma) \left( (\nabla \partial^{\beta} \psi) \cdot (\nabla  \partial^{\gamma} \psi ) - (G^\mu_0 \partial^{\beta}  \psi ) (G^\mu_0 \partial^{\gamma} \psi ) \right).\]
Then, using the unknown $\psia \eqdef  \partial^{\alpha} \psi - \epsilon (G^\mu_0 \psi) (\J   \partial^\alpha\zeta)$, we can rewrite the previous equation as
\[\partial_t \psia + \zetaa + \epsilon (G^\mu_0 \partial_{t} \psi) (\J \zetaa)  + \epsilon (G^\mu_0 \psi)( \J \partial_{t} \zetaa )+ \epsilon \J  B_{(\alpha)} = \epsilon  \widetilde R_{2,i}^{(\alpha)} \]
and using~\eqref{eq.quasi_eq1} and reorganizing terms,
\begin{multline*}\partial_t \psia + \mathfrak{\tilde{a}}[\epsilon \J  \zeta, \epsilon\nabla\psi] \zetaa    
+ \epsilon \J  \big(  (\nabla  \psi )\cdot(\nabla  \psia)-(G^\mu_0 \psi ) (G^\mu_0 \psia) \big)\\
+ \epsilon (G^\mu_0 \psi)( \J  G^\mu_0 \psia   )
 = \epsilon \widetilde R_{2,i}^{(\alpha)} - \epsilon^2 (G^\mu_0 \psi)( \J \widetilde R_{1}^{(\alpha)}  )
\end{multline*}
where
\begin{align*}
\mathfrak{\tilde{a}}[\epsilon \J  \zeta, \psi] \zetaa = &\zetaa + \epsilon (G^\mu_0 \partial_{t} \psi) (\J \zetaa) - \epsilon^2 (G^\mu_0 \psi)( \J \nabla\cdot((\J \zetaa)\nabla\psi)  )\\
&+\epsilon^2 \J \big( (\nabla  \psi )\cdot \nabla  \{ (G^\mu_0 \psi) (\J   \zetaa)\} \big) - \epsilon^2 \J \big( (G^\mu_0 \psi ) (G^\mu_0  \{ (G^\mu_0 \psi) (\J   \zetaa)\}) \big) \\
=&\mfa_{\J}[\epsilon \zeta, \psi] \zetaa  + \epsilon (G^\mu_0 (\partial_{t} \psi+\zeta)) (\J \zetaa) - \epsilon^2 (G^\mu_0 \psi)( \J \nabla\cdot((\J \zetaa)\nabla\psi)  )\\
&+\epsilon^2 \J \big( (\nabla  \psi )\cdot \nabla  \{ (G^\mu_0 \psi) (\J   \zetaa)\} \big)\\
&- \epsilon^2 \J \big( (G^\mu_0 \psi ) \big((G^\mu_0-|D|)  \{ (G^\mu_0 \psi) (\J   \zetaa)\}\big) \big).
\end{align*}
Let us estimate each of the contributions. 
\begin{itemize}
\item We get from the second product estimates of \Cref{L.product-tame} (with $s=\frac12$, $s_1=N-\frac12-|\beta|$, $s_2= t_0+1-|\gamma|$, $s_1'=t_0+1-|\beta|$ $s_2'=N-\frac12-|\gamma|$ )  
and \Cref{L.control_P}
\[ \norm{\widetilde R_{2,i}^{(\alpha)}}_{H^{\frac12}} \leq \No{J}  \underset{\substack{\beta+\gamma=\alpha \\ 1 \leq |\beta| \leq N-1 } }{\sum}  \norm{B_{(\beta,\gamma)}}_{H^{\frac{1}{2}}} \lesssim \No{J} \norm{ \nabla \psi }_{H^{t_{0}+1}} \norm{ \nabla \psi }_{H^{N-\frac{1}{2}}}. \]
\item Denoting
\[ \widetilde R_{2,ii}^{(\alpha)} \eqdef - \J  \big(  (G^\mu_0 \psi ) (G^\mu_0 \psia) \big) + (G^\mu_0 \psi ) ( \J  G^\mu_0 \psia)\]
and by the commutator estimate in \Cref{L.commut_order0} with $s=1/2$ and \Cref{L.control_P}, we have
\[ \norm{\widetilde R_{2,ii}^{(\alpha)}  }_{H^{1/2}} \lesssim  \dotNi{J}\norm{\nabla \psi}_{H^{t_0+1}}\norm{\nabla \psia}_{H^{-1/2}}.\]
\item We have by \Cref{L.product} and \Cref{L.control_P}

\[\norm{ (G^\mu_0 \psi)( \J \widetilde R_{1}^{(\alpha)}  ) }_{H^{1/2}} \lesssim \norm{G^\mu_0 \psi}_{H^{t_0}} \norm{ \J \widetilde R_{1}^{(\alpha)}  }_{H^{1/2}} \leq \norm{\nabla \psi}_{H^{t_0}} \norm{ \J \widetilde R_{1}^{(\alpha)}  }_{H^{1/2}}\]
and by the previously obtained estimate on $\widetilde R_{1}^{(\alpha)}$,  and that $\Norm{\J}_{L^2\to H^{1/2}}= \N{\frac12}{J} $, 
\[\norm{ \J \widetilde R_{1}^{(\alpha)}  }_{H^{1/2}}\lesssim \N{\frac12}{J} \big( \No{J} \norm{ \nabla \psi }_{H^{t_{0}+1}} \norm{ \zeta }_{H^{N}} + \No{J} \norm{ \zeta }_{H^{t_{0}+2}} \norm{ \nabla \psi }_{H^{N-1}} \big).\]
\item Rewriting
\begin{align*} \widetilde R_{3,i}^{(\alpha)}&\eqdef \J \big( (\nabla  \psi )\cdot(\nabla  \{ (G^\mu_0 \psi) (\J   \zetaa)\}) \big)- (G^\mu_0 \psi)( \J \nabla\cdot((\J \zetaa)\nabla\psi)  ) \\
&  = \big[\J ,G^\mu_0 \psi \big]\big( (\nabla  \psi )\cdot(\J  \nabla  \zetaa)\big)\\
&\qquad + \J \big( (\nabla  \psi )\cdot(   (\J   \zetaa) (\nabla G^\mu_0 \psi)) \big)- (G^\mu_0 \psi)( \J \{ (\J \zetaa) (\Delta\psi) \} ),
\end{align*}
we find using as above \Cref{L.commut_order0} for the first contribution and \Cref{L.product} for the other ones
\begin{multline*}\norm{\widetilde R_{3,i}^{(\alpha)}}_{H^{1/2}} \lesssim \dotNi{J}\norm{\nabla \psi}_{H^{t_0+1}}\norm{\nabla \psi}_{H^{t_0}}\norm{\J  \nabla  \zetaa }_{H^{-1/2}}\\
	+ \No{J}\norm{\nabla \psi}_{H^{t_0+1}}\norm{\nabla \psi}_{H^{t_0}}\norm{\J   \zetaa }_{H^{1/2}} .
\end{multline*}
\item Then, we easily find, since for any $\xi\in\RR^+$, $|\tanh(\sqrt\mu|\xi|)-1|\leq |\tanh(|\xi|)-1| \lesssim \langle\xi\rangle^{-3/2}$, that
\begin{align*} \widetilde R_{3,ii}&\eqdef 
- \J\big(  (G^\mu_0 \psi ) (G^\mu_0  \{ (G^\mu_0 \psi) (\J   \zetaa)\})\big) + \J\big(  (G^\mu_0 \psi ) (|D| \{ (G^\mu_0 \psi) (\J   \zetaa)\})\big)\\
&= 
-\J\big(  (G^\mu_0 \psi ) (|D|(\tanh(\sqrt\mu |D|)-1)  \{ (G^\mu_0 \psi) (\J   \zetaa)\})\big)
\end{align*}
can be estimated as 
\[ \norm{\widetilde R_{3,ii}^{(\alpha)}}_{H^{1/2}}\lesssim \No{J} \norm{\nabla \psi}_{H^{t_0}}^2\norm{\J\zetaa}_{L^2}.\]
\item Finally, we set
\begin{align*} \widetilde R_{3,iii}^{(\alpha)}&\eqdef  \big(G^\mu_0 (\partial_t\psi+\zeta) \big) (\J \zetaa)  \\
&= -\frac\epsilon{2 }   \big(G^\mu_0 \J \left( |\nabla\psi|^2- (G^\mu_0 \psi )^2 \right) \big) (\J \zetaa)  
\end{align*}
(where we used the first equation of~\eqref{RWW2}) and by \Cref{L.product} 
  \[ \norm{\widetilde R_{3,iii}^{(\alpha)}}_{H^{1/2}}\lesssim \epsilon \No{J} \norm{\nabla \psi}_{H^{t_0+1}} \norm{\nabla \psi}_{H^{t_0}} \norm{\J\zetaa}_{H^{1/2}}.\]
\end{itemize}
We conclude combining the above estimates and the fact that $\norm{\J \zetaa }_{H^{1/2}}  \leq \N{\frac12}{J} \norm{\zeta}_{H^N}$.
\end{proof}

\begin{Remark}\label{R.trueRT}
Several contributions in $\widetilde R_{3}^{(\alpha)}$ (but also in $\widetilde R_{2}^{(\alpha)}$ via the control of $\norm{\nabla\psi}_{H^{N-\frac12}}$) require the regularizing properties of $\J$ to be controlled. 
A more careful analysis shows that most ---but not all--- contributions could be tackled through additional terms on the Rayleigh--Taylor operator. Shortly put, if we set $\J=\Id$,
 then \Cref{P.quasi} holds with suitable estimates on the remainders (by which we mean controlled by the energy functional) if we replace $\mfa_{\Id}[\epsilon\zeta,\epsilon\nabla\psi]\zeta_{(\alpha)}$ with $\widetilde{\mfa_{\Id}}[\epsilon\zeta,\epsilon\nabla\psi]\zeta_{(\alpha)}+\mathfrak{b}_\alpha(\zeta,\psi)$ where
\begin{multline}\label{eq.trueRT}
	\widetilde{\mfa_{\Id}}[\epsilon\zeta,\epsilon\nabla\psi]\zeta = \Big( 1 + \epsilon (G^\mu_0 \partial_t\psi)   + \epsilon^2  (\nabla  \psi )\cdot (\nabla G^\mu_0 \psi) -\epsilon^2 (G^\mu_0 \psi)( \Delta\psi) \Big) \, \zetaa\\
- \epsilon^2 \Big((G^\mu_0 \psi) |D| \big\{ (G^\mu_0 \psi)  \zetaa \big\}\Big)   
 \end{multline}
 and
 \begin{equation} 
 {\mathfrak{b}}_\alpha(\zeta,\psi)
 =\epsilon \hspace{-0.5cm} \sum_{|\beta|=1,\ \beta\leq\alpha} \binom{\alpha}{\beta} \Big( (G^\mu_0 \partial^\beta\psi) (G^\mu_0\partial^{\alpha-\beta} \psi )+ \epsilon(G^\mu_0 \psi) G^\mu_0\big((G^\mu_0 \partial^\beta \psi)  (\partial^{\alpha-\beta} \zeta) \big)\Big).
 \end{equation}
The contribution ${\mathfrak{b}}_\alpha(\zeta,\psi)$ is ``bad'' in the sense that it cannot be estimated as an order-zero term and has no particular skew-symmetric structure (when tested against $G^\mu_0 \psia$). It can be seen as the genuine source of instability for the system \eqref{RWW2}.
\end{Remark}

\subsubsection{Energy estimates}\label{S.energy-estimates}

This section is devoted to the proof of energy estimates on smooth solutions to~\eqref{RWW2}.
We start with some elementary results which provide useful tools for the comparison of the energy functional, $\cE^N(\zeta,\psi)$, defined in~\eqref{eq.def-E}, and suitable norms of $(\zeta,\psi)$. 

\begin{Lemma}\label{P.psi_controls}
Let $d\in\{1,2\}$, $t_{0}>\frac{d}{2}$ and $N \in \NN$ with $N \geq t_{0}+\frac{3}{2}$. There exists $C>0$ such that for any $\epsilon\geq 0$, $\mu \geq 1$, $\zeta\in H^N(\RR^d)$ and $\psi\in\Hdot^{N+\frac12}(\RR^d)$, and $\alpha \in \NN^{d}$ with $|\alpha|=N$,
\begin{align*}
&\norm{\nabla \psi }_{H^{t_{0}}} \leq \norm{ \mfP \psi }_{H^{t_0+1/2}} \leq \norm{ \mfP \psi }_{H^{N-1}} ,\\
&\norm{ \nabla \psi }_{H^{N-1}} \leq C\, \big(  \norm{ \mfP \psi }_{H^{N-1}}+\sup_{|\beta|=N}\norm{ \mfP \psib}_{L^2}+\epsilon \No{J}\norm{ \nabla \psi }_{H^{t_{0}}}   \norm{ \zeta }_{H^{N}}  \big),\\
&\norm{ \psia }_{H^{\frac{1}{2}}} \leq C\, \big(  \norm{ \mfP \psi }_{H^{N-1}} + \sup_{|\beta|=N}\norm{ \mfP \psib}_{L^2}+\epsilon \No{J}\norm{ \nabla \psi }_{H^{t_{0}}}    \norm{ \zeta }_{H^{N}}  \big),\\
&\norm{ \partial^{\alpha} \psi }_{H^{\frac{1}{2}}} \leq C\, \big( \norm{ \mfP \psia}_{L^2} +\epsilon  \N{\frac12}{J}\norm{ \nabla \psi }_{H^{t_{0}}}  \norm{ \zeta }_{H^{N}}  \big),
\end{align*}
where we recall that  $\psia \eqdef \partial^{\alpha} \psi - \epsilon (G^\mu_0 \psi) (\J \partial^{\alpha} \zeta)$.
\end{Lemma}
\begin{proof}
The first inequality has been proved in \Cref{L.control_P}.
Then the following holds for any $\beta\in \NN^d$ with $|\beta|=N$:
\begin{align*}
\norm{\partial^{\beta} \psi }_{L^2} &\leq \norm{ \tfrac{(1+|D|^2)^{1/4}}{|D|} \mfP \partial^{\beta} \psi }_{L^2} \leq \norm{ (\tfrac{1}{|D|} + \tfrac{1}{(1+|D|^2)^{1/4}}) \mfP \partial^{\beta} \psi }_{L^2} \\
&\leq \norm{\mfP\psi}_{H^{N-1}} +  \norm{   \mfP \partial^{\beta} \psi }_{H^{-1/2}}
\end{align*}
and, since $\partial^{\beta} \psi=\psib + \epsilon (G^\mu_0 \psi) (\J \partial^{\beta} \zeta)$,
\[\norm{ \mfP \partial^{\beta} \psi }_{H^{-1/2}}
\leq \norm{ \mathfrak{P} \psib }_{H^{-1/2}} + \epsilon \norm{ (G^\mu_0 \psi) \J \partial^{\beta} \zeta }_{L^2}\lesssim \norm{ \mathfrak{P} \psib }_{L^2} + \epsilon \No{J}\norm{ \nabla \psi}_{H^{t_0}} \norm{\partial^{\beta} \zeta }_{L^2}  
\]
where we use \Cref{L.control_P} and the continuous Sobolev embedding $H^{t_0}(\RR^d)\subset L^\infty(\RR^d)$. The control of $\norm{\partial^\beta\psi}_{L^2}$, and hence $\norm{\nabla\psi}_{H^{N-1}}$ using the first inequality, immediately follows. We infer the same control on $\norm{\psia}_{L^2}$ which, combined with $\norm{\nabla\psia}_{H^{-1/2}} \leq\norm{\mfP\psia}_{L^2}$ (see \Cref{L.control_P}), yields the third inequality.
Finally,
\[
\norm{ \nabla\partial^{\alpha} \psi }_{H^{-1/2}} \leq \norm{\mfP \partial^{\alpha} \psi }_{L^2}\leq \norm{\mfP \psia }_{L^2} +  \epsilon \norm{ (G^\mu_0 \psi) \J \partial^{\alpha} \zeta }_{H^{1/2}}
\]
and, by the product estimate in \Cref{L.product}
\[ \norm{ (G^\mu_0 \psi ) \J \partial^{\alpha} \zeta }_{H^{1/2}}\lesssim \norm{\nabla\psi}_{H^{t_0}}\norm{\J\zeta}_{H^{N+1/2}}\leq \N{\frac12}{J} \norm{\nabla\psi}_{H^{t_0}} \norm{ \zeta }_{H^{N}} ,\]
and the fourth inequality follows. 
\end{proof}

Now we are ready to prove the following key estimates.
\begin{Proposition}\label{P.energy_est}
Let $d\in\{1,2\}$, $t_{0}>\frac{d}{2}$, $N \in \NN$ with $N \geq t_{0}+2$, $\mfa_\star>0$, $M_J,M_U>0$. There exists $C>0$ such that for any $\epsilon \geq 0$, $\mu \geq 1$, regular rectifier $\J$ (see \Cref{D.J}) with $\Ni{J} \leq  M_J$, and any  $(\zeta,\psi)$ smooth solution to~\eqref{RWW2} on the time interval $I\subset \RR$ and satisfying for any $t\in I$
\[ \epsilon M(t)\eqdef \epsilon \big( \norm{ \zeta(t,\cdot) }_{H^{t_{0}+2}} + \norm{ \nabla \psi(t,\cdot) }_{H^{t_{0}+1}} \big) \leq M_U\]
and the Rayleigh--Taylor condition (with $\mfa_{\J}$ defined in~\eqref{eq.def-RT})
\[ \forall f\in L^2(\RR^d) \text{ , } \left( f , \mfa_\J[\epsilon \zeta(t,\cdot),\epsilon\nabla \psi(t,\cdot)]  f\right)_{L^2} \geq \mfa_\star \norm{ f }_{L^2}^{2},\]
one has for any $t\in I$,
\begin{equation*}
\frac{\dd}{\dd t} \cE^{N}(\zeta(t,\cdot),\psi(t,\cdot)) \ \leq\  C \left( \epsilon M(t) + \dotN{\frac12}{J}^{2} \epsilon^2 M(t)^2 \right) \cE^{N}(\zeta(t,\cdot),\psi(t,\cdot))\,,
\end{equation*}
where the functional $\cE^{N}$ is defined in~\eqref{eq.def-E}.
\end{Proposition}
\begin{proof}In this proof, we denote by $C$ a constant which depends uniquely on $N$ and $t_0$, and by $C_{M}$ a constant which depends additionally and non-decreasingly on $M_{U}$ and $M_{J}$. They vary from line to line.
Let $\alpha\in \NN^d$ and consider \Cref{P.quasi}. 

When $|\alpha| \leq N-1$, we sum the $L^{2}(\RR^d)$ inner product of~\eqref{eq.quasi_eq1_low}
 against  $\partial^{\alpha} \zeta$ and the $L^{2}(\RR^d)$ inner product of~\eqref{eq.quasi_eq2_low} in with $G^\mu_0 \partial^{\alpha} \psi $. It follows, using Cauchy-Schwarz inequality and \Cref{L.control_P},
\begin{multline}\label{eq.energy-est_low}
\frac{\dd}{\dd t} \Big(\norm{\partial^{\alpha} \zeta}_{L^2}^2+\norm{\partial^{\alpha} \mfP\psi}_{L^2}^2 \Big) \\
\leq  \epsilon\,
C \, \No{J}\  \Big( \norm{ \nabla \psi }_{H^{t_{0}+1}} \norm{ \zeta}_{H^{N}}^2 + \norm{ \zeta }_{H^{t_{0}+2}} \norm{ \zeta}_{H^{N}} \norm{ \nabla \psi }_{H^{N-1}} + \norm{ \nabla \psi }_{H^{t_{0}+1}} \norm{ \nabla \psi }_{H^{N-1}}^2 \Big).
\end{multline}

When $|\alpha|=N$, we sum the $L^{2}(\RR^d)$ inner product of~\eqref{eq.quasi_eq1}
against  $\mfa_{\J}[\epsilon \zeta, \epsilon\nabla\psi] \zetaa$ and the $L^{2}(\RR^d)$ inner product of~\eqref{eq.quasi_eq2} with $G^\mu_0 \psia $,
where we recall the notations $\zetaa\eqdef \partial^{\alpha} \zeta$ and $\psia\eqdef   \partial^{\alpha} \psi - \epsilon (G^\mu_0 \psi) (\J   \zetaa)$.
This yields
\begin{subequations}
\label{eq.energy-est}
\begin{align}
\nonumber&\frac12\frac{\dd}{\dd t} \left( \big(  \zetaa , \mfa_{\J}[\epsilon \zeta, \epsilon\nabla\psi] \zetaa \big)_{L^2} + \norm{ \mfP \psia }_{L^2}^2 \right)\\
\label{C1}    & \hspace{3em} = \tfrac12 \big( \zetaa ,\left[\partial_{t},\mfa_{\J}[\epsilon \zeta, \epsilon\nabla\psi] \right]  \zetaa \big)_{L^2} +
  \tfrac\epsilon2 \big( \left[  \J ,  G^\mu_0 \J \zeta  \right] \partial_{t} \zetaa , \zetaa \big)_{L^2}  \\
\label{C2}  & \hspace{4em} - \epsilon \big( \nabla\cdot((\J \zetaa)\nabla\psi) , \mfa_{\J}[\epsilon\zeta, \epsilon\nabla\psi]  \zetaa\big)_{L^2}- \epsilon \big(\J \left( \nabla \psi \cdot \nabla \psia \right) , G^\mu_0 \psia \big)_{L^2} \\
\label{C3}  &\hspace{4em} 
+ \epsilon \big( \widetilde R_{1}^{(\alpha)}, \mfa_{\J}[\epsilon \zeta, \epsilon\nabla\psi]  \zetaa)_{L^2} +\epsilon \big( \widetilde R_2^{(\alpha)}+ \epsilon \widetilde R_{3}^{(\alpha)}, G^\mu_0 \psia\big)_{L^2}.
  \end{align}
  \end{subequations}
  We now estimate each term on the right-hand side of~\eqref{eq.energy-est}.
  
{\em Contributions in~\eqref{C1}.} By direct inspection, we find
\begin{multline*}\left[\partial_{t},\mfa_{\J}[\epsilon  \zeta, \epsilon\nabla\psi] \right] \zetaa  = \epsilon (G^\mu_0  \J \partial_t \zeta) \J  \zetaa \\
- \epsilon^2\J \Big((G^\mu_0 \partial_t  \psi) |D| \big\{ (G^\mu_0 \psi) (\J \zetaa)\big\}- \epsilon^2\J (G^\mu_0 \psi) |D| \big\{ (G^\mu_0 \partial_t \psi) (\J \zetaa)\big\}\Big)
\end{multline*}
and hence, using triangular inequality and product estimates in \Cref{L.product},
\[ \norm{ \left[\partial_{t},\mfa_{\J}[\epsilon \zeta, \epsilon\nabla\psi] \right] \zetaa  }_{L^2} \lesssim   \epsilon\,\,\No{J}^2 \norm{ \partial_t\zeta}_{H^{t_0+1}} \norm{ \zetaa}_{L^2} + \epsilon^2\,\, \dotN{\frac12}{J}^2 \norm{ \partial_t \nabla\psi }_{H^{t_0}} \norm{ \nabla\psi }_{H^{t_0}}   \norm{ \zetaa}_{L^2} \,.
 \]
Using the equations~\eqref{RWW2}, the first inequality in \Cref{L.commutator-tanh} (with $s=t_0+1$, $r=1$) we have
\[ \norm{\partial_t  \zeta}_{H^{t_0+1}} \leq \norm{ \nabla \psi }_{H^{t_0+1}} + \epsilon\,  C \, \No{J}  \norm{ \zeta}_{H^{t_{0}+2}} \,   \norm{ \nabla \psi }_{H^{t_{0}}}\,,\]
and by the product estimate in \Cref{L.product} we infer
\[
\norm{\partial_t \nabla \psi}_{H^{t_0}} \leq \norm{ \zeta}_{H^{t_0+1}} + \epsilon\,  C \No{J} \norm{\nabla\psi}_{H^{t_{0}}} \norm{\nabla\psi}_{H^{t_{0}+1}}  \,. \]
By Cauchy-Schwarz inequality and collecting the above, we find
\begin{multline}\label{est-C11}
\big\vert\big( \zetaa , \left[\partial_{t},\mfa_{\J}[\epsilon \zeta, \epsilon\nabla\psi] \right] \zetaa \big)_{L^2}  \big\vert\leq \epsilon C_{M}  \norm{ \nabla \psi }_{H^{t_{0}+1}} \norm{ \zetaa}_{L^2}^2 \\
+ \epsilon^2\,  C\, \N{\frac12}{J}^2 \big( \norm{ \zeta}_{H^{t_{0}+2}} + \norm{\nabla\psi}_{H^{t_{0}+1}} \big)^2  \norm{ \zetaa}_{L^2}^2 \,.
\end{multline}
For the second term in~\eqref{C1}, we have by \Cref{L.commut_order0} with $s=0$,
\[
\norm{\left[  \J ,  G^\mu_0 \J \zeta  \right] \partial_{t} \zetaa  }_{L^2} \leq C \Ni{J} \norm{\nabla \J \zeta}_{H^{t_0+1}}\norm{\partial_{t} \zetaa }_{H^{-1}}\]
which yields, proceeding as above but with the second inequality in \Cref{L.commutator-tanh} (with $s=N-1$) 
\begin{equation}\label{est-C12}
\big\vert \tfrac\epsilon2 \big( \left[  \J ,  G^\mu_0 \J \zeta  \right] \partial_{t} \zetaa , \zetaa \big)_{L^2}   \big\vert \leq \epsilon\, C_{M} \norm{ \zeta}_{H^{t_0+2}} \big( \norm{ \zeta}_{H^{N}} \,  + \norm{ \nabla \psi }_{H^{N-1}} \big) \norm{\zetaa}_{L^2}\,.
\end{equation}

{\em Contributions in~\eqref{C2}.} Using integration by parts and that $\J$ is self-adjoint, we find
\begin{align*} & \big( \nabla\cdot((\J \zetaa)\nabla\psi) , \mfa_{\J}[\epsilon \zeta, \epsilon\nabla\psi]  \zetaa\big)_{L^2}\\
&\quad=
    \Big( \nabla\cdot((\J \zetaa)\nabla\psi) ,  \zetaa  - \epsilon (G^\mu_0 \J \zeta) \J  \zetaa - \epsilon^2\J \big((G^\mu_0 \psi) |D| \big\{ (G^\mu_0 \psi) (\J \zetaa)\big\}\big) \Big)_{L^2}\\
    &\quad=\tfrac12\big( (\J \zetaa) \Delta\psi ,  \zetaa\big)_{L^2}- \tfrac12 \big( [\J,\nabla\psi  ]\nabla\zetaa ,\zetaa \big)_{L^2}\\
    &\qquad - \tfrac\epsilon2\big(  (\J \zetaa) \Delta\psi ,  (G^\mu_0 \J \zeta) \J  \zetaa \big)_{L^2} +  \tfrac\epsilon2\big(  (\J \zetaa) \nabla\psi , (\J  \zetaa ) \nabla(G^\mu_0 \J \zeta) \big)_{L^2}\\
    &\qquad- \tfrac{\epsilon^2}2\big(  (\J \zetaa) \Delta\psi , \J \big((G^\mu_0 \psi) |D| \big\{ (G^\mu_0 \psi) (\J \zetaa)\big\}\big) \big)_{L^2} \\
    &\qquad+  \tfrac{\epsilon^2}2\big(  (\J \zetaa) \nabla\psi , \J \big(\{\nabla(G^\mu_0 \psi) \}|D| \big\{ (G^\mu_0 \psi) (\J \zetaa)\big\}\big)  \big)_{L^2}\\
    &\qquad+  \tfrac{\epsilon^2}2\big(  (\J \zetaa) \nabla\psi ,\J \big((G^\mu_0 \psi) |D| \big\{ \{\nabla (G^\mu_0 \psi)\} (\J \zetaa)\big\}\big) \big)_{L^2}\\
       &\qquad-  \tfrac{\epsilon^2}2\big(  \zetaa, \J  \big[ \J \big((G^\mu_0 \psi) |D| \big\{ (G^\mu_0 \psi) \bullet \big\}\big), \nabla\psi\big]   \cdot\nabla\J \zetaa\big)_{L^2} .
\end{align*}
The first two contributions are easily estimated, using \Cref{L.commut_order0} with $s=0$, as 
\[ \big\vert \star \big\vert \leq C  \Ni{J}  \norm{\nabla\psi}_{H^{t_{0}+1}}\norm{\zetaa}_{L^2}^2.\]
The next two are straightforward by the continuous embedding $H^{t_0}(\RR^d)\subset L^\infty(\RR^d)$:
\[ \big\vert \star \big\vert \leq \epsilon^2\, C \, \No{J}^3 \, \norm{\nabla\psi}_{H^{t_{0}+1}}\norm{\zeta}_{H^{t_{0}+2}}\norm{\zetaa}_{L^2}^2.\]
For the next three, we use the regularizing effect of $\J$ and obtain by a repeated use of the product estimates in \Cref{L.product}
\[ \big\vert \star \big\vert \leq \epsilon^2\, C \, \No{J}\N{\frac12}{J}^2 \, \norm{\nabla\psi}_{H^{t_{0}+1}}^3 \norm{\zetaa}_{L^2}^2.\]
Finally, for the last term, we decompose
\begin{align*}
\J \big[ \J \big((G^\mu_0 \psi) |D| \big\{ (G^\mu_0 \psi) \bullet \big\}\big), \nabla\psi\big]   \nabla\J \zetaa 
&= \J^2 \big((G^\mu_0 \psi)  \big[ |D| , \nabla\psi\big] \big\{ (G^\mu_0 \psi)   \nabla\J \zetaa \big\}\big)\\
&\quad +\J \big[ \J , \nabla\psi ]\big((G^\mu_0 \psi) |D| \big\{ (G^\mu_0 \psi)  \nabla\J \zetaa \big\}\big).
\end{align*}
We estimate the right-hand side in $L^2(\RR^d)$ thanks the regularizing effect of $\J$, and the commutator estimates in \Cref{L.commut_order1}  (see also \Cref{L.commut_Tmu}) and \Cref{L.commut_order0} with $s=-1/2$, and again product estimates in \Cref{L.product}, which yields by Cauchy-Schwarz inequality
\[ \big\vert \star \big\vert \leq \epsilon^2\, C \, \Ni{J}\N{\frac12}{J}^2 \, \norm{\nabla\psi}_{H^{t_{0}+1}}^3\norm{\zetaa}_{L^2}^2.\]
Collecting the above, we find
\begin{equation}\label{est-C21}
\big\vert \epsilon \big( \nabla\cdot((\J \zetaa)\nabla\psi) , \mfa_{\J}[\epsilon \zeta, \epsilon\nabla\psi]  \zetaa\big)_{L^2}  \big\vert\leq \epsilon\, C_{M} \norm{\nabla\psi}_{H^{t_{0}+1}} \Big( 1 + \epsilon^2 \N{\frac12}{J}^2 \norm{\nabla\psi}_{H^{t_{0}+1}}^2\Big)\norm{\zetaa}_{L^2}^2 \,.
\end{equation}

Next, we have
\[\big(\J \left( \nabla \psi \cdot \nabla \psia \right) , G^\mu_0 \psia \big)_{L^2} =\tfrac12\big( [\J , \nabla \psi] \cdot \nabla \psia , G^\mu_0 \psia \big)_{L^2} -\tfrac12\big( [ \T^\mu , \nabla \psi] \cdot \J \nabla \psia , \nabla \psia \big)_{L^2} ,\]
where we denote $\T^\mu \eqdef -\frac{\tanh(\sqrt\mu|D|)}{|D|}\nabla $. 
Using \Cref{L.control_P} and the commutator estimates in \Cref{L.commut_order0} with $s=1/2$ (see also \Cref{L.commut_Tmu}) we find
\begin{equation}\label{est-C22}
\big\vert \epsilon \big(\J \left( \nabla \psi \cdot \nabla \psia \right) , G^\mu_0 \psia \big)_{L^2} \big\vert \leq \epsilon\, C  \Ni{J} \norm{\nabla\psi}_{H^{t_{0}+1}} \norm{\mfP\psia}_{L^2}^2.
\end{equation}

{\em Contributions in~\eqref{C3}.} We have by \Cref{L.product},
\[\norm{\mfa_{\J}[\epsilon \zeta, \epsilon\nabla\psi]  \zetaa }_{L^2}\leq \norm{\zetaa}_{L^2} \big(1+\epsilon\No{J}^2 \norm{\zeta}_{H^{t_0+1}}
+\epsilon^2 \N{\frac12}{J}^2 \norm{\nabla\psi}_{H^{t_0}}^2\big)\] 
and hence by the estimate for $\widetilde R_{1}^{(\alpha)}\in L^2(\RR^d)$ displayed in \Cref{P.quasi},
\begin{equation}\label{est-C31} 
\big\vert \epsilon \big( \widetilde R_{1}^{(\alpha)}, \mfa_{\J}[\epsilon \zeta, \epsilon\nabla\psi]  \zetaa)_{L^2} \big\vert \leq \epsilon \, C_{M} \norm{ \zeta}_{H^{N}} \big( \norm{ \nabla \psi }_{H^{t_{0}+1}} \norm{ \zeta}_{H^{N}} + \norm{  \zeta }_{H^{t_{0}+2}} \norm{ \nabla \psi }_{H^{N-1}} \big). 
\end{equation}
Finally, using \Cref{L.control_P} and the estimates for $\widetilde R_{2}^{(\alpha)},\widetilde R_{3}^{(\alpha)}\in H^{1/2}(\RR^d)$ in \Cref{P.quasi},
\begin{multline}\label{est-C32} 
\big\vert \big( \epsilon\widetilde R_2^{(\alpha)}+ \epsilon^2 \widetilde R_{3}^{(\alpha)}, G^\mu_0 \psia\big)_{L^2} \big\vert\leq \epsilon\, C_{M} \Big( \big( \norm{ \zeta }_{H^{t_{0}+2}} + \norm{ \nabla \psi }_{H^{t_{0}+1}} \big) \big( \norm{ \nabla \psi }_{H^{N-\frac12}} + \norm{\psia }_{H^{\frac{1}{2}}}  \big) \\
+ \epsilon \N{\frac12}{J}  \norm{ \nabla \psi }_{H^{t_{0}}} \big( \norm{ \zeta }_{H^{t_{0}+2}} + \norm{ \nabla \psi }_{H^{t_{0}+1}} \big) \big(\norm{ \nabla \psi }_{H^{N-1}} +  \norm{ \zeta }_{H^{N}} \big)   \Big) 
 \times \norm{ \mfP \psia }_{L^2} .
\end{multline}
Finally we note that by \Cref{P.psi_controls},
\begin{align*} \norm{ \nabla \psi }_{H^{N-1}} + \sup_{|\beta|=N} \norm{\psi_{(\beta)} }_{H^{\frac{1}{2}}} &\leq C\, \big(1+\epsilon \No{J} \norm{ \nabla \psi }_{H^{t_{0}}}  \big) \, \cE^{N}(\zeta,\psi)^{1/2}\leq C_M \cE^{N}(\zeta,\psi)^{1/2},\\
\norm{ \nabla \psi }_{H^{N-\frac{1}{2}}} &\leq C\, \big(1+\epsilon \N{-\frac12}{J}\norm{ \nabla \psi }_{H^{t_{0}}}    \big)\cE^{N}(\zeta,\psi)^{1/2} ,
\end{align*}
and by~\eqref{RTpositive} 
\[ \norm{ \mfP \psia }_{L^2}  \leq   \cE^{N}(U)^{1/2}, \quad \norm{ \zetaa }_{L^2} \leq \big( \tfrac1{\mfa_\star} \cE^{N}(U)\big)^{1/2}  , \qquad \norm{ \zeta }_{H^{N}}   \leq \big(\max(1,\tfrac1{\mfa_\star})\cE^{N}(U)\big)^{1/2} . \]
The result is now a direct consequence~\eqref{eq.energy-est_low} and~\eqref{eq.energy-est} with~\eqref{est-C11}--\eqref{est-C32} and since
\[\N{\frac12}{J} \leq \dotN{\frac12}{J} + \No{J}.\]
The proof is complete.
\end{proof}
 
 We conclude this section with the following result showing that  the Rayleigh--Taylor condition,~\eqref{RTpositive}, propagates in time.

\begin{Lemma}\label{P.hyperbolicity_propagates}
Let $d\in\{1,2\}$, $t_{0}>\frac{d}{2}$, $N \in \NN$ with $N \geq t_{0}+2$, $M_J,M_U>0$. There exists $C>0$ such that for any $\epsilon \geq 0$, $\mu \geq 1$, rectifier $\J$ regularizing of order $-1/2$ with $\No{J}\leq M_J$, and any  $(\zeta,\psi)$ smooth solution to~\eqref{RWW2} on the time interval $I\subset \RR$ and satisfying 
\[ \epsilon M \eqdef  \sup_{t\in I}  \big( \epsilon\norm{ \zeta(t,\cdot) }_{H^{t_{0}+2}} + \epsilon\norm{ \nabla \psi(t,\cdot) }_{H^{t_{0}+1}} \big) \leq M_U\]
one has for any $f\in L^2(\RR^d)$ and any $t,t'\in I$,
\[   \Big| \left( f,\mfa_{\J}[\epsilon \zeta(t,\cdot), \epsilon\nabla\psi(t,\cdot)] f \right)_{L^2} -  \left( f,\mfa_{\J}[\epsilon \zeta(t',\cdot), \epsilon\nabla\psi(t',\cdot)] f  \right)_{L^2}  \Big|\leq   K \, |t-t'| \, \norm{f}_{L^2}^2 \]
with $K=C \! \times \! \big( \epsilon M  +\dotN{\frac12}{J}^2 (\epsilon M)^2 \big) $.
\end{Lemma}
\begin{proof}
We have
 \begin{equation}
 \left( f ,\mfa_{\J}[\epsilon \zeta(t,\cdot), \epsilon\nabla\psi(t,\cdot)]  f \right)_{L^2} = \left( f,\mfa_{\J}[\epsilon \zeta(t',\cdot), \epsilon\nabla\psi(t',\cdot)] f  \right)_{L^2} + \int_{t'}^{t} \left(f , [ \partial_t,\mfa_{\J}[\epsilon \zeta, \epsilon\nabla\psi]]  f \right)_{L^2}  \dd \tau
\end{equation}
and the result readily follows from the estimate~\eqref{est-C11}, and $\N{\frac12}{J} \leq \dotN{\frac12}{J} + \No{J}$. 
\end{proof}

\subsubsection{Proof of \texorpdfstring{\Cref{P.WP-large-time}}{Sec}}\label{S.WP-large-completion}

We shall now complete the proof of \Cref{P.WP-large-time}. By \Cref{P.WP-short-time} and \Cref{C.blowup}, we have the existence and uniqueness of $(\zeta,\psi)\in \cC((-T_\star,T^\star);H^N(\RR^d)\times \Hdot^{N+1/2}(\RR^d))$ maximal solution to~\eqref{RWW2} with initial data $(\zeta,\psi)\big\vert_{t=0}=(\zeta_0,\psi_0) \in H^N(\RR^d)\times \Hdot^{N+1/2}(\RR^d)$,  thus we only need to show that the solution is estimated as in the proposition on the prescribed time interval.

To this aim, we first consider $\chi:\RR^d\to\RR^+$ a smooth cut-off function (radial, infinitely differentiable, with compact support, and such that $\chi(\bxi)=1$ for $|\bxi|\leq 1$), and define for $n\in\NN$,  $(\zeta_0^n,\psi_0^n)\eqdef (\chi(\tfrac{D}{2^n})\zeta_0,\chi(\tfrac{D}{2^n})\psi_0)$. By construction, $(\zeta_0^n,\psi_0^n)\in \bigcap_{s\in\NN} (H^s(\RR^d) \times \Hdot^{s+\frac12}(\RR^d))$ and $(\zeta_0^n,\psi_0^n)\to (\zeta_0,\psi_0)$ in $H^N(\RR^d)\times \Hdot^{N+1/2}(\RR^d)$ as $n\to\infty$. By \Cref{P.WP-short-time}, and the blowup alternative in \Cref{C.blowup}, we have for any $n\in\NN$ the existence and uniqueness of $T_\star^n$, $T^\star_n>0$ and $(\zeta^n,\psi^n)\in \bigcap_{s\in\NN}\cC((-T_\star^n,T^\star_n);  H^s(\RR^d)\times \Hdot^{s+1/2}(\RR^d))$ maximal solution to~\eqref{RWW2} with initial data $(\zeta^n,\psi^n)\big\vert_{t=0}=(\zeta_0^n,\psi_0^n)$. Let us first remark that for any $f\in L^2$,
\begin{multline}
 \big|  \big( f, \mfa_{\J}[\epsilon \zeta_0^n, \epsilon\nabla\psi_0^n]  f\big)_{L^2} - \big( f, \mfa_{\J}[\epsilon \zeta_0, \epsilon\nabla\psi_0] f \big)_{L^2} \big|\\
 \leq C_{0}  \left( \epsilon \No{J}^2\norm{ \zeta_0^n-\zeta_0}_{H^{t_0 + 1}}  + \epsilon^2 \N{\frac12}{J}^2 \norm{\nabla \psi_0^n+\nabla \psi_0}_{H^{t_0}}\norm{\nabla \psi_0^n-\nabla \psi_0}_{H^{t_0}} \right)\times \norm{f}_{L^2}^2, 
 \end{multline}
where we used \Cref{L.control_P} and product estimates in \Cref{L.product}, and $C_{0}$ depends uniquely on $t_0$. Hence by restricting to $n\geq n_\star$ with $n_\star$ sufficiently large, we have that $(\zeta_0^n,\psi_0^n)$ satisfy~\eqref{RTpositive} with coercivity factor $\mfa_\star/2$ (say). Similarly, augmenting $n_\star$ if necessary, we have for any $n\geq n_\star$
\begin{align*}    \norm{ \zeta_0^n }_{H^{t_{0}+\frac12}} &\leq 2 \norm{ \zeta_0 }_{H^{t_{0}+\frac12}} , &\quad \norm{ \mfP \psi_0^n }_{H^{t_{0}+\frac12}} &\leq 2 \norm{ \mfP \psi_0 }_{H^{t_{0}+\frac12}}  , \\
 \cE^N(\zeta_0^n,\psi_0^n) &\leq\sqrt{C} \cE^N(\zeta_0,\psi_0), &\quad  \cE^{\lceil t_0 \rceil +2}(\zeta_0^n,\psi_0^n) &\leq \sqrt{C} \cE^{\lceil t_0 \rceil +2}(\zeta_0,\psi_0),
 \end{align*}
where $C$ is prescribed in the assumptions of \Cref{P.WP-large-time}. In the rest of the proof, we assume that $n \geq n_\star$. We introduce now $C_{1}>1$ to be defined later, and the interval $I^n\subset \RR^+$ as the set of $T \geq 0$ such that,
\begin{equation}\label{conditions_prop3.2} \forall t\in [-T,T],\quad 
\norm{ \zeta^n(t,\cdot) }^2_{H^{t_{0}+2}} + \norm{ \nabla \psi^n(t,\cdot) }^2_{H^{t_{0}+1}}  \leq C_1 \cE^{ \lceil t_0 \rceil +2}(\zeta_0,\psi_0) = C_{1} M_{0}^2. 
\end{equation}
We need some preliminary estimates. Assume that $(\zeta^n,\psi^n)$ is defined on $[-T,T]$ for some $T>0$. We claim the following.
\begin{itemize}
\item[(a)] Let $t \in [-T,T]$. Assume that $\left( f,\mfa_{\J}[\epsilon \zeta^n(t,\cdot), \epsilon\nabla\psi^n(t,\cdot)] f \right)_{L^2} \geq (\mfa_\star/3) \norm{f}_{L^{2}}^{2}$ for any $f \in L^{2}$. There exists $C_{t_{0}}>0$, depending only on $t_{0}$, such that
\[
\norm{ \nabla \psi^n (t,\cdot) }_{H^{t_{0}}}^2 \leq C_{t_{0}} \cE^{\lceil t_{0} \rceil + 2}(\zeta^n(t,\cdot),\psi^n(t,\cdot)).
\]
Furthermore, if $\norm{ \nabla \psi^n(t,\cdot)}_{H^{t_{0}}}^2 \leq C_{t_{0}} C M_{0}^2$, there exists $\tilde C_{t_{0}}>0$, depending only on $t_{0}$, such that
\begin{multline*}
\norm{ \zeta^n(t,\cdot) }_{H^{t_{0}+2}}^2 + \norm{ \nabla \psi^n(t,\cdot) }_{H^{t_{0}+1}}^2 \leq \tilde C_{t_{0}} \left( 1 + \tfrac{3}{\mfa_\star} \right) \left(1 + \No{J} \sqrt{C_{t_{0}} C} M_{U} \right)\\
\times \cE^{\lceil t_{0} \rceil + 2}(\zeta^n(t,\cdot),\psi^n(t,\cdot)),
\end{multline*}
and, there exists $C_{2}>0$, depending only on $t_{0}$, $N$, $3/\mfa_\star$, $\sqrt{C} M_{U}$ and $\dotN{\frac12}{J}$ such that for any $t \in [-T,T]$ and any $N_\star \in \{ \lceil t_{0} \rceil +2 , N\}$,
\[
\tfrac{1}{C_{2}} \cE^{N_\star}(\zeta^n(t,\cdot),\psi^n(t,\cdot)) \leq \norm{\zeta^n(t,\cdot)}_{H^{N_\star}}^{2} + \norm{ \mfP \zeta^n(t,\cdot)}_{H^{N_\star}}^{2} \leq C_{2} \cE^{N_\star}(\zeta^n(t,\cdot),\psi^n(t,\cdot)).
\]
\item[(b)]  Assume that $\norm{ \zeta^n(t,\cdot) }_{H^{t_{0}+2}}^2 + \norm{ \nabla \psi^n(t,\cdot) }_{H^{t_{0}+1}}^2 \leq C_{1} M_{0}^2$ for any $t \in [-T,T]$. Then there exists $C_{3}>0$, depending only on $t_{0}$, $\No{J}$ and $\sqrt{C_{1}} M_{U}$ such that, for any $t \in [-T,T]$ and $f \in L^{2}$,
\[
\left( f,\mfa_{\J}[\epsilon \zeta^n(t,\cdot), \epsilon\nabla\psi^n(t,\cdot)] f \right)_{L^2} \geq \frac{\mfa_\star}{2} \norm{f}_{L^{2}}^2 -  C_{3} \times \big( \sqrt{C_{1}} \epsilon M_{0} + \dotN{\frac12}{J}^2 C_{1} (\epsilon M_{0})^2 \big) |t| \norm{f}_{L^2}^2.
\]
Furthermore, if  $\left( f,\mfa_{\J}[\epsilon \zeta^n(t,\cdot), \epsilon\nabla\psi^n(t,\cdot)] f \right)_{L^2} \geq (\mfa_\star/3) \norm{ f }_{L^{2}}^{2}$ for any $t \in [-T,T]$ and any $f \in L^{2}$, there exists $C_{4}>0$, depending only on $t_{0}$, $N$, $\mfa_\star$, $\Ni{J}$ and $\sqrt{C_{1}} M_{U}$, such that, for any $t \in [-T,T]$ and any $N_\star \in \{ \lceil t_{0} \rceil +2 , N\}$,
\[
\cE^{N_\star}(\zeta^n(t,\cdot),\psi^n(t,\cdot)) \leq \exp \big(  C_{4} \big(\sqrt{C_{1}} \epsilon M_{0} + \dotN{\frac12}{J}^{2} C_{1} (\epsilon M_{0})^2 \big) |t| \big) \sqrt{C} \cE^{N_\star}(\zeta_0,\psi_0).
\]
\end{itemize}
Estimates~(a) follow from \Cref{P.psi_controls} and the definition of the energy in~\eqref{eq.def-E}. Estimates~(b) follow from \Cref{P.hyperbolicity_propagates} and \Cref{P.energy_est} (using that $(\zeta^{n},\psi^{n})$ is smooth). Using the previous notations, we can now define $C_{1}$  such that
\[
C_{1} = 2 C \tilde C_{t_{0}} \left( 1 + \frac{3}{\mfa_\star} \right) \left( 1 + \No{J} \sqrt{C_{t_{0}} C} M_{U} \right).
\]
With such definition of $C_{1}$, we have $0 \in I^{n}$ from Estimates (a). We now introduce $T_{2}$ as the largest time such that 
\[
C_{4} (\sqrt{C_{1}} \epsilon M_{0} + \dotN{\frac12}{J}^{2} C_{1} (\epsilon M_{0})^2 \big) T_{2} \leq \ln \big( C^{1/2} \big) \text{ and } C_{3} \big( \sqrt{C_{1}}  \epsilon M_{0} + \dotN{\frac12}{J}^2 C_{1} (\epsilon M_{0})^2 \big) T_{2} \leq \frac{\mfa_\star}{6}.
\]
We claim that $T^n\eqdef \max(I^n)\geq T_2$. We argue by contradiction and assume that $T_{2} > T^n$. Notice that $T^n>0$ and is well-defined (that is $\sup(I^n)\in I^n$ when $\sup(I^n)<\infty$) using the continuity in time of the solution with arbitrary smoothness in space, provided by \Cref{C.blowup}  and the prescribed bound on $( \zeta^n(t,\cdot) , \psi^n(t,\cdot) ) \in H^{t_{0}+2} \times \Hdot^{t_{0}+2} $.
Using Estimates (b) and then (a), we get successively that for any $t \in [0,T^{n}]$ and any $f \in L^{2}$,
\begin{align*}
&\left( f,\mfa_{\J}[\epsilon \zeta^n(t,\cdot), \epsilon\nabla\psi^n(t,\cdot)] f \right)_{L^2} \geq \frac{ \mfa_\star }3 \norm{f}_{L^{2}}^2, & & \cE^{\lceil t_{0} \rceil +2}(\zeta^n(t,\cdot),\psi^n(t,\cdot)) \leq C \cE^{\lceil t_{0} \rceil +2 }(\zeta_{0},\psi_{0}),\\
&\norm{ \nabla \psi^n(t,\cdot) }_{H^{t_{0}}}^2 \leq C C_{t_{0}} M_{0}^{2} ,\quad \text{  and  } & & \norm{ \zeta^n(t,\cdot) }_{H^{t_{0}+2}}^2 + \norm{ \nabla \psi^n(t,\cdot) }_{H^{t_{0}+1}}^2 \leq  \frac12C_1 M_{0}^2.
\end{align*}
Using again \Cref{C.blowup} we obtain a time $T>T^n$ such that $T \in I^n$, which is a contradiction.

It remains to pass to the limit. Thanks to \Cref{R.WP}, there exists $T>0$, depending uniquely on $t_0$, $N$, $\epsilon$, $\mu$, $\N{1}{J}$ and a prescribed bound on the initial data, such that if $(\tilde\zeta_0^n,\tilde\psi_0^n)\to (\tilde\zeta_0,\tilde\psi_0)$ in $H^N(\RR^d)\times \Hdot^{N+1/2}(\RR^d) $, then the emerging solution is defined on the time interval $[-T,T]$ and one has $\Phi((\tilde\zeta_0^n,\tilde\psi_0^n))\to \Phi( (\tilde\zeta_0,\tilde\psi_0))$ in $\cC([-T,T] ;H^N(\RR^d)\times \Hdot^{N+1/2}(\RR^d) ) $. 
Notice that we have a uniform bound for $(\zeta^n,\psi^n)\in \cC([-T_2,T_2] ;H^N(\RR^d)\times \Hdot^{N+1/2}(\RR^d) ) $, by~\eqref{conditions_prop3.2} together with the last inequality of Estimates (a). This provides a lower bound on $T>0$ which allows to show after a finite number of iterations (and thanks to the uniqueness of the solution to the Cauchy problem) that $(\zeta^n,\psi^n)\to (\zeta,\psi)$ in $ \cC([-T_2,T_2] ;H^N(\RR^d)\times \Hdot^{N+1/2}(\RR^d) ) $. In particular, $(\zeta,\psi)$ satisfy the desired energy control on $[-T_2,T_2]$ thanks to the second inequality of Estimates (b).
The proof of \Cref{P.WP-large-time} is complete.

\subsection{Global-in-time well-posedness}\label{S.WP-global}

We conclude this section with the following result showing the global-in-time existence of solutions for sufficiently small initial data.

\begin{Proposition}\label{P.WP-global-in-time}
 Let $\r<-(\frac32+\frac{d}2)$ and $C>1$. There exists $\epsilon_0>0$  such that for any ${\mu\geq 1}$, $\epsilon>0$, any $\delta>0$ and $\J^\delta=J(\delta D)$ being regularizing of order $\r$ (see \Cref{D.J}), and for any ${(\zeta_0,\psi_0)\in  L^2(\RR^d)\times \Hdot^{1/2}(\RR^d)}$ such that
 \begin{equation}\label{eq.initial-data-small}
\epsilon \N{-\r}{J(\delta\cdot)} \big( \norm{\zeta_0}_{L^2}^2+\norm{\mfP\psi_0}_{L^2}^2\big)^{1/2}\leq \epsilon_0,
\end{equation}
there exists a unique $(\zeta,\psi)\in \cC(\RR;L^2(\RR^d)\times \Hdot^{1/2}(\RR^d))$ global-in-time solution to~\eqref{RWW2} with initial data ${(\zeta,\psi)\id{t=0}=(\zeta_0,\psi_0)}$. Moreover for any $t\in \RR$ one has
\begin{equation}\label{eq.control-H}
\tfrac1C\, \big( \tfrac12\norm{\zeta(t,\cdot)}_{L^2}^2+\tfrac12\norm{\mfP\psi(t,\cdot)}_{L^2}^2\big) \leq  \cH^\mu_{\J^\delta}(\zeta(t,\cdot),\psi(t,\cdot))=\cH^\mu_{\J^\delta}(\zeta_0,\psi_0) \leq C  \big(\tfrac12 \norm{\zeta_0}_{L^2}^2+\tfrac12\norm{\mfP\psi_0}_{L^2}^2\big)
\end{equation}
where we recall that
\[\cH^\mu_{\J^\delta}(\zeta,\psi)\eqdef\frac12\int_{\RR^d} \zeta^2+ (\mfP\psi)^2  + \epsilon (\J^\delta\zeta) \left( |\nabla\psi|^2-( G^\mu_0 \psi)^2\right)\dd \bx.\]
\end{Proposition}
\begin{Remark}
In the second equation of~\eqref{RWW2}, the meaning of quadratic terms for low-regularity functions $\psi \in\Hdot^{1/2}(\RR^d)$ is clarified by \Cref{L.compensation-tanh}.
\end{Remark}
\begin{proof} First we recall that $\cH^\mu_{\J^\delta}(\zeta,\psi)$ is an invariant of~\eqref{RWW2}, in the sense that for any $(\zeta,\psi)$ smooth solutions to~\eqref{RWW2} defined on a time interval $I\subset \RR$, $t\mapsto \cH^\mu_{\J^\delta}(\zeta(t,\cdot),\psi(t,\cdot))$ is constant on $I$. This is readily checked by computing
\begin{multline*} \frac{\dd }{\dd t}\cH^\mu_{\J^\delta}(\zeta(t,\cdot),\psi(t,\cdot)) = \big(\zeta,\partial_t\zeta\big)_{L^2} +\big(\mfP\psi, \partial_t\mfP\psi\big)_{L^2}+\frac12 \big(\epsilon (\J^\delta\partial_t\zeta) \left( |\nabla\psi|^2-( G^\mu_0 \psi)^2\right) \big)_{L^2}\\
+ \big((\epsilon \J^\delta\zeta)\left( (\nabla\psi\cdot\nabla\partial_t\psi)-( G^\mu_0 \psi)( G^\mu_0 \partial_t\psi)\right)\big)_{L^2},
\end{multline*}
replacing time derivatives with the formula provided by~\eqref{RWW2} and using suitable integration by parts or Parseval's theorem (see \Cref{L.control_P}) to infer that $\frac{\dd }{\dd t}\cH^\mu_{\J^\delta}(\zeta(t,\cdot),\psi(t,\cdot))=0$.
The result for $(\zeta,\psi)\in \cC((-T_\star,T^\star);L^2(\RR^d)\times \Hdot^{1/2}(\RR^d))$ maximal solution with initial data ${(\zeta,\psi)\id{t=0}=(\zeta_0,\psi_0)}$, defined by \Cref{P.WP-short-time}, follows by the density of the Schwartz space into Sobolev spaces,  and the continuity of the solution map (see \Cref{R.WP}).

Then we remark that, by \Cref{L.compensation-tanh} and Cauchy-Schwarz inequality, and then \Cref{L.control_P} and Young's inequality, we have for any $(\zeta,\psi)\in L^2(\RR^d)\times \Hdot^{1/2}(\RR^d)$,
\begin{align*} \big| \cH^\mu_{\J^\delta}(\zeta,\psi)-\tfrac12\norm{\zeta}_{L^2}^2 -\tfrac12\norm{\mfP\psi}_{L^2}^2  \big|&\leq \epsilon C_\r \norm{\zeta}_{L^2} \big( \N{-\r}{J(\delta\cdot)} \norm{\nabla\psi}_{H^{-1/2}}^2\big)\\
&\leq \epsilon C_\r \N{-\r}{J(\delta\cdot)} \big( \norm{\zeta}_{L^2}^2+\norm{\mfP\psi}_{L^2}^2 \big)^{3/2} ,
\end{align*}
where $C_\r$ depends only on $\r<-(1+\frac{d}2)$. By choosing $\epsilon_0>0$ such that 
$\epsilon_0 C_\r <\frac{C-1}{C}$ 
 we infer that for all initial data satisfying~\eqref{eq.initial-data-small}, 
$ \big\{t\in[0,T^\star),~\eqref{eq.control-H} \text{ holds} \big\}$
is an open subset of $[0,T^\star)$. Since it is also closed and non-empty, we obtain that~\eqref{eq.control-H} holds on $[0,T^\star)$. If $T^\star<\infty$, we may use \Cref{P.WP-short-time} with an ``initial'' time sufficiently close to $T^\star$ and the control provided by~\eqref{eq.control-H} to construct $\tilde T^\star>T^\star$ and $(\tilde\zeta,\tilde\psi)\in \cC([0,\tilde T^\star];L^2(\RR^d)\times \Hdot^{1/2}(\RR^d))$ satisfying $(\tilde\zeta,\tilde\psi)\big\vert_{[0,T^\star)}=(\zeta,\psi)\big\vert_{[0,T^\star)}$ and~\eqref{RWW2} on the time interval $[0,\widetilde T^\star]$. This brings the desired contradiction, and proves that $T^\star=+\infty$. Symmetrically, $-T_\star=-\infty$ and the proof is complete.
\end{proof}

\appendix

\section{The  water waves system}\label{S.WW}

Let us recall a few facts on the water waves system, describing the motion of inviscid, incompressible and homogeneous fluids with a free surface. Under the assumption of potential flow, and assuming that the bottom is flat, the evolution equations can be rewritten (with dimensionless variables set accordingly to the deep water regime, following the convention of~\cite[Appendix~A]{Alvarez_Lannes} with $\nu=\frac{1}{\sqrt{\mu}}$) as 
\begin{equation}\label{WW}\tag{WW}
\left\{\begin{array}{l}
\partial_t\zeta - G^\mu[\epsilon\zeta]\psi=0\\[1ex]
\partial_t\psi+\zeta+\frac{\epsilon}{2} |\nabla\psi|^2-\frac{\epsilon }{2} \frac{ \left( G^\mu[\epsilon\zeta]\psi+\epsilon\nabla\zeta\cdot\nabla\psi \right)^2}{1+\epsilon^2|\nabla\zeta|^2}=0,
\end{array}\right.
\end{equation}
where $(\zeta,\psi):(t,\bx)\in \RR\times \RR^d\to\RR^2$ (with $d\in\{1,2\}$), $\epsilon,\mu>0$ and  the Dirichlet-to-Neumann operator, $G^\mu$, is defined for sufficiently nice functions as
\[ G^\mu[\epsilon\zeta]\psi \eqdef (\partial_z \Phi)\big\vert_{z=\epsilon \zeta}  -  \epsilon \nabla\zeta\cdot(\nabla_\bx \Phi) \big\vert_{z=\epsilon \zeta} \]
with $\Phi$ being the solution to the Laplace problem
\begin{equation*}
\left\{\begin{array}{l}
\Delta_\bx \Phi + \partial^{2}_{z} \Phi =0 \quad \text{ on   } \{(\bx,z) \in \RR^{d+1},\ -\sqrt{\mu} < z < \epsilon \zeta(\bx)\},\\[1ex]
\Phi_{|z=\epsilon \zeta} = \psi ,\quad  \partial_{z} \Phi_{|z=-\sqrt{\mu}} = 0.
\end{array}\right.
\end{equation*}

Basic properties about the Dirichlet-to-Neumann operator can be found in~\cite[Chapter 3]{Lannes_ww} or~\cite[Section 3]{Alvarez_Lannes}. The only result that we use in this paper is the following asymptotic expansion adapted from Proposition 3.9 and Proposition 3.3 in~\cite{Alvarez_Lannes}. Recall that $G^\mu_0\eqdef |D|\tanh(\sqrt\mu|D|)$.

\begin{Proposition}\label{asymptotic_DN}
Let $d \in \{1,2\}$, $s \geq 0$, $t_{0}>\frac{d}{2}$ and $h_\star>0$ and $M>0$. There exists $C>0$ such that for any  $\epsilon>0$, $\mu \geq 1$ and  $\zeta \in H^{\max(s+\frac{3}{2},t_{0}+2)}$ such that 
\[ 1+\tfrac{\epsilon}{\sqrt\mu} \zeta \geq h_\star, \qquad \norm{\epsilon\zeta}_{H^{t_0+2}}\leq M,\] 
and for any $\psi \in \Hdot^{1+s}(\RR^d)$, we have $G^\mu[\epsilon\zeta]\psi \in H^s(\RR^d)$ and
\[ G^\mu[\epsilon\zeta]\psi - G^\mu_0 \psi - \epsilon \left( G^\mu_0 \left( \zeta G^\mu_0 \psi \right) + \nabla \cdot( \zeta \nabla \psi) \right) = \epsilon^2 R  \]
where
\[ \norm{ R }_{H^{s}} \leq C  \norm{ \zeta }_{H^{t_{0}+1}} \left( \norm{ \zeta }_{H^{t_{0}+1}} \norm{ \mfP \psi }_{H^{s+\frac12}} + \norm{ \zeta }_{H^{s+\frac{3}{2}}} \norm{ \mfP \psi }_{H^{t_{0}}}  \right). \]
\end{Proposition}
The asymptotic expansion above provides the foundation for the rigorous justification of~\eqref{WW2-intro} as an asymptotic model for~\eqref{WW} with precision $\cO(\epsilon^2)$, in the sense of consistency. 
The precise statement is displayed in \Cref{T.Consistency} (see the discussion below).

\section{A toy model}\label{S.toy}

 In order to describe the high-frequency instability mechanism which we diagnose in this work, and the features of the proposed regularization, we consider in this section the following toy model
\begin{equation}\label{RWW2-toy}\left\{\begin{array}{l}
\partial_t\zeta- G^\mu_0 \psi=0,\\[1ex]
\partial_t \psi +\zeta- \epsilon^2 \alpha[\psi]  (\J^\delta)^2  |D|\zeta=0,
\end{array}\right.\end{equation}
where we recall that $G^\mu_0\eqdef |D|\tanh(\sqrt\mu|D|)$ with $\mu\geq 1$, $\delta>0$ and $\J^\delta=J(\delta |D|)$ is a Fourier multiplier associated with a real-valued symbol $J$. Finally, we set
\[ \alpha[\psi] \eqdef \int(G^\mu_0 \psi)^2 \dd\bx.\]
System~\eqref{RWW2-toy} is inspired by \Cref{P.quasi} and specifically~\eqref{eq.quasi_eq1}--\eqref{eq.quasi_eq2}. It mimics the (possible) destabilization effect on the high-frequency component of solutions to~\eqref{RWW2}  
stemming from the lack of coercivity of the operator $\mfa_{\J^\delta}[\epsilon \zeta, \epsilon\nabla\psi]  $ defined in~\eqref{eq.def-RT}, while disregarding other contributions such as advection terms, bounded operators in $\mfa_{\J^\delta}$, and the choice of Alinhac's good unknowns.

In the following discussion, we consider the $(2\pi\ZZ)^d$-periodic framework for convenience, although our the results can be adapted to the Euclidean space framework, $\RR^d$. We can hence rewrite~\eqref{RWW2-toy}, using the decomposition in Fourier series, as an infinite system of ordinary differential equations. Specifically, we have for sufficiently regular real-valued solutions to~\eqref{RWW2-toy} defined on $(2\pi\TT)^d$,
\begin{equation}\label{RWW2-toy-Fourier}
\forall \bk\in\ZZ^d, \qquad 
\left\{\begin{array}{l}
\frac{\dd}{\dd t}\widehat \zeta_\bk- \tanh(\sqrt\mu|\bk|) |\bk|\widehat\psi_\bk=0,\\[1ex]
\frac{\dd}{\dd t}\widehat\psi_\bk +\Big(1 -  \epsilon^2 \alpha[\psi]  J(\delta|\bk|)^2  |\bk|\Big)\widehat\zeta_\bk=0,
\end{array}\right.\end{equation}
with 
\[ \alpha[\psi] =   (2\pi)^d  \sum_{\bk\in \ZZ^d} (\tanh(\sqrt\mu|\bk|) |\bk| |\widehat\psi_\bk|)^2 .\]
Here, we set the convention
\[\widehat \zeta_{\bk} =\frac1{ (2\pi)^d } \int_{(2\pi\TT)^d} \zeta(\bx) e^{-i \bk\cdot\bx}\dd\bx \quad ; \quad \zeta(\bx) = \sum_{\bk\in \ZZ^d }\widehat\zeta_{\bk} e^{i \bk\cdot\bx}.\]
In the following, we denote
\[ \ell^{2,s}\eqdef \big\{ a\eqdef   (a_\bk)_{\bk\in\ZZ^d}, \
  a_{-\bk}=\overline{a_\bk},\ 
 \textstyle \norm{a}_{\ell^{2,s}}^2\eqdef \sum_{\bk\in\ZZ^d} (1+|\bk|^2)^s |a_\bk|^2 <\infty\big\}\]
 and $H^{s} ((2\pi\TT)^d) $ the space of $(2\pi\ZZ)^d$-periodic distributions such that $ \widehat\zeta \in \ell^{2,s}$, 
 endowed with the norm $\norm{\zeta}_{H^s}\eqdef\norm{\widehat\zeta}_{\ell^{2,s}}$, and $L^2((2\pi\TT)^d) $ the $(2\pi\ZZ)^d$-periodic real-valued square-integrable functions.

 In~\eqref{RWW2-toy-Fourier}, the coupling between each mode arises only through the coefficient $\alpha[\psi] $. This is of course a simplistic model for nonlinear interactions, but we believe it captures the essence of the instability mechanism at stake in~\eqref{RWW2}. For the sake of the discussion, let us now simplify even further the problem by fixing $\alpha$ as a constant. Then the system is explicitly solvable. We observe that a plane wave solution with wave vector $\bk$ is stable (that is its amplitude remains bounded for all times) if and only if $ 1-\alpha \epsilon^2 J(\delta|\bk|)^2  |\bk|>0$. In the opposite situation, that is when $ 1-\alpha \epsilon^2 J(\delta|\bk|)^2  |\bk|<0$, the corresponding mode experiences an exponential growth with growth rate $c(|\bk|)=\big(\tanh(\sqrt\mu|\bk|) |\bk|(\alpha \epsilon^2 J(\delta|\bk|)^2  |\bk|-1)\big)^{1/2} $ (whereas the growth is linear in the critical situation, that is when $\alpha \epsilon^2 J(\delta|\bk|)^2  |\bk|=1$). 
 Notice that  if $\alpha>0$ and $\J^\delta=\Id$, then modes associated with sufficiently large $|\bk|$ are necessarily unstable, and that the growth rate is unbounded. As a consequence, the initial-value problem for the dynamical system~\eqref{RWW2-toy-Fourier} is strongly ill-posed in any polynomially weighted $\ell^{2,s}$ spaces in that case.
  On the contrary, if $\alpha \epsilon^2\lim\sup_{k\to \infty} (k J^2(\delta k))<1$, then the system is (globally) well-posed in $\ell^{2,s}\times  \ell^{2,s+\frac12}$.
 
 The following propositions show that this naive analysis describes fairly well the behavior of the toy model with variable $\alpha=\alpha[\psi]$. 

\begin{Proposition}[Local well-posedness in the sub-critical case]\label{P.toy-WP1}
Let $J:\NN\to \RR$ be such that 
\begin{equation}\label{eq.subcritical}
 k J^2(k)\to 0 \quad \text{ as } k\to\infty.
 \end{equation}
Let $M\geq 0$ and $\delta>0$. Then there exists $T_\delta>0$ such that for any $\mu\geq 1$, any $s\geq 3/4$ and any initial data $( \zeta_{0} , \psi_{0})\in  H^{s} ((2\pi\TT)^d)\times  H^{s+\frac12} ((2\pi\TT)^d) $ such that 
\[\epsilon \norm{ \zeta_{0}}_{H^{\frac12}} + \epsilon \norm{\nabla \psi_{0}}_{L^2}  \leq M ,\]
there exists a unique  $(\zeta, \psi)\in\cC([-T_{\delta},T_{\delta}]; H^{s} ((2\pi\TT)^d)\times  H^{s+\frac12} ((2\pi\TT)^d) )$ solution to~\eqref{RWW2-toy} with ${\J^\delta=J(\delta|D|)}$ and satisfying $(\zeta,\psi)\id{t=0}=(\zeta_0,\psi_{0})$, where
\[
T=\frac{T_\delta}{\big(\epsilon \norm{ \zeta_{0}}_{H^{\frac34}} + \epsilon \norm{\nabla \psi_{0}}_{H^{\frac14}}\big)^2 }.
\]
Moreover, if we assume $\displaystyle \sup_{k\in\NN}k^{3/2} J^2(k)<\infty$, then the above holds with $s\geq 1/2$ and
\[
T=\frac{T_0\, \delta^{3/2}}{\big(\epsilon \norm{ \zeta_{0}}_{H^{\frac12}} + \epsilon \norm{\nabla \psi_{0}}_{L^{2}}\big)^2 },
\]
where $T_0$ depends uniquely on $M$ and $\displaystyle \sup_{k\in\NN}k^{3/2} J^2(k)$.
\end{Proposition}
\begin{proof}
Let us first consider the case where only a finite number of modes, $\widehat \zeta_{\bk,0} ,\widehat \psi_{\bk,0} $, are non-zero. Then~\eqref{RWW2-toy-Fourier} is a finite system of ordinary differential equations. Cauchy-Lipschitz theorem applies, and defines uniquely a maximal local-in-time solution. On the maximal time of existence, we denote 
\[ a_\bk(t)= 1-\epsilon^2 \alpha(t)  J(\delta|\bk|)^2  |\bk|, \quad  \alpha(t)=(2\pi)^d    \sum_{\bk\in \ZZ^d} (\tanh(\sqrt\mu|\bk|) |\bk| |\widehat\psi_\bk(t)|)^2,\]
and 
\[\cE_s(t)\eqdef \sum_{\bk \in\ZZ^d} \langle\bk\rangle^{2s} \Big( |\widehat \zeta_{\bk}|^2(t) +\tanh(\sqrt\mu|\bk|) |\bk| |\widehat\psi_{\bk,0}|^2(t)\Big) \approx \norm{\widehat\zeta}_{\ell^{2,s}}^2+\norm{\widehat\psi}_{\ell^{2,s+\frac12}}^2  \]
(we assume above and henceforth that $\widehat\psi_{\bk,0}=0$, the general case follows after subtracting its mean-value to $\psi$).
In the following, we set $T\in(0,+\infty]$ the maximal value such that 
\[ \forall t\in (-T,T), \quad \cE_{1/2}(t)\leq 4\cE_{1/2}(0), \quad  \cE_{3/4}(t)\leq 2\cE_{3/4}(0) .\]
In particular, we have  $\alpha(t)\leq 4(2\pi)^d\cE_{1/2}(0)$ for $t\in(-T,T)$, and by~\eqref{eq.subcritical} we can choose
$k_\star\in\NN$ such that 
\begin{equation}\label{eq.condition-1}
	\forall k\geq k_\star, \qquad 4(2\pi)^d \epsilon^2 \cE_{1/2}(0)\, k J^2(\delta k) \leq 1/2.
\end{equation}
Hence we have
\[\begin{cases}
1/2\leq a_\bk(t)\leq 1 & \text{ if } |\bk|\geq k_\star,\\
0\leq  1-a_\bk(t) \leq  4(2\pi)^d \epsilon^2 \cE_{1/2}(0)  \max_{k\in\NN} k J^2(\delta k) &\text{ if } |\bk|< k_\star.
\end{cases}
\]
We first estimate the low-frequency modes, $|\bk|< k_\star$. By~\eqref{RWW2-toy-Fourier} and since $\tanh(\sqrt\mu|\bk|)\leq 1$ we infer
\[ \tfrac{\dd}{\dd t} \big(  | \widehat \zeta_\bk|^2(t) +\tanh(\sqrt\mu|\bk|) |\bk| |\widehat\psi_\bk|^2(t) \big) \leq 2 |\bk| |a_\bk(t)-1| |\widehat \zeta_\bk|(t)|\widehat\psi_\bk|(t)\ \]
and hence, using Gronwall's inequality, we find that there exists $C_1>0$ depending only on $ \max_{k\in\NN} k J^2(\delta k)$ such that for any  $|\bk|< k_\star$, and $t\in(-T,T)$,
\[   |\widehat \zeta_\bk|^2(t) +\tanh(\sqrt\mu|\bk|) |\bk| |\widehat\psi_\bk|^2(t) \leq \big(  |\widehat \zeta_{\bk,0}|^2 +\tanh(\sqrt\mu|\bk|) |\bk| |\widehat\psi_{\bk,0}|^2 \big) e^{C_1\, \epsilon^2\cE_{1/2}(0)\, k_\star^{1/2} \, |t|}.\]
In particular, restricting to $t\in (-T_1,T_1)$ with $T_1=\ln(\frac32) (C_1\, \epsilon^2\cE_{1/2}(0)\, k_\star^{1/2})^{-1}$ if necessary, we find that for any $s\in\RR$,
\[\sum_{|\bk|< k_\star} \langle\bk\rangle^s \Big(  |\widehat \zeta_\bk|^2 +\tanh(\sqrt\mu|\bk|) |\bk| |\widehat\psi_\bk|^2 \Big)(t) \leq \frac32\sum_{|\bk|< k_\star} \langle\bk\rangle^s \Big( |\widehat \zeta_{\bk,0}|^2 +\tanh(\sqrt\mu|\bk|) |\bk| |\widehat\psi_{\bk,0}|^2\Big).\]
We now consider the high-frequency modes, $ |\bk|\geq k_\star$. Testing the first equation in~\eqref{RWW2-toy-Fourier} with $a_\bk \overline{\widehat \zeta_\bk}$ and the second equation with $\tanh(\sqrt\mu|\bk|) |\bk| \overline{\widehat\psi_\bk}$, we find
\[\tfrac{\dd}{\dd t}  \Big(a_\bk |\widehat \zeta_\bk|^2(t) + \tanh(\sqrt\mu|\bk|)|\bk||\widehat\psi_\bk|^2(t)\Big) = a_\bk'(t) |\widehat \zeta_\bk|^2(t) = -\epsilon^2 \alpha'(t)  J(\delta|\bk|)^2  |\bk| |\widehat \zeta_\bk|^2(t) ,\]
with
\[\alpha'(t)=  -  (2\pi)^d    \sum_{\bk\in \ZZ^d} 2 (\tanh(\sqrt\mu|\bk|) |\bk|)^2  a_\bk(t) \Re( \widehat \zeta_\bk\overline{\widehat\psi_\bk} )(t).\]
Hence,   since $\tanh(\sqrt\mu|\bk|)\leq 1$ and using Cauchy-Schwarz inequality, we have for all $t\in(-T,T)$
\[ |\alpha'(t)|\leq (2\pi)^d   \sum_{\bk\in \ZZ^d}  | a_\bk(t)||\bk|^{3/2}\left(  |\widehat \zeta_\bk|^2(t) + \tanh(\sqrt\mu|\bk|)|\bk||\widehat\psi_\bk|^2(t)\right) \leq C_2  \cE_{3/4}(0)  \]
where $C_2$ depends uniquely on $\epsilon^2 \cE_{1/2}(0)  \max_{k\in\NN} k J^2(\delta k)$. In particular, since $1/2\leq a_\bk(t)\leq 1 $ when $|\bk| \geq k_\star$ and restricting to $t\in (-T_2,T_2)$ with $T_2=\ln(\frac54) (C_2\, \epsilon^2\cE_{3/4}(0)\, 2 \max_{k\in\NN} k J^2(\delta k))^{-1}$ if necessary,
we have for any $s\in\RR$,
\[\sum_{|\bk|\geq k_\star} \langle\bk\rangle^s \Big(  |\widehat \zeta_\bk|^2(t) +\tanh(\sqrt\mu|\bk|) |\bk| |\widehat\psi_\bk|^2(t) \Big)\leq \frac52\sum_{|\bk|\geq  k_\star} \langle\bk\rangle^s \Big( |\widehat \zeta_{\bk,0}|^2 +\tanh(\sqrt\mu|\bk|) |\bk| |\widehat\psi_{\bk,0}|^2\Big).\]

Combining the above, we find that for any $s\in\RR$, and $ t\in (-T_\star,T_\star)$ with  $t_\star=\min(T,T_1,T_2)$,
\begin{equation}\label{eq.control-alpha}
 \cE_s(t)\leq \frac52 \cE_s(0).
 \end{equation}
The standard continuity argument shows $T\geq \min(T_1,T_2)$ and $(\widehat\zeta_\bk, \widehat\psi_\bk)\in\cC([-T,T]; \ell^{2,s}\times  \ell^{2s+\frac12} )$  satisfies~\eqref{eq.control-alpha} on $[-T,T]$. There remains to notice that $T_\delta\eqdef \epsilon^2\cE_{3/4}(0) \min(T_1,T_2)$  depends only on $M$, $J$ and $\delta$ (in particular through the value of $k_\star$).

By using these uniform bounds one easily obtains the corresponding result for general initial data (that is with an infinite number of non-zero modes, $\widehat \zeta_{\bk,0} ,\widehat \psi_{\bk,0}$) by truncating Fourier modes and taking the limit. This concludes the first part of the proof.

The second part is immediate recalling the inequality valid for any $\bk\in\ZZ^d$
\[ \tfrac{\dd}{\dd t } \big(  | \widehat \zeta_\bk|^2 +\tanh(\sqrt\mu|\bk|) |\bk| |\widehat\psi_\bk|^2 \big)(t) \leq 2 |\bk| |a_\bk(t)-1| |\widehat \zeta_\bk|(t)|\widehat\psi_\bk|(t)\ \]
and the fact that for any $t$ in the maximal time interval,
\[ |\bk|^{1/2}|a_\bk(t)-1| = \epsilon^2| \alpha(t) |  J(\delta |\bk|)^2  |\bk|^{3/2} = 4(2\pi)^d \delta^{-3/2}\epsilon^2 \cE_{1/2}(0)  \max_{k\in\NN} (k^{3/2} J^2( k)) .\]
The proof is complete.
\end{proof}

\begin{Proposition}[Conditional well-posedness in the critical case]\label{P.toy-WP2}
Let $J:\NN\to \RR$ be such that 
\begin{equation}\label{eq.critical}
M_J\eqdef \limsup_{k\to\infty} (k J^2(k))<\infty.
\end{equation}
There exists $C_0>0$ and $T_0>0$ such for any $s\geq 3/4$, any $\mu\geq 1$ and any $\delta>0$, and for any initial data $( \zeta_{0} , \psi_{0})\in  H^{s} ((2\pi\TT)^d)\times  H^{s+\frac12} ((2\pi\TT)^d) $ such that 
\begin{equation} \label{eq.condition-2}
	\delta^{-1}\big( \epsilon^2 \norm{ \zeta_{0}}_{H^{\frac12}}^2 + \epsilon^2 \norm{\nabla \psi_{0}}_{L^2}^2 \big) M_J \leq C_0 ,
\end{equation}
there exists a unique $(\zeta, \psi)\in\cC([-T,T]; H^{s} ((2\pi\TT)^d)\times  H^{s+\frac12} ((2\pi\TT)^d) )$ solution to~\eqref{RWW2-toy} with ${\J^\delta=J(\delta|D|)}$ and satisfying $(\zeta,\psi)\id{t=0}=(\zeta_0,\psi_{0})$, where
\[T=\frac{T_0\,\delta}{\big( \epsilon \norm{ \zeta_{0}}_{H^{\frac34}} + \epsilon \norm{\nabla \psi_{0}}_{H^{\frac14}}\big)^2 }.\]
\end{Proposition}
\begin{proof}
The proof of \Cref{P.toy-WP2} is a direct adaptation of the proof of \Cref{P.toy-WP1}, noticing that the restriction~\eqref{eq.condition-2} allows to set $k_\star=1$ in~\eqref{eq.condition-1} if $C_0$ is set sufficiently small. We can then continue the proof and remark that $T_2= C_2\, (\delta^{-1}\epsilon^2\cE_{3/4}(0)\, \max_{k\in\NN} k J^2( k))^{-1}$ where $C_2$ depends only on $\delta^{-1}\epsilon^2\cE_{1/2}(0)\, \max_{k\in\NN} k J^2( k)\leq C_0$, which provides the desired control for $T=T_2$.
\end{proof}
 
\begin{Remark}[Time of existence] \label{R.toy-time}
	A combination of \Cref{P.toy-WP1,P.toy-WP2} offers a lower bound on the time of existence of solutions to~\eqref{RWW2-toy} as 
	\[ T\gtrsim \frac{\delta}{\epsilon^2 M_0^2}, \qquad \text{ where } M_0\eqdef \norm{ \zeta_{0}}_{H^{\frac34}} +  \norm{\nabla \psi_{0}}_{H^{\frac14}} \]
	 in the situation where $\displaystyle \sup_{k\in\NN}k^{3/2} J^2(k)<\infty$, and if we impose additionally the restriction $\delta \geq \epsilon M_0$.
	 Indeed, we can then apply \Cref{P.toy-WP2} for sufficiently small values of $\epsilon M_0$, while the last conclusion of \Cref{P.toy-WP1} provides $T\gtrsim \delta^{3/2}/(\epsilon^2M_0^2)\gtrsim (\epsilon M_0)^{1/2}\delta/ (\epsilon^2M_0^2)\gtrsim \delta/ (\epsilon^2M_0^2)$ when $\epsilon M_0\gtrsim 1$.
\end{Remark}

\begin{Proposition}[Ill-posedness in the super-critical case] \label{P.toy-IP}
Let $J:\NN\to \RR$ be such that 
\begin{equation}\label{eq.supercritical}
\limsup_{k\to\infty} (k J^2(k))= \infty.
\end{equation}
For all $\epsilon>0$ and $\delta>0$, there exists $ (\zeta_{0}^n, \psi_{0}^n)_{n\in\NN} \in\cC^\infty((2\pi\TT)^d)^\NN$ and $(T_n)_{n\in\NN}\in(\RR^+)^\NN$ such that
\[  \forall s\in\RR , \qquad  \norm{ \zeta_{0}^n}_{H^s} + \norm{\psi_{0}^n}_{H^s} \to 0 \quad \text{ and }  \quad T_n \searrow 0 \quad (n\to\infty),\]
and for any  $\mu\geq 1$ the solution  $(\zeta^n, \psi^n)$ to 
~\eqref{RWW2-toy} with  ${\J^\delta=J(\delta|D|)}$ and $(\zeta^n, \psi^n)\id{t=0}=(\zeta_0^n, \psi_0^n)$
satisfies
\[ \forall s'\in\RR , \qquad \norm{ \psi^n(t,\cdot)}_{H^{s'}}  \to  \infty \quad (t\nearrow T_n) .\]
The same result holds backwards in time.
\end{Proposition}

\begin{proof}

We use the initial data
\begin{equation} \label{eq.init-n}
\zeta_0^n(\bx)=0 \qquad  ; \qquad \psi_0^n(\bx)=b_n\cos(\bk_0\cdot \bx) + c_n \cos(\bk_n\cdot\bx) , 
\end{equation}
with $\bk_0\in \ZZ^d$ such that $k_0\eqdef |\bk_0|$ satisfies $k_0 \J^2(\delta k_0)\neq 0$, $\bk_n\in\ZZ^d$ such that $|\bk_n| =k_n$ where $(k_n)_{n \in \NN} \in \NN^{\NN}$ is a sequence such that $k_n\to\infty$ and $k_n\ J^2(\delta k_n)\to \infty$
 as $n\to\infty$, and $b_n,c_n>0$   will be determined later on.

Since~\eqref{RWW2-toy-Fourier} is a finite system of ordinary differential equations, the Cauchy-Lipschitz theorem applies, and defines uniquely maximal local-in-time solutions, which we denote $\widehat\zeta^n=( \widehat\zeta^n_{ \pm \bk_0},\widehat\zeta^n_{\pm \bk_n})$ and $\widehat\psi^n=( \widehat\psi^n_{ \pm \bk_0}, \widehat\psi^n_{ \pm \bk_n})$, on the maximal (forward-in-time) interval $I^{n}=[0,T^n_\star)$. On the maximal time of existence, we denote for $\bk\in\{ \pm  \bk_0,  \pm \bk_n\}$
\[ a_\bk^n(t)= 1-\epsilon^2 \alpha^n(t)  J(\delta|\bk|)^2  |\bk|, \quad  \alpha_\bk^n(t)=(2\pi)^d    (\tanh(\sqrt\mu|\bk|) |\bk| |\widehat\psi^n_\bk(t)|)^2.\]
Notice that for any $t\in I^n$, $\alpha_\bk^n(t)\geq 0$ and
\[ \alpha^n(t)= 2 \alpha_{\bk_0}^n(t) +2\alpha_{\bk_n}^n(t).\]
We define $\alpha_{0} =  \frac14 (2\pi)^{d} (\tanh(\sqrt{\mu} k_{0}) k_{0})^2$. Henceforth we assume that $b_n>0$ is such that
\begin{equation}\label{hyp_bn}
\tfrac14 b_{n}^2 \alpha_0  \epsilon^2J(\delta k_{n})^2 k_{n} > 1 \quad \text{   and   } \quad    b_n^2 < \tfrac25 (\alpha_{0} \epsilon^2 J(\delta k_{0})^2 k_{0})^{-1} .
\end{equation}

\noindent{\em Step 1: exponential growth.}\\
We have the following controls.
\begin{itemize}
\item[(a)] Assume there exists $t_{0} \in I^{n}$ such that $\alpha_{\bk_n}^{n}(t_{0}) \geq \frac14 b_{n}^{2}\alpha_0$ with $\widehat\zeta^n_{\bk_{n}}(t_{0}) \geq 0$ and $\widehat\psi^n_{\bk_{n}}(t_{0}) > 0$. Then $\widehat\zeta^n_{\bk_{n}}$ and $\widehat\psi^n_{\bk_{n}}$ are increasing on $I^{n} \cap [t_{0},+\infty)$ and for any $t \in I^{n} \cap [t_{0},+\infty)$, 
\begin{equation}
\widehat\psi^n_{\bk_{n}}(t) \geq \frac{\widehat\psi^n_{\bk_{n}}(t_{0})}{2} \exp \big( \sqrt{\tanh(\sqrt{\mu} k_{n}) k_{n}} (t-t_{0}) \big),
\end{equation}
\item[(b)] Assume there exists $t_{0} \in I^{n} \cap [0,k_{0}^{-\frac12}]$ such that $\alpha_{\bk_n}^{n} (t)\leq  \frac14 b_{n}^{2}\alpha_0 $ holds for any  $t\in[0,t_{0}]$. Then,  $\widehat\zeta^n_{\bk_{n}}$ and $\widehat\psi^n_{\bk_{n}}$ are increasing on $[0,t_{0}]$, and for any $t \in[0,t_{0}]$, 
\begin{equation}
\widehat\psi^n_{\bk_{n}}(t) \geq \frac{c_{n}}{4} \exp \big( \sqrt{\tanh(\sqrt{\mu} k_{n}) k_{n}} t \big).
\end{equation} 

\end{itemize}
First we prove controls (a). Since $\alpha^{n}(t_{0}) \geq 2 \alpha_{\bk_{n}}^n(t_{0}) \geq  \frac12 b_{n}^2 \alpha_0$ and using the first assumption in~\eqref{hyp_bn}, we get that $a_{\bk_n}^{n}(t_{0}) < -1$. By continuity, for $t$ close enough to $t_{0}$, we have on $[ t_{0} ,t]$
\[ \frac{\dd}{\dd t}\widehat \zeta_{\bk_n}^n =\tanh(\sqrt\mu k_n) k_n \widehat \psi_{\bk_n}^n \quad ; \quad \frac{\dd}{\dd t}\widehat \psi_{\bk_n}^n \geq \widehat \zeta_{\bk_n}^n\]
from which we infer that $\widehat\zeta^n_{\bk_{n}}$ and $\widehat\psi^n_{\bk_{n}}$ (and hence $\alpha_{\bk_n}^{n}$) are increasing and  (considering the differential inequalities for $\sqrt{\tanh(\sqrt\mu k_n) k_n} \widehat \psi_{\bk_n}^n\pm \widehat\zeta_{\bk_n}^n$) 
\[ \widehat \psi_{\bk_n}^n(t) \geq \widehat\psi^n_{\bk_{n}}(t_{0}) \cosh(\sqrt{\tanh(k_n)k_n} (t-t_{0})) \geq \frac{\widehat\psi^n_{\bk_{n}}(t_{0})}2 \exp(\sqrt{\tanh(k_n)k_n} (t-t_{0})).\]
Since $\alpha_{\bk_n}^{n}$ is increasing, continuity arguments show that the above holds on $ I^{n} \cap [t_{0},+\infty)$.

We now prove controls (b). We note that,  using the second assumption in~\eqref{hyp_bn} and since $\alpha^n_{\bk_n}(0)\leq  \frac14 b_n^2\alpha_0$,
\[ \epsilon^2 \alpha^{n}(0) J(\delta k_{0})^2 k_{0} =  \epsilon^2 ( 2 b_{n}^2\alpha_{0} +2\alpha_{\bk_n}^{n} (0)) J(\delta k_{0})^2 k_{0} < 1.
\]
By continuity, we have $a_{\bk_0}^n\in[0,1]$ on $[0,t]$ for $t$  close enough to $0$, and hence
\[ \frac{\dd}{\dd t}\widehat \zeta_{\bk_0}^n = \tanh(\sqrt\mu k_0)k_0 \widehat \psi^n_{\bk_0} \quad ; \quad 0\geq \frac{\dd}{\dd t}\widehat \psi^n_{\bk_0} \geq - \widehat \zeta^n_{\bk_0}.\]
Restricting to $t\leq t_0^{-1/2}$, we infer that $\widehat \zeta_{\bk_0}^n$  is increasing, $\widehat \psi_{\bk_0}^n$  (and hence $\alpha_{\bk_0}^{n}$) is decreasing, and 
\[ \widehat \psi_{\bk_0}^n(t) \leq \frac{b_n}2 ,\quad 0\leq \widehat \zeta_{\bk_0}^n(t) \leq t k_0\frac{b_n}2, \qquad  \widehat \psi_{\bk_0}^n(t) \geq \frac{b_n}2-\frac{b_n}{2} \frac{k_0t^2}2 \geq \frac{b_n}4.\]
In particular, $\alpha_{\bk_0}^{n} (t)\leq  b_{n}^{2}\alpha_0$ so that, using the assumption $\alpha^n_{\bk_n}(t)\leq  \frac14 b_n^2\alpha_0$ and~\eqref{hyp_bn}, we infer
\[\epsilon^2 \alpha^{n}(t) J(\delta k_{0})^2 k_{0} \leq \tfrac52 \epsilon^2 \alpha_0 b_n^2 J(\delta k_{0})^2 k_{0}<1.  \] 
Hence by continuity we find that the above holds on $[0,t_0]$. Notice now that since $\alpha^n_{\bk_{0}}(t)\geq \frac{1}{4} \alpha_0 b_n^2$, we have for any $t\in[0,t_0]$,
\[ \alpha^n(t) \geq 2 \alpha_{\bk_0}^n(t)  \geq \tfrac12 b_n^2\alpha_0  .\]
We can then use similar arguments than in the proof of control (a) to infer that $\widehat\zeta^n_{\bk_{n}}$ and $\widehat\psi^n_{\bk_{n}}$ are increasing and the desired lower bound on  $\widehat\psi^n_{\bk_{n}}(t) $.

Gathering controls (a) and (b), and arguing on the sign of $\alpha_{\bk_{n}}^{n}(0)-\frac14b_{n}^2 \alpha_{0}$, we find that under the assumption~\eqref{hyp_bn}, there holds for any $t \in I^{n} \cap [0,k_{0}^{-\frac12}]$,
\begin{equation}\label{eq.final-lower-bounds-psi}
\widehat\psi^n_{\bk_{n}}(t) \geq \frac{c_{n}}{8} \exp \big( \sqrt{\tanh(\sqrt{\mu} k_{n}) k_{n}} t \big) \text{  ,  } 
\end{equation}
and hence, using the identity $\frac{\dd}{\dd t}\widehat \zeta_{\bk_n}^n = \tanh(\sqrt\mu k_n)k_n \widehat \psi^n_{\bk_n}$, 
\begin{equation}\label{eq.final-lower-bounds-zeta}
\widehat\zeta^n_{\bk_{n}}(t) \geq \frac{c_{n}}{8} \sqrt{\tanh(\sqrt{\mu} k_{n}) k_{n}} \left( \exp \big( \sqrt{\tanh(\sqrt{\mu} k_{n}) k_{n}} t \big) - 1 \right).
\end{equation}

\noindent{\em Step 2: blowup.}\\
We set $b_n=(\frac18 \epsilon^2 \alpha_0  J( \delta k_n )^2   k_n)^{-1/2}$ so that Condition~\eqref{hyp_bn} holds for $n$ sufficiently large, and we define $c_n=8\exp(- k_n^{1/4})$. We also consider $n$ sufficiently large so that $\sqrt{\tanh(\sqrt{\mu} k_{n})} \geq \frac12$. Then~\eqref{eq.final-lower-bounds-psi}-\eqref{eq.final-lower-bounds-zeta} yields, for any $t \in I^{n} \cap [0,k_{0}^{-\frac12}]$,
\[
\widehat \psi_{\bk_n}^n(t)\geq \exp(\frac12 k_n^{1/2} t- k_n^{1/4}), \quad \widehat \zeta_{\bk_n}^n(t)\geq \frac12 k_n^{1/2} \exp(- k_n^{1/4}) \Big( \exp(\frac12 k_n^{1/2} t) -1\Big).
\]
Then, one has
\[  a_{\bk_n}^n(t)\leq 1-\epsilon^2 \alpha^n(t) J(\delta k_n)^2 k_n \leq 1-2\epsilon^2\alpha_{\bk_n}^n(t)J(\delta k_n)^2 k_n \leq 1- \epsilon^2 (2\pi)^d \tfrac18  J(\delta k_n)^2 k_n^3 |\widehat\psi_{\bk_n}^n(t)|^2.\]
Hence for $n$ sufficiently large, we have $2 k_n^{-1/4}<k_{0}^{-\frac12}$ and for any $t \in I^{n} \cap [2 k_n^{-1/4},k_{0}^{-\frac12}]$, there holds $\widehat \psi_{\bk_n}^n(t) \geq 1$,  $\widehat \zeta_{\bk_n}^n(t) \geq \frac14 k_n^{1/2}$ and $ a_{\bk_n}^n(t)\leq -\widehat \psi_{\bk_n}^n(t)^2 $, from which we infer
\[ \frac{\dd}{\dd t}\widehat \psi_{\bk_n}^n \geq  (\widehat\psi_{\bk_n}^n)^2  \widehat \zeta_{\bk_n}^n \geq \tfrac14 k_n^{1/2} (\widehat\psi_{\bk_n}^n)^2 .\]
This yields the desired blowup in time $T^n_\star\leq 2 k_n^{-1/4}+4 k_n^{-1/2}$. 

The blowup for negative time follows from time-reversibility: $(t,\zeta_0,\psi_0)\leftarrow (-t,\zeta_0,-\psi_0)$, and this completes the proof.
\end{proof}

\begin{Remark}
A direct inspection of the proof shows that our ill-posedness holds in fact for initial data in the Gevrey-$\sigma^{-1}$ class for any $\sigma\in[0,\frac12)$ and for any  $\sigma\in[0,1)$ if $\J^\delta=\Id$.
\end{Remark}

\begin{Remark}\label{R.toy-for-the-win}
Despite its simplistic nature, we believe that the toy model~\eqref{RWW2-toy} faithfully accounts for the instability mechanism at stake in the full model~\eqref{WW2-intro} and its regularized analogue,~\eqref{RWW2-intro}.

Specifically, we observe in the proof of \Cref{P.toy-WP1} (and also in \Cref{P.toy-IP}) a dichotomy between the treatment of stable and unstable modes, and the fact that a sufficiently regularizing operator $\J^\delta$ allows for a stability criterion (see~\eqref{eq.condition-1}) depending on both $\delta$ and $\epsilon$ parameters. 

It follows a local-in-time well-posedness theory with controls over solutions on a time interval that is
\begin{itemize}
	\item ``small'' in general, \ie with poor lower bounds when delta tends towards $0$;
	\item ``large"  (\ie with improved lower bounds) provided the initial data are sufficiently small;

\end{itemize} 
	see \Cref{P.toy-WP1,P.toy-WP2}. 

Note that the smallness condition on the initial data as well as the lower bounds  for the maximal existence time of the solutions are not uniform with respect to $\delta$, consistently with the model being ill-posed if $\delta=0$, which corresponds to $\J^\delta=\Id$.
Useful results can be obtained by imposing a lower bound on $\delta$ as a function of $\epsilon$, as shown in \Cref{R.toy-time}.

This strategy faithfully reflects the procedure that leads to one of our main results, \Cref{T.WP-delta}, on the well-posedness of~\eqref{RWW2-intro}. Indeed, its proof is based on (i) an unconditional well-posedness theorem for small times, combined with (ii) an existence theorem for solutions over large times under a condition of sufficient smallness on the initial data, the first result making stronger use of the regularization induced by the operator $\J^\delta$.

We also constructed the toy model~\eqref{RWW2-toy} so as to reproduce scales and orders accurately. Specifically, we see that the threshold on the order of $\J^\delta$ as a regularization operator that distinguishes between well-posedness and ill-posedness is $\r=-1/2$. In our result concerning~\eqref{WW2-intro}, we impose that $\J^\delta$ is regularizing of order $\r=-1$, but argue in \Cref{R.order-J} that $\r=-1/2$ would suffice for the large-time existence property, while $\r=-1$ is only used to simplify the small-time well-posedness. Finally, the time of existence and control of solutions to~\eqref{RWW2-toy} given in \Cref{R.toy-time} is roughly speaking of order $T\approx \delta/\epsilon^2$ under the assumption $\delta\gtrsim \epsilon$. In \Cref{T.WP-delta}, we obtain $T\approx \min(\delta/\epsilon^2,1/\epsilon)$ under the assumption $\delta\gtrsim \epsilon$, the additional contraint $T\lesssim 1/\epsilon$ being due to standard quadratic interactions that we neglected in the toy model.
\end{Remark}

\paragraph{Acknowledgments} This work benefited from several discussions with Jean-Claude Saut. Both the authors would like to take this opportunity to thank him more generally for his generosity, kindness and constant support. The first author thanks the program Centre Henri Lebesgue ANR-11-LABX-0020-0, for fostering an attractive mathematical environment. The second author has been partially funded by the ANR project CRISIS (ANR-20-CE40-0020-01) and would like to thank the University of Rennes for its hospitality and support.

\end{document}